\documentclass[12pt,reqno] {amsart}%
\usepackage{amsfonts}
\usepackage{amsmath}
\usepackage{amssymb}
\usepackage{graphicx}%
\usepackage{hyperref}
\hypersetup{
    colorlinks = true,
    linkcolor = blue,
    urlcolor = blue,
    citecolor = blue,
    anchorcolor = black,
}

\usepackage{cleveref}
\usepackage{comment}
 \usepackage{mathrsfs}
 \usepackage{float}
 \usepackage{tikz}
 \usepackage{bm}
 \usepackage{ulem}
\usepackage{slashed} 
 \usepackage{braket}
\usepackage{enumerate}
\newcommand{\Star}{\scalebox{1.4}{$\star$}}
 
\newcommand{\blue}[1]{{\color{blue}#1}}

\makeatletter
\def\l@section{\@tocline{1}{0pt}{1pc}{}{}}
\def\l@subsection{\@tocline{2}{0pt}{1pc}{4.6em}{}}
\def\l@subsubsection{\@tocline{3}{0pt}{1pc}{7.6em}{}}
\renewcommand{\tocsection}[3]{%
  \indentlabel{\@ifnotempty{#2}{\makebox[2.3em][l]{%
    \ignorespaces#1 #2.\hfill}}}#3}
\renewcommand{\tocsubsection}[3]{%
  \indentlabel{\@ifnotempty{#2}{\hspace*{2.3em}\makebox[2.3em][l]{%
    \ignorespaces#1 #2.\hfill}}}#3}
\renewcommand{\tocsubsubsection}[3]{%
  \indentlabel{\@ifnotempty{#2}{\hspace*{4.6em}\makebox[3em][l]{%
    \ignorespaces#1 #2.\hfill}}}#3}
\makeatother
\theoremstyle{plain}
\numberwithin{equation}{section}
\newtheorem{theorem}{Theorem}[section]
\newtheorem{corollary}[theorem]{Corollary}
\newtheorem{lemma}[theorem]{Lemma}
\newtheorem{proposition}[theorem]{Proposition}
\newtheorem{notation}[theorem]{Notation}

\newtheorem{definition}[theorem]{Definition}
\newtheorem{remark}[theorem]{Remark}
\newtheorem{main theorem}{Main Theorem}

\newcommand{\Z}{\mathbb{Z}}
\newcommand{\bs}[1]{\boldsymbol{#1}}
\newcommand{\mbff}{\scalebox{1.1}{$\boldsymbol{\mathcal{F}}$}}
\newcommand{\mff}{\mathfrak{F}}

\newcommand{\bN}{\mathbb{N}}

\newcommand{\bC}{\mathbb{C}}

\newcommand{\sP}{\mathscr{P}}

\newcommand{\be}{\begin{equation}}
\newcommand{\ee}{\end{equation}}

\renewcommand{\leq}{\leqslant}
\renewcommand{\geq}{\geqslant}

\begin{document}

\title{The  Quantum Perron-Frobenius Space}
\author{Linzhe Huang}
\address{Linzhe Huang, Yau Mathematical Sciences Center, Tsinghua University, Beijing, 100084, China}
\email{huanglinzhe@mail.tsinghua.edu.cn}

\author{Arthur Jaffe}
\address{Arthur Jaffe, 17 Oxford Street, Harvard University, Cambridge, MA 02138, USA}
\email{Arthur\_Jaffe@harvard.edu}

\author{Zhengwei Liu}
\address{Zhengwei Liu, Yau Mathematical Sciences Center and Department of Mathematics, Tsinghua University, Beijing, 100084, China, and Yanqi Lake Beijing Institute of Mathematical Sciences and Applications, Huairou District, Beijing, 101408, China}
\email{liuzhengwei@mail.tsinghua.edu.cn}

\author{Jinsong Wu}
\address{Jinsong Wu, Yanqi Lake Beijing Institute of Mathematical Sciences and Applications, Beijing, 101408, China}
\email{wjs@bimsa.cn}
\date{}

\maketitle

\begin{abstract}
We introduce $\mathfrak{F}$-positive elements in planar algebras. We establish the Perron-Frobenius theorem for $\mathfrak{F}$-positive elements.
We study the existence and uniqueness of the Perron-Frobenius eigenspace. When it is not one-dimensional, we characterize its  multiplicative structure.
Moreover, we consider the Perron-Frobenius eigenspace as the space of logical qubits for mixed states in quantum information. We describe the relation between these mathematical results and quantum error correction, especially to the Knill-Laflamme theorem. 
\end{abstract}

\tableofcontents
\newpage

\section{Introduction}
The Perron-Frobenius (PF) theorem~\cite{Per08, Fro0809,Fro12}, proved in the early 1900s, is a fundamental result for positive matrices, with applications in many branches of mathematics, including  probability theory~\cite{Sen06}, dynamical systems~\cite{CA22}, graph theory~\cite{BC09}, etc., as well as in science and engineering~\cite{PSC05,LM06}:  power control of  wireless networks, to commodity pricing models, and the ranking of web pages.  The PF theorem also leads to analyticity in some statistical mechanical systems, showing uniqueness for ground states and the absence of phase transitions~\cite{Ruelle}. 

Krein and Rutman generalized the PF theorem to positive, compact operators on ordered Banach spaces~\cite{KR48}. 
Glimm and Jaffe established a PF theorem for certain positive, non-compact operators and used this to prove uniqueness of  the ground state for quantum field theory Hamiltonians~\cite{GJ70}. 
Evans and H{\o}egh-Krohn established a PF theorem  for positive maps and studied multiplicative properties of the spectrum of positive (not necessarily completely positive) maps~\cite{EH78}. 
In quantum information,  a quantum channel is a completely positive, trace preserving linear map. The
PF theorem for quantum channels is called the quantum PF theorem.

In this paper, we interpret the  quantum PF theorem for completely positive maps in terms of  tensor networks, a diagrammatic language in quantum information; this  provides unity between the classic PF theorem and a quantum version. 
Furthermore, we establish the PF theorem on quantum symmetries in terms of Jones' subfactor planar algebras \cite{Jon83,Jon99,Jon12}.  
The tensor network case becomes a special case of spin-model planar algebras.

We call an element $\mathfrak{F}$-positive if its Fourier transform is a positive operator. 
This $\mathfrak{F}$-positivity is a generalization of the complete positivity of a quantum channel in the tensor network approach to quantum information theory. 
We elaborate the relation of this concept to 
the Perron-Frobenius theorem for planar algebras. 
Moreover, we  give several,  easy-to-check,  sufficient conditions for irreducibility of an $\mathfrak{F}$-positive element.  We obtain a  Collatz-Wielandt formula for planar algebras.
We provide a self-contained proof of the PF theorem for (not-necessarily compact) positive linear maps on an ordered Banach space.  

When the PF eigenvector is not unique, we study its PF eigenspace $\mathcal{E}$.
Inspired by quantum error correction, we consider an  $\mathfrak{F}$-positive element as a quantum channel, arising from the composition of an error channel, a recovery map, and an encoding map. 
We apply the PF eigenspace to encode logical qubits which are preserved by the quantum channel, up to a scalar. 
The size of logical qubits  depends on the multiplicative structure of density matrices, so it is insufficient to simply count the dimension of the PF eigenspace. This is our major motivation to study the multiplicative structure of the PF eigenspace, corresponding to the matrix product of density matrices. The multiplicative structure on the PF eigenspace that we study is different from the earlier approach of Choi and Effros \cite{ChoEff77}.

Density matrices of logical qubits are $*$-invariant.
We show that the maximal $*$-invariant subspace of the PF eigenspace is decomposed as the multiplication of a maximally-supported PF eigenvector $\zeta$ with a $C^*$-algebra $\mathcal{C}$. In quantum information,
we could regard $\zeta$ as a mixed state of ancillary qubits and a subalgebra $\mathcal{D}\cong M_{2^k}(\mathbb{C})$ as $k$ logical qubits. This PF theoretical approach to quantum error corrections suggests us to encode logical information by mixed states. 

The paper is organized as follows:
 In \S \ref{sec:pre} we recall the classic PF theorem and quantum PF theorem in the tensor network language. Moreover, we also recall basic notions of Jones planar algebras, leading to our study in \S\ref{sec:PF theorem} of PF theory in the planar algebra context. In \S \ref{sec:PFT for Banach space} we give a self-contained proof of the approximate existence and uniqueness results for PF eigenvectors of (noncompact) positive operators on infinite-dimensional, ordered Banach spaces. 
 In \S\ref{sec:structure} we characterize the multiplicative structures of the PF eigenpaces of $\mathfrak{F}$-positive elements. 
 
In \S\ref{sec:QEC} we relate our results to quantum information, and in particular to quantum error correction in a mixed state. 
In Theorem \ref{thm:PF to QEC},
we state the quantum error correction  for mixed states in the PF eigenspace. Moreover, we  contrast our approach to the widely known Knill-Laflamme theory.
In our approach, we choose the recovery map as the pseudo inverse of the error map similar to the ones in the Knill-Laflamme theory. The resulting quantum error correcting codes of mixed states has a close connection to those codes of pure states in  Knill-Laflamme theory.
It would be interesting to explore other types of recovery maps and to  study the PF eigenspaces to encode logical information. 

\section{Preliminary}\label{sec:pre}
In this section, we recall the classic PF theorem for positive matrices and interpret quantum PF theorem for completely positive maps in tensor networks, and then we show that the tensor network representations of matrices and completely positive maps could be unified  in a more general framework of planar algebras. We also introduce a 4-string graphical representation of positive matrices based on the Quon language in the end of the section.

\subsection{Tensor Networks and Classic PF Theorem} In quantum information, 
tensor networks make it convenient to compute vectors, matrices and high rank tensors. Let us recall some tensor network notations and operations. 
A complex vector $v\in\mathbb{C}^n$ is a one tensor with label $j$, which means the $j^{th}$ component can be represented pictorially as:
\begin{align}
    v_j=\raisebox{-0.2cm}{
\begin{tikzpicture}[rotate=90]
\draw (0,-1)--(0,0);
\draw [blue, fill=white] (-0.3,-0.3) rectangle (0.3,0.3);
\node at (0,0) {$v$};
\node at (0.25,-0.65) {$j$};
\end{tikzpicture}}\;.
\end{align}
A matrix $A\in M_n(\mathbb{C})$ is a two tensor with two labels $i$ and $j$, which means the $i^{th}$ row and $j^{th}$ column are represented pictorially as
\begin{align}
    A_{ij}=\raisebox{-0.2cm}{
\begin{tikzpicture}[rotate=90]
\draw (0,-1)--(0,1);
\draw [blue, fill=white] (-0.3,-0.3) rectangle (0.3,0.3);
\node at (0,0) {$A$};
\node at (0.25,0.65) {$j$};
\node at (0.25,-0.65) {$i$};
\end{tikzpicture}}\;.
\end{align}
A matrix $A\in M_n(\mathbb{C})$ acting on a vector $v$ is the vertical composition:
\begin{align}
   (Av)_{i}=\sum_j A_{ij}v_{j}= \raisebox{-0.2cm}{
\begin{tikzpicture}[rotate=90]
\draw (0,-2)--(0,0);
\draw [blue, fill=white] (-0.3,-0.3) rectangle (0.3,0.3);
\node at (0,0) {$v$};
\begin{scope}[shift={(0, -1)}]
\draw [blue, fill=white] (-0.3,-0.3) rectangle (0.3,0.3);
\node at (0,0) {$A$};
\node at (0.25,-0.7) {$i$};
\end{scope}
\end{tikzpicture}}.
\end{align}
Similarly, the multiplication for two matrices $A,B\in  M_n(\mathbb{C})$ is the picture:
\begin{align}
   (AB)_{ij}=\sum_k A_{ik}B_{kj}= \raisebox{-0.2cm}{
\begin{tikzpicture}[rotate=90]
\draw (0,-2)--(0,1);
\draw [blue, fill=white] (-0.3,-0.3) rectangle (0.3,0.3);
\node at (0,0) {$B$};
\begin{scope}[shift={(0, -1)}]
\draw [blue, fill=white] (-0.3,-0.3) rectangle (0.3,0.3);
\node at (0,0) {$A$};
\node at (0.25,1.7) {$j$};
\node at (0.25,-0.7) {$i$};
\end{scope}
\end{tikzpicture}}\;.
\end{align}
For the trace of $A$, one can connect the output to the input as,
\begin{align}
     {\rm Tr}(A)=\sum_i A_{ii}=\raisebox{-0.5cm}{
\begin{tikzpicture}[rotate=90]
\draw [blue, fill=white] (-0.3,-0.3) rectangle (0.3,0.3);
\node at (0,0) {$A$};
\draw (0, 0.3) .. controls +(0, 0.3) and +(0, 0.3) .. (0.6, 0.3);
\draw (0, -0.3) .. controls +(0, -0.3) and +(0, -0.3) .. (0.6, -0.3);
\draw(0.6,0.3)--(0.6,-0.3);
\end{tikzpicture}}\;.
\end{align}
We can easily see  the cyclic property, which proves  ${\rm Tr}(AB)={\rm Tr}(BA)$, namely 
\begin{align}
     {\rm Tr}(AB)=\raisebox{-0.45cm}{
\begin{tikzpicture}[rotate=90]
\draw(0.6,1.1)--(0.6,-0.3) (0,1.1)--(0,-0.3);
\draw [blue, fill=white] (-0.3,-0.3) rectangle (0.3,0.3);
\node at (0,0) {$A$};
\begin{scope}[shift={(0,0.8)}]
\draw [blue, fill=white] (-0.3,-0.3) rectangle (0.3,0.3);
\node at (0,0) {$B$};
\end{scope}
\draw (0, 1.1) .. controls +(0, 0.3) and +(0, 0.3) .. (0.6, 1.1);
\draw (0, -0.3) .. controls +(0, -0.3) and +(0, -0.3) .. (0.6, -0.3);
\end{tikzpicture}}=\raisebox{-0.6cm}{
\begin{tikzpicture}[rotate=90]
\draw [blue, fill=white] (-0.3,-0.3) rectangle (0.3,0.3);
\node at (0,0) {$A$};
\begin{scope}[shift={(0.8,0)}]
\draw [blue, fill=white] (-0.3,-0.3) rectangle (0.3,0.3);
\node at (0,0) {\rotatebox{180}{$B$}};
\end{scope}
\draw (0, 0.3) .. controls +(0, 0.3) and +(0, 0.3) .. (0.8, 0.3);
\draw (0, -0.3) .. controls +(0, -0.3) and +(0, -0.3) .. (0.8, -0.3);
\end{tikzpicture}}=\raisebox{-0.6cm}{
\begin{tikzpicture}[rotate=90]
\draw [blue, fill=white] (-0.3,-0.3) rectangle (0.3,0.3);
\node at (0,0) {$A$};
\begin{scope}[shift={(0.8,0)}]
\draw [blue, fill=white] (-0.3,-0.3) rectangle (0.3,0.3);
\node at (0,0) {$B^T$};
\end{scope}
\draw (0, 0.3) .. controls +(0, 0.3) and +(0, 0.3) .. (0.8, 0.3);
\draw (0, -0.3) .. controls +(0, -0.3) and +(0, -0.3) .. (0.8, -0.3);
\end{tikzpicture}}=
\raisebox{-0.45cm}{
\begin{tikzpicture}[rotate=90]
\draw(0.6,1.1)--(0.6,-0.3) (0,1.1)--(0,-0.3);
\draw [blue, fill=white] (-0.3,-0.3) rectangle (0.3,0.3);
\node at (0,0) {$B$};
\begin{scope}[shift={(0,0.8)}]
\draw [blue, fill=white] (-0.3,-0.3) rectangle (0.3,0.3);
\node at (0,0) {$A$};
\end{scope}
\draw (0, 1.1) .. controls +(0, 0.3) and +(0, 0.3) .. (0.6, 1.1);
\draw (0, -0.3) .. controls +(0, -0.3) and +(0, -0.3) .. (0.6, -0.3);
\end{tikzpicture}}={\rm Tr}(BA).
\end{align}
We refer the readers to e.g. \cite{BriChu17} for more details of tensor networks.

Recall that a matrix $A\in M_n(\mathbb{C})$ is called pointwise positive if $A_{ij}\geq0$; is called strictly pointwise positive if $A_{ij}>0$. A 
positive matrix $A$ is irreducible if $(I+A)^{n-1}$ is strictly pointwise positive. 
Now we restate the classic PF theorem in tensor networks.

\begin{theorem}[\bf The Classic PF Theorem~\cite{Per08,Fro0809,Fro12}]\label{thm:CPF}
Suppose $A\in M_n(\mathbb{C})$ is a pointwise-positive matrix, i.e., $A_{ij}\geq0$. Then there exists a nonzero-positive vector $v$, i.e., $v_j\geq0$, such that
$Av=r(A)v$, 
where $r(A)$ is the spectral radius of $A$. If $A$ is irreducible, then there exists a unique (up to a scalar) strictly positive vector $v$, i.e., $v_j>0$, satisfying
 Equation \eqref{eq:matrix PF vector} and  $r(A)$ is a simple eigenvalue.
\end{theorem}

The equation $Av=r(A)v$ can be 
depicted as
\begin{align}\label{eq:matrix PF vector}
    \raisebox{-.2cm}{
\begin{tikzpicture}[rotate=90]
\draw (0,-2)--(0,0);
\draw [blue, fill=white] (-0.3,-0.3) rectangle (0.3,0.3);
\node at (0,0) {$v$};
\begin{scope}[shift={(0, -1)}]
\draw [blue, fill=white] (-0.3,-0.3) rectangle (0.3,0.3);
\node at (0,0) {$A$};
\end{scope}
\end{tikzpicture}}=r(A)\raisebox{-0.2cm}{
\begin{tikzpicture}[rotate=90]
\draw (0,-1)--(0,0);
\draw [blue, fill=white] (-0.3,-0.3) rectangle (0.3,0.3);
\node at (0,0) {$v$};
\end{tikzpicture}}\;,
\end{align}
in tensor network with $2$-tensors. 
In \S\ref{Sect:PFCompletelyPositiveMaps} we interpret the quantum PF theorem in tensor network with $4$-tensors and the classical PF theorem can be included as an example.
Let $w$ be the normalized left PF eigenvector of $A$: 
\begin{align}
    \raisebox{-0.2cm}{
\begin{tikzpicture}[rotate=90]
\draw (0,-1)--(0,1);
\draw [blue, fill=white] (-0.3,-0.3) rectangle (0.3,0.3);
\node at (0,0) {$A$};
\begin{scope}[shift={(0, -1)}]
\draw [blue, fill=white] (-0.3,-0.3) rectangle (0.3,0.3);
\node at (0,0) {$w$};
\end{scope}
\end{tikzpicture}}=r(A)\raisebox{-0.2cm}{
\begin{tikzpicture}[rotate=90]
\draw (0,0)--(0,1);
\draw [blue, fill=white] (-0.3,-0.3) rectangle (0.3,0.3);
\node at (0,0) {$w$};
\end{tikzpicture}}\;.
\end{align}
Let $v$ be the normalized right PF eigenvector in Theorem \ref{thm:CPF}, define 
\begin{align}
    P_{ij}=\raisebox{-0.2cm}{
\begin{tikzpicture}[rotate=90]
\draw (0,-1)--(0,1);
\draw [blue, fill=white] (-0.3,-0.3) rectangle (0.3,0.3);
\node at (0,0) {$P$};
\node at (0.2,0.65) {$j$};
\node at (0.2,-0.65) {$i$};
\end{tikzpicture}}=\raisebox{-0.2cm}{
\begin{tikzpicture}[rotate=90]
\draw (0,0)--(0,1);
\draw [blue, fill=white] (-0.3,-0.3) rectangle (0.3,0.3);
\node at (0,0) {$w$};
\node at (0.2,0.65) {$j$};
\begin{scope}[shift={(0, -1)}]
    \draw (0,-1)--(0,0);
\draw [blue, fill=white] (-0.3,-0.3) rectangle (0.3,0.3);
\node at (0,0) {$v$};
\node at (0.2,-0.65) {$i$};
\end{scope}
\end{tikzpicture}}=v_iw_j.
\end{align}
Then $P$ is the PF projection, and we have
\begin{align}
     \raisebox{-0.2cm}{
\begin{tikzpicture}[rotate=90]
\draw (0,-2)--(0,1);
\draw [blue, fill=white] (-0.3,-0.3) rectangle (0.3,0.3);
\node at (0,0) {$P$};
\begin{scope}[shift={(0, -1)}]
\draw [blue, fill=white] (-0.3,-0.3) rectangle (0.3,0.3);
\node at (0,0) {$A$};
\end{scope}
\end{tikzpicture}}= \raisebox{-0.2cm}{
\begin{tikzpicture}[rotate=90]
\draw (0,-2)--(0,1);
\draw [blue, fill=white] (-0.3,-0.3) rectangle (0.3,0.3);
\node at (0,0) {$A$};
\begin{scope}[shift={(0, -1)}]
\draw [blue, fill=white] (-0.3,-0.3) rectangle (0.3,0.3);
\node at (0,0) {$P$};
\end{scope}
\end{tikzpicture}}=r(A) \raisebox{-0.2cm}{
\begin{tikzpicture}[rotate=90]
\draw (0,-1)--(0,1);
\draw [blue, fill=white] (-0.3,-0.3) rectangle (0.3,0.3);
\node at (0,0) {$P$};
\end{tikzpicture}}\;.
\end{align}

\subsection{The Quantum PF Theorem: Completely Positive Maps\label{Sect:PFCompletelyPositiveMaps}}
Let $\text{Hom}(M_n(\mathbb{C}),M_n(\mathbb{C}))$ be the set of all linear maps on $M_n(\mathbb{C})$.
 Let $\Phi\in\text{Hom}(M_n(\mathbb{C}),M_n(\mathbb{C}))$.
Graphically, we represent $\Phi$ as the following $4$-tensor:
\begin{align}
    \raisebox{-0.65cm}{\begin{tikzpicture}
\draw (-1,1.2)--(2.5,1.2);
\draw (-1,0.3)--(2.5,0.3);
\draw [blue,fill=white] (0,0) rectangle (1.5,1.5);
\node at (0.75,0.75) {$\Phi$};
\end{tikzpicture}},
\end{align}
the action of $\Phi$ on $D\in M_n(\mathbb{C})$ as
\begin{align}\label{eq:4 tensor action}
   \Phi(D)=\raisebox{-0.65cm}{ \begin{tikzpicture}
   \draw (-1,1.2)--(-1,0.3);
   \begin{scope}[shift={(-1,0.75)}]
     \draw [blue, fill=white] (-0.3,-0.3) rectangle (0.3,0.3);
\node at (0,0) {$D$};
   \end{scope}
\draw (-1,1.2)--(2.5,1.2);
\draw (-1,0.3)--(2.5,0.3);
\draw [blue,fill=white] (0,0) rectangle (1.5,1.5);
\node at (0.75,0.75) {$\Phi$};
\end{tikzpicture}}.
\end{align}
Up to a system of matrix basis $\{E_{ij}\}_{i,j=1}^n$, we have
\begin{align}
 \raisebox{-0.65cm}{\begin{tikzpicture}
\draw (-1,1.2)--(2.5,1.2);
\draw (-1,0.3)--(2.5,0.3);
\draw [blue,fill=white] (0,0) rectangle (1.5,1.5);
\node at (0.75,0.75) {$\Phi$};
\end{tikzpicture}}=\sum_{i,j=1}^n\raisebox{-1.6cm}{
\begin{tikzpicture}[scale=2]
\draw (0,-0.8)--(0,0.8);
\draw(0,0.8)--(-0.5,0.8);
\draw(0,-0.8)--(-0.5,-0.8);
\draw [blue, fill=white] (-0.3,-0.3) rectangle (0.3,0.3);
\node at (0,0) {$E_{ij}$};
\end{tikzpicture}}\raisebox{-1.6cm}{
\begin{tikzpicture}[scale=2]
\draw (0,-0.8)--(0,0.8);
\draw(0,0.8)--(0.5,0.8);
\draw(0,-0.8)--(0.5,-0.8);
\draw [blue, fill=white] (-0.3,-0.3) rectangle (0.3,0.3);
\node at (0,0) {$\Phi (E_{ij})$};
\end{tikzpicture}}.    
    \end{align}
 The Fourier transform $\mathcal{F}$: $\text{Hom}(M_n(\mathbb{C}),M_n(\mathbb{C}))\rightarrow M_n(\mathbb{C})\otimes M_n(\mathbb{C})$ is naturally represented as
 \begin{align}
     \raisebox{-0.6cm}{
\begin{tikzpicture}
\draw (-1,1.2)--(2.5,1.2);
\draw (-1,0.3)--(2.5,0.3);
\draw [blue,fill=white] (0,0) rectangle (1.5,1.5);
\node at (0.75,0.75) {$\Phi$};
\end{tikzpicture}}\quad\stackrel{\mathcal{F}}{\longmapsto}\quad
 \raisebox{-1.1cm}{
\begin{tikzpicture}
\draw (-0.5,1.2)--(2,1.2);
\draw (-0.5,0.3)--(2,0.3);
\draw(2,1.2)--(2,2);
\draw(2,0.3)--(2,-0.5);
\draw (-0.5,1.2)--(-0.5,2);
\draw (-0.5,0.3)--(-0.5,-0.5);
\draw [blue,fill=white] (0,0) rectangle (1.5,1.5);
\node at (0.75,0.75) {$\Phi$};
\end{tikzpicture}}.
 \end{align}
We denote the Fourier transform as
\begin{align}
\raisebox{-1.4cm}{
\begin{tikzpicture}
\draw (1,-1.5)--(1,-0.8);
\draw (1,1.5)--(1,0.8);
\draw (2,-1.5)--(2,-0.8);
\draw (2,1.5)--(2,0.8);
\draw [blue,fill=white] (0.7,-0.8) rectangle (2.3,0.8);
\node at (1.5,0) {$\widehat{\Phi}$};
\end{tikzpicture}}=
\raisebox{-1.1cm}{
\begin{tikzpicture}
\draw (-0.5,1.2)--(2,1.2);
\draw (-0.5,0.3)--(2,0.3);
\draw(2,1.2)--(2,2);
\draw(2,0.3)--(2,-0.5);
\draw (-0.5,1.2)--(-0.5,2);
\draw (-0.5,0.3)--(-0.5,-0.5);
\draw [blue,fill=white] (0,0) rectangle (1.5,1.5);
\node at (0.75,0.75) {$\Phi$};
\end{tikzpicture}}=\sum_{i,j=1}^n\raisebox{-1.8cm}{
\begin{tikzpicture}[scale=2]
\draw (0,-1)--(0,1);
\draw [blue, fill=white] (-0.3,-0.3) rectangle (0.3,0.3);
\node at (0,0) {$E_{ij}$};
\end{tikzpicture}}\raisebox{-1.8cm}{
\begin{tikzpicture}[scale=2]
\draw (0,-1)--(0,1);
\draw [blue, fill=white] (-0.3,-0.3) rectangle (0.3,0.3);
\node at (0,0) {$\Phi (E_{ij})$};
\end{tikzpicture}}\;.
\end{align}
  We reinterpret $\Phi(D)$ as a convolution $D* \widehat{\Phi}$ in the Fourier dual:
 \begin{align}
     \raisebox{-1.4cm}{
\begin{tikzpicture}
\draw (0,-1)--(0,1) arc (180:0:.5) --++(0,-.3) arc (-180:0:.5) --++(0,.8) ;
\draw (0,-1) arc (-180:0:.5) --++(0,.3) arc (180:0:.5) --++(0,-.8);
\draw [blue, fill=white] (-0.3,-0.3) rectangle (0.3,0.3);
\node at (0,0) {$D$};
\draw [blue,fill=white] (0.7,-0.8) rectangle (2.3,0.8);
\node at (1.5,0) {$\widehat{\Phi}$};
\end{tikzpicture}}\;.
 \end{align}
Using the generalized Choi-Kraus representation, see e.g. \cite[Corollary 3.12]{Pis01}, there exists a finite set of matrices $(V_i)_{i=1}^m$ and $(U_i)_{i=1}^m$, $1\leq m\leq n^2$, such that 
\begin{align}
\Phi(D)=\sum_{i=1}^m V_i D U_i=\sum_{i=1}^{m}
\raisebox{-1.65cm}{
\begin{tikzpicture}
\draw (0,-1.8)--(0,1.8) ;
\draw (0,-1.3) 
(0,-.8);
\draw [blue, fill=white] (-0.3,-0.3) rectangle (0.3,0.3);
\node at (0,0) {$D$};
\draw [blue, fill=white] (-0.3,-1.3) rectangle (.3,-0.7);
\node at (0,-1) {$V_i$};
\draw [blue, fill=white] (-0.3,0.7) rectangle (.3,1.3);
\node at (0,1) {$U_i$};
\end{tikzpicture}}=\sum_{i=1}^m 
\raisebox{-1.4cm}{
\begin{tikzpicture}
\draw (0,-1)--(0,1) arc (180:0:.5) --++(0,-.3) arc (-180:0:.5) --++(0,.8) ;
\draw (0,-1) arc (-180:0:.5) --++(0,.3) arc (180:0:.5) --++(0,-.8);
\draw [blue, fill=white] (-0.3,-0.3) rectangle (0.3,0.3);
\node at (0,0) {$D$};
\draw [blue, fill=white] (1.7,-1.3) rectangle (2.3,-0.7);
\node at (2,-1) {$V_i$};
\draw [blue, fill=white] (1.7,0.7) rectangle (2.3,1.3);
\node at (2,1) {$U_i$};
\end{tikzpicture}}\;.
\end{align}
Note  that
 \begin{align}
     \raisebox{-1.4cm}{
\begin{tikzpicture}
\draw (1,-1.5)--(1,-0.8);
\draw (1,1.5)--(1,0.8);
\draw (2,-1.5)--(2,-0.8);
\draw (2,1.5)--(2,0.8);
\draw [blue,fill=white] (0.7,-0.8) rectangle (2.3,0.8);
\node at (1.5,0) {$\widehat{\Phi}$};
\end{tikzpicture}}\in\text{span}\left\{\raisebox{-1.6cm}{
\begin{tikzpicture}
\draw (1, 1.7)--(1,1)--++(0,-.3) arc (-180:0:.5) --++(0,1) ;
\draw (1,-1.7) --(1, -1) --++(0,.3) arc (180:0:.5) --++(0,-1);
\draw [blue, fill=white] (1.7,-1.3) rectangle (2.3,-0.7);
\node at (2,-1) {$V$};
\draw [blue, fill=white] (1.7,0.7) rectangle (2.3,1.3);
\node at (2,1) {$U$};
\end{tikzpicture}}:\ U,V\in M_n(\mathbb{C}) \right\}=\text{span}\left\{\raisebox{-1.65cm}{
\begin{tikzpicture}
\draw (0,-1.8)--(0,1.8) ;
\draw (0,-1.3) 
(0,-.8);
\draw [blue, fill=white] (-0.3,-0.3) rectangle (0.3,0.3);
\node at (0,0) {$U$};
\begin{scope}[shift={(1, 0)}]
  \draw (0,-1.8)--(0,1.8) ;
\draw (0,-1.3) 
(0,-.8);
\draw [blue, fill=white] (-0.3,-0.3) rectangle (0.3,0.3);
\node at (0,0) {$V$};
\end{scope}
\end{tikzpicture}}:\ U,V\in M_n(\mathbb{C}) \right\}.
 \end{align}
 
A quantum channel $\Phi$ on $M_n(\mathbb{C})$ is a completely positive, trace preserving (CPTP) linear map. (Sometimes one allows trace non-increasing maps.) 
Choi's Theorem 1 of ~\cite{Choi75} assures that  there exist $m$ operators $F_j\in M_n(\mathbb{C})$ (called Kraus operators in quantum information), $1\leq m\leq n^2$, such that
\begin{align}
    \Phi(D)=\sum_{j=1}^m F_j D F_j^*, ~\quad\forall D \in M_n(\mathbb{C})\;.
\end{align}
One can represent $\Phi$ pictorially  as a tensor network, namely  
\begin{align}
\Phi(D)=\sum_{j=1}^{m}
\raisebox{-1.65cm}{
\begin{tikzpicture}
\draw (0,-1.8)--(0,1.8) ;
\draw (0,-1.3) 
(0,-.8);
\draw [blue, fill=white] (-0.3,-0.3) rectangle (0.3,0.3);
\node at (0,0) {$D$};
\draw [blue, fill=white] (-0.3,-1.3) rectangle (.3,-0.7);
\node at (0,-1) {$F_j$};
\draw [blue, fill=white] (-0.3,0.7) rectangle (.3,1.3);
\node at (0,1) {$F_j^*$};
\end{tikzpicture}}=\sum_{j=1}^m
\raisebox{-1.4cm}{
\begin{tikzpicture}
\draw (0,-1)--(0,1) arc (180:0:.5) --++(0,-.3) arc (-180:0:.5) --++(0,.8) ;
\draw (0,-1) arc (-180:0:.5) --++(0,.3) arc (180:0:.5) --++(0,-.8);
\draw [blue, fill=white] (-0.3,-0.3) rectangle (0.3,0.3);
\node at (0,0) {$D$};
\draw [blue, fill=white] (1.7,-1.3) rectangle (2.3,-0.7);
\node at (2,-1) {$F_j$};
\draw [blue, fill=white] (1.7,0.7) rectangle (2.3,1.3);
\node at (2,1) {$F_j^*$};
\end{tikzpicture}}
\;.
\end{align}
The Fourier transform $\widehat{\Phi}$ then  is  the following 4-tensor
\begin{align}\label{eq:cptp Fourier}
\raisebox{-1.4cm}{
\begin{tikzpicture}
\draw (1,-1.5)--(1,-0.8);
\draw (1,1.5)--(1,0.8);
\draw (2,-1.5)--(2,-0.8);
\draw (2,1.5)--(2,0.8);
\draw [blue,fill=white] (0.7,-0.8) rectangle (2.3,0.8);
\node at (1.5,0) {$\widehat{\Phi}$};
\end{tikzpicture}}=\sum_{i=1}^m
\raisebox{-1.6cm}{
\begin{tikzpicture}
\draw (1, 1.7)--(1,1)--++(0,-.3) arc (-180:0:.5) --++(0,1) ;
\draw (1,-1.7) --(1, -1) --++(0,.3) arc (180:0:.5) --++(0,-1);
\draw [blue, fill=white] (1.7,-1.3) rectangle (2.3,-0.7);
\node at (2,-1) {$F_j$};
\draw [blue, fill=white] (1.7,0.7) rectangle (2.3,1.3);
\node at (2,1) {$F_j^*$};
\end{tikzpicture}},
\end{align}
which is a positive operator in $M_n(\mathbb{C})\otimes M_n(\mathbb{C})$.
 This gives a graphical interpretation  of Choi's Theorem 2 of ~\cite{Choi75}: A linear map $\Phi\in\text{Hom}(M_n(\mathbb{C}),M_n(\mathbb{C}))$ is completely positive if and only if its Fourier transform $\widehat{\Phi}$ is positive-semidefinite. 
\begin{remark}
Let $\Phi$ be a completely positive map on $M_n(\mathbb{C})$.
By the spectral decomposition of  $\widehat{\Phi}$, we could choose Kraus operators $\{F_j\}_{i=1}^m$ in $M_n(\mathbb{C})$ such that ${\rm Tr}(F_i^*F_j)=0$, $\forall i\neq j$, 
where $m$ is the rank of $\widehat{\Phi}$. 
\end{remark}

A quantum channel $\Phi$ is called irreducible if $\lambda p-\Phi(p)$ is positive-semidefinite
for some projection $p\in M_n(\mathbb{C})$ and $\lambda>0$ implies $p=0$ or $I$. This definition has several equivalent descriptions that are discussed in Theorem \ref{thm:equivirr}.
\begin{theorem}[\bf The PF Theorem for Quantum Channel~\cite{EH78}]\label{thm:PFT for quantum channel}
Let $\Phi$ be a quantum channel on $M_n(\mathbb{C})$. Then 
 there exists a nonzero positive-semidefinite matrix $D$ such that
\begin{align}\label{eq:cptp PF vector}
    \raisebox{-0.9cm}{
\begin{tikzpicture}[scale=1.5]
\begin{scope}[shift={(-0.6, 0)}]
\draw [blue] (0,0) rectangle (0.35, 0.5);
\node at (0.175, 0.25) {$D$};
\end{scope}
\draw [blue] (0,0) rectangle (0.7, 0.5);
\node at (0.35, 0.25) {$\widehat{\Phi}$};
\draw  (0.5, 0)--(0.5, -0.4) (0.5, 0.5)--(0.5, 0.9);
\draw (0.2, 0.5) .. controls +(0, 0.4) and +(0, 0.4) .. (-0.425, 0.5);
\draw (0.2, 0) .. controls +(0, -0.4) and +(0, -0.4) .. (-0.425, 0);
\end{tikzpicture}}
=r
\raisebox{-0.9cm}{
\begin{tikzpicture}[scale=1.5]
\begin{scope}[shift={(1.5, 0)}]
\draw [blue] (0,0) rectangle (0.35, 0.5);
\node at (0.175, 0.25) {$D$};
\draw  (0.175, 0)--(0.175, -0.4) (0.175, 0.5)--(0.175, 0.9);
\end{scope}
\end{tikzpicture}},
\end{align}
where $r$ is the spectral radius of $\Phi$.
 Moreover, if $\Phi$ is irreducible, then $D$ is positive-definite and unique up to a scalar.
\end{theorem}
\begin{remark} 
Let $\Phi$ be a quantum channel on $M_n(\mathbb{C})$.
Considering the right action of $\widehat{\Phi}$, we have that there exists a nonzero positive-semidefinite matrix $C$ such that
\begin{align}
    \raisebox{-0.9cm}{
\begin{tikzpicture}[scale=1.5]
\begin{scope}[shift={(0.95, 0)}]
\draw [blue] (0,0) rectangle (0.35, 0.5);
\node at (0.175, 0.25) {$C$};
\end{scope}
\draw [blue] (0,0) rectangle (0.7, 0.5);
\node at (0.35, 0.25) {$\widehat{\Phi}$};
\draw  (0.2, 0)--(0.2, -0.4) (0.2, 0.5)--(0.2, 0.9);
\draw (0.5, 0.5) .. controls +(0, 0.4) and +(0, 0.4) .. (1.125, 0.5);
\draw (0.5, 0) .. controls +(0, -0.4) and +(0, -0.4) .. (1.125, 0);
\end{tikzpicture}}=r
\raisebox{-0.9cm}{
\begin{tikzpicture}[scale=1.5]
\begin{scope}[shift={(1.5, 0)}]
\draw [blue] (0,0) rectangle (0.35, 0.5);
\node at (0.175, 0.25) {$C$};
\draw  (0.175, 0)--(0.175, -0.4) (0.175, 0.5)--(0.175, 0.9);
\end{scope}
\end{tikzpicture}}\;,
\end{align}
where $r$ is the spectral radius of $\Phi$.
We define the 4-tensor $\Phi_{C,D}$ as
\begin{align}
     \raisebox{-0.65cm}{\begin{tikzpicture}
\draw (-1,1.2)--(2.5,1.2);
\draw (-1,0.3)--(2.5,0.3);
\draw [blue,fill=white] (0,0) rectangle (1.5,1.5);
\node at (0.75,0.75) {$\Phi_{C,D}$};
\end{tikzpicture}}= \raisebox{-0.9cm}{
\begin{tikzpicture}[scale=1.5]
\begin{scope}[shift={(1.5, 0)}]
\draw [blue] (0,0) rectangle (0.35, 0.5);
\node at (0.175, 0.25) {$C$};
\draw  (0.175, 0)--(0.175, -0.4) (0.175, 0.5)--(0.175, 0.9);
\draw(0.175,-0.4)--(-0.3,-0.4) (0.175, 0.9)--(-0.3, 0.9);
\end{scope}
\end{tikzpicture}}\raisebox{-0.9cm}{
\begin{tikzpicture}[scale=1.5]
\begin{scope}[shift={(1.5, 0)}]
\draw [blue] (0,0) rectangle (0.35, 0.5);
\node at (0.175, 0.25) {$D$};
\draw  (0.175, 0)--(0.175, -0.4) (0.175, 0.5)--(0.175, 0.9);
\draw (0.175, -0.4)--(0.475, -0.4) (0.175, 0.9)--(0.475, 0.9);
\end{scope}
\end{tikzpicture}}\;.
\end{align}
The Fourier transform $\widehat{\Phi_{C,D}}$ is
\begin{align}
    \raisebox{-1.4cm}{
\begin{tikzpicture}
\draw (1,-1.5)--(1,-0.8);
\draw (1,1.5)--(1,0.8);
\draw (2,-1.5)--(2,-0.8);
\draw (2,1.5)--(2,0.8);
\draw [blue,fill=white] (0.7,-0.8) rectangle (2.3,0.8);
\node at (1.5,0) {$\widehat{\Phi_{C,D}}$};
\end{tikzpicture}}= \raisebox{-0.9cm}{
\begin{tikzpicture}[scale=1.5]
\begin{scope}[shift={(1.5, 0)}]
\draw [blue] (0,0) rectangle (0.35, 0.5);
\node at (0.175, 0.25) {$C$};
\draw  (0.175, 0)--(0.175, -0.4) (0.175, 0.5)--(0.175, 0.9);
\end{scope}
\end{tikzpicture}}\raisebox{-0.9cm}{
\begin{tikzpicture}[scale=1.5]
\begin{scope}[shift={(1.5, 0)}]
\draw [blue] (0,0) rectangle (0.35, 0.5);
\node at (0.175, 0.25) {$D$};
\draw  (0.175, 0)--(0.175, -0.4) (0.175, 0.5)--(0.175, 0.9);
\end{scope}
\end{tikzpicture}}\;,
\end{align}
which is positive-semidefinite.
 Moreover, we have
\begin{align}
   \raisebox{-1.2cm}{
\begin{tikzpicture}[scale=2]
\path [fill=white] (0.35, 0.5) .. controls +(0, 0.3) and +(0, 0.3) .. (0.85, 0.5)-- (0.85, 0) .. controls +(0, -0.3) and +(0, -0.3) .. (0.35, 0)--(0.35, 0.5);
\draw [blue, fill=white] (0,0) rectangle (0.5, 0.5);
\node at (0.25, 0.25) {$\widehat{\Phi_{C,D}}$};
\draw (0.15, 0.5)--(0.15, 0.9) (0.15, 0)--(0.15, -0.4);
\begin{scope}[shift={(0.7, 0)}]
\draw [blue, fill=white] (0,0) rectangle (0.5, 0.5);
\node at (0.25, 0.25) {$\widehat{\Phi}$};
\draw (0.35, 0.5)--(0.35, 0.9)  (0.35, 0)--(0.35, -0.4);
\end{scope}
\draw (0.35, 0.5) .. controls +(0, 0.3) and +(0, 0.3) .. (0.85, 0.5);
\draw (0.35, 0) .. controls +(0, -0.3) and +(0, -0.3) .. (0.85, 0);
\end{tikzpicture}}=\raisebox{-1.2cm}{
\begin{tikzpicture}[scale=2]
\path [fill=white] (0.35, 0.5) .. controls +(0, 0.3) and +(0, 0.3) .. (0.85, 0.5)-- (0.85, 0) .. controls +(0, -0.3) and +(0, -0.3) .. (0.35, 0)--(0.35, 0.5);
\draw [blue, fill=white] (0,0) rectangle (0.5, 0.5);
\node at (0.25, 0.25) {$\widehat{\Phi}$};
\draw (0.15, 0.5)--(0.15, 0.9) (0.15, 0)--(0.15, -0.4);
\begin{scope}[shift={(0.7, 0)}]
\draw [blue, fill=white] (0,0) rectangle (0.5, 0.5);
\node at (0.25, 0.25) {$\widehat{\Phi_{C,D}}$};
\draw (0.35, 0.5)--(0.35, 0.9)  (0.35, 0)--(0.35, -0.4);
\end{scope}
\draw (0.35, 0.5) .. controls +(0, 0.3) and +(0, 0.3) .. (0.85, 0.5);
\draw (0.35, 0) .. controls +(0, -0.3) and +(0, -0.3) .. (0.85, 0);
\end{tikzpicture}}=r
\raisebox{-1.2cm}{
\begin{tikzpicture}[scale=2]
\draw [blue, fill=white] (0,0) rectangle (0.5, 0.5);
\node at (0.25, 0.25) {$\widehat{\Phi_{C,D}}$};
\draw (0.15, 0.5)--(0.15, 0.9) (0.15, 0)--(0.15, -0.4);
\draw (0.35, 0.5)--(0.35, 0.9)  (0.35, 0)--(0.35, -0.4);
\end{tikzpicture}}.
\end{align}
\end{remark}
We now show  Theorem \ref{thm:PFT for quantum channel} unifies the classic PF theorem.
An $n$-dimensional vector $v\in\mathbb{C}^n$ could been regarded as a diagonal matrix $D$ in $M_n(\mathbb{C})$: its graphical representation has the same label $i$,
\begin{align}
    v_i=D_{ii}=\raisebox{-0.9cm}{
\begin{tikzpicture}
\draw (0,-1)--(0,1);
\draw [blue, fill=white] (-0.3,-0.3) rectangle (0.3,0.3);
\node at (0,0) {$D$};
\node at (0.2,0.65) {$i$};
\node at (0.2,-0.65) {$i$};
\end{tikzpicture}}.
\end{align}
A matrix $A$ acting on the vector space then could be represented as a 4-tensor $\widehat{\Phi_A}$ acting on the diagonal matrix subalgebra of $M_n(\mathbb{C})$. Specifically, 
\begin{align}
    Av=\sum_{i,j=1}^{n}A_{ij}
\raisebox{-1.65cm}{
\begin{tikzpicture}
\draw (0,-1.8)--(0,1.8) ;
\draw (0,-1.3) 
(0,-.8);
\draw [blue, fill=white] (-0.3,-0.3) rectangle (0.3,0.3);
\node at (0,0) {$D$};
\draw [blue, fill=white] (-0.3,-1.3) rectangle (.3,-0.7);
\node at (0,-1) {$E_{ij}$};
\draw [blue, fill=white] (-0.3,0.7) rectangle (.3,1.3);
\node at (0,1) {$E_{ji}$};
\end{tikzpicture}}=\sum_{i,j=1}^n A_{ij}
\raisebox{-1.4cm}{
\begin{tikzpicture}
\draw (0,-1)--(0,1) arc (180:0:.5) --++(0,-.3) arc (-180:0:.5) --++(0,.8) ;
\draw (0,-1) arc (-180:0:.5) --++(0,.3) arc (180:0:.5) --++(0,-.8);
\draw [blue, fill=white] (-0.3,-0.3) rectangle (0.3,0.3);
\node at (0,0) {$D$};
\draw [blue, fill=white] (1.7,-1.3) rectangle (2.3,-0.7);
\node at (2,-1) {$E_{ij}$};
\draw [blue, fill=white] (1.7,0.7) rectangle (2.3,1.3);
\node at (2,1) {$E_{ji}$};
\end{tikzpicture}}.
\end{align}
We define $\widehat{\Phi_A}$ as the 4-tensor
\begin{align}\label{eq:positive matrix fourier}
\raisebox{-1.4cm}{
\begin{tikzpicture}
\draw (1,-1.5)--(1,-0.8);
\draw (1,1.5)--(1,0.8);
\draw (2,-1.5)--(2,-0.8);
\draw (2,1.5)--(2,0.8);
\draw [blue,fill=white] (0.7,-0.8) rectangle (2.3,0.8);
\node at (1.5,0) {$\widehat{\Phi_A}$};
\end{tikzpicture}}=\sum_{i,j=1}^n A_{ij}
\raisebox{-1.6cm}{
\begin{tikzpicture}
\draw (1, 1.7)--(1,1)--++(0,-.3) arc (-180:0:.5) --++(0,1) ;
\draw (1,-1.7) --(1, -1) --++(0,.3) arc (180:0:.5) --++(0,-1);
\draw [blue, fill=white] (1.7,-1.3) rectangle (2.3,-0.7);
\node at (2,-1) {$E_{ij}$};
\draw [blue, fill=white] (1.7,0.7) rectangle (2.3,1.3);
\node at (2,1) {$E_{ji}$};
\end{tikzpicture}}.
\end{align}
Note that $\widehat{\Phi_A}$ is a linear sum of $n^2$ orthogonal rank-one projections. So $\widehat{\Phi_A}$ is positive-semidefinite if and only if $A$ is a positive matrix, i.e., $A_{ij}\geq0$. Moreover, if $A$ is irreducible then $\Phi_A$ is also irreducible.
Now  take $\Phi=\Phi_A$ in Theorem \ref{thm:PFT for quantum channel}, then we recover 
the classic PF theorem.
\subsection{$C^*$-Planar Algebras}\label{subsec:planar algebra}
In this section, we introduce planar algebras \cite{Jon12,Jon21}, which contain the above graphical computations in  tensor networks.
Moreover, planar algebras is a more general picture language to study quantum symmtries such as quantum groups and  subfactors.
We focus on the $C^*$-planar algebra $\mathscr{P}=\{\mathscr{P}_{n,\pm}\}_{n\geq0}$.
This algebra is graded, with each $\mathscr{P}_{n,\pm}$ being a finite-dimensional $C^*$-algebra consisting of shaded $n$-boxes with a trace $tr_{n,\pm}$. 
 Let us introduce some  planar tangles such as multiplication, convolution, Fourier transform, etc. We are interested in two box space.
An element $x\in\mathscr{P}_{2,+}$ is depicted as the following picture:
\begin{align}\label{eq:2box}
    \raisebox{-0.9cm}{
\begin{tikzpicture}
\path [fill=gray!40] (-0.1, -1) rectangle (0.1, 1);
\draw (-0.1,-1)--(-0.1,1);
\draw (0.1,-1)--(0.1,1);
\draw [blue, fill=white] (-0.3,-0.3) rectangle (0.3,0.3);
\node at (0,0) {$x$};
\node at (-0.5,0){\$};
\end{tikzpicture}}\;,
\end{align}
which is a two box with shading and the dollar sign \$ on the left side of the rectangle (we may omit them when considering an element in $\mathscr{P}_{2,\pm}$ if there is no confusion).
Multiplication of two elements $x,y\in \mathscr{P}_{2,+}$ is given by their vertical composition:
\begin{align}\label{eq:multiplication}
   xy:= \raisebox{-1.4cm}{
\begin{tikzpicture}
\path [fill=gray!40] (-0.1, -2) rectangle (0.1, 1);
\draw (-0.1,-2)--(-0.1,1);
\draw (0.1,-2)--(0.1,1);
\draw [blue, fill=white] (-0.3,-0.3) rectangle (0.3,0.3);
\node at (0,0) {$y$};
\begin{scope}[shift={(0, -1)}]
\draw [blue, fill=white] (-0.3,-0.3) rectangle (0.3,0.3);
\node at (0,0) {$x$};
\end{scope}
\end{tikzpicture}}
\;.
\end{align}
One obtains the unnormalized trace of $x$, by connecting the inputs with the outputs:
\begin{align}\label{eq:trace}
      tr_{2,+}(x):=tr_{0,+}\left(\raisebox{-0.7cm}{
\begin{tikzpicture}
\path [fill=gray!40](-0.1, 0.3) .. controls +(0, 0.5) and +(0, 0.5) .. (0.9, 0.3)--(0.9,-0.3)
..controls +(0,- 0.5) and +(0, -0.5)..(-0.1, -0.3);
\path [fill=white](0.1, 0.3) .. controls +(0, 0.3) and +(0, 0.3) .. (0.7, 0.3)--(0.7,-0.3)
..controls +(0,- 0.3) and +(0, -0.3)..(0.1, -0.3);
\draw (0.1, 0.3) .. controls +(0, 0.3) and +(0, 0.3) .. (0.7, 0.3);
\draw (-0.1, 0.3) .. controls +(0, 0.5) and +(0, 0.5) .. (0.9, 0.3);
\draw (0.1, -0.3) .. controls +(0, -0.3) and +(0, -0.3) .. (0.7, -0.3);
\draw (-0.1, -0.3) .. controls +(0,- 0.5) and +(0, -0.5) .. (0.9, -0.3);
\draw(0.7,0.3)--(0.7,-0.3);
\draw(0.9,0.3)--(0.9,-0.3);
\draw [blue, fill=white] (-0.3,-0.3) rectangle (0.3,0.3);
\node at (0,0) {$x$};
\end{tikzpicture}}\right).
\end{align}
If a $C^*$-planar algebra is a subfactor planar algebra, then $\dim_{\mathbb{C}}(\mathscr{P}_{0,\pm})=1$, and the trace is just \raisebox{-0.7cm}{
\begin{tikzpicture}
\path [fill=gray!40](-0.1, 0.3) .. controls +(0, 0.5) and +(0, 0.5) .. (0.9, 0.3)--(0.9,-0.3)
..controls +(0,- 0.5) and +(0, -0.5)..(-0.1, -0.3);
\path [fill=white](0.1, 0.3) .. controls +(0, 0.3) and +(0, 0.3) .. (0.7, 0.3)--(0.7,-0.3)
..controls +(0,- 0.3) and +(0, -0.3)..(0.1, -0.3);
\draw (0.1, 0.3) .. controls +(0, 0.3) and +(0, 0.3) .. (0.7, 0.3);
\draw (-0.1, 0.3) .. controls +(0, 0.5) and +(0, 0.5) .. (0.9, 0.3);
\draw (0.1, -0.3) .. controls +(0, -0.3) and +(0, -0.3) .. (0.7, -0.3);
\draw (-0.1, -0.3) .. controls +(0,- 0.5) and +(0, -0.5) .. (0.9, -0.3);
\draw(0.7,0.3)--(0.7,-0.3);
\draw(0.9,0.3)--(0.9,-0.3);
\draw [blue, fill=white] (-0.3,-0.3) rectangle (0.3,0.3);
\node at (0,0) {$x$};
\end{tikzpicture}}\;.
It has the following property:
\begin{align}
     tr_{2,+}(xy)=\raisebox{-1.1cm}{
\begin{tikzpicture}
\path [fill=gray!40] (-0.1, 1.1) .. controls +(0, 0.5) and +(0, 0.5) .. (0.9,1.1)--(0.9,-0.3)..
 controls +(0, -0.5) and +(0, -0.5) .. (-0.1, -0.3);
 \path [fill=white](0.1, 1.1) .. controls +(0, 0.3) and +(0, 0.3) .. (0.7,1.1)--(0.7,-0.3).. controls +(0, -0.3) and +(0, -0.3)..(0.1, -0.3);
\draw(0.7,1.1)--(0.7,-0.3) (0.9,1.1)--(0.9,-0.3) (0.1,1.1)--(0.1,-0.3)(-0.1,1.1)--(-0.1,-0.3);
\draw [blue, fill=white] (-0.3,-0.3) rectangle (0.3,0.3);
\node at (0,0) {$x$};
\begin{scope}[shift={(0,0.8)}]
\draw [blue, fill=white] (-0.3,-0.3) rectangle (0.3,0.3);
\node at (0,0) {$y$};
\end{scope}
\draw (0.1, 1.1) .. controls +(0, 0.3) and +(0, 0.3) .. (0.7, 1.1);
\draw (-0.1, 1.1) .. controls +(0, 0.5) and +(0, 0.5) .. (0.9, 1.1);
\draw (0.1, -0.3) .. controls +(0, -0.3) and +(0, -0.3) .. (0.7, -0.3);
\draw (-0.1, -0.3) .. controls +(0, -0.5) and +(0, -0.5) .. (0.9, -0.3);
\end{tikzpicture}}=\raisebox{-0.7cm}{
\begin{tikzpicture}
\path [fill=gray!40] (-0.1, 0.3) .. controls +(0, 0.5) and +(0, 0.5) .. (0.9, 0.3);
\path [fill=white](0.1, 0.3) .. controls +(0, 0.3) and +(0, 0.3) .. (0.7, 0.3);
\path [fill=gray!40] (-0.1, -0.3) .. controls +(0,- 0.5) and +(0,- 0.5) .. (0.9,- 0.3);
\path [fill=white](0.1,- 0.3) .. controls +(0, -0.3) and +(0, -0.3) .. (0.7, -0.3);
\draw [blue, fill=white] (-0.3,-0.3) rectangle (0.3,0.3);
\node at (0,0) {$x$};
\begin{scope}[shift={(0.8,0)}]
\draw [blue, fill=white] (-0.3,-0.3) rectangle (0.3,0.3);
\node at (0,0) {\rotatebox{180}{$y$}};
\end{scope}
\draw (0.1, 0.3) .. controls +(0, 0.3) and +(0, 0.3) .. (0.7, 0.3);
\draw (-0.1, 0.3) .. controls +(0, 0.5) and +(0, 0.5) .. (0.9, 0.3);
\draw (0.1,- 0.3) .. controls +(0, -0.3) and +(0, -0.3) .. (0.7, -0.3);
\draw (-0.1, -0.3) .. controls +(0,- 0.5) and +(0,- 0.5) .. (0.9,- 0.3);
\end{tikzpicture}}=\raisebox{-0.7cm}{
\begin{tikzpicture}
\path [fill=gray!40] (-0.1, 0.3) .. controls +(0, 0.5) and +(0, 0.5) .. (0.9, 0.3);
\path [fill=white](0.1, 0.3) .. controls +(0, 0.3) and +(0, 0.3) .. (0.7, 0.3);
\path [fill=gray!40] (-0.1, -0.3) .. controls +(0,- 0.5) and +(0,- 0.5) .. (0.9,- 0.3);
\path [fill=white](0.1,- 0.3) .. controls +(0, -0.3) and +(0, -0.3) .. (0.7, -0.3);
\draw [blue, fill=white] (-0.3,-0.3) rectangle (0.3,0.3);
\node at (0,0) {$x$};
\begin{scope}[shift={(0.8,0)}]
\draw [blue, fill=white] (-0.3,-0.3) rectangle (0.3,0.3);
\node at (0,0) {$\overline{y}$};
\end{scope}
\draw (0.1, 0.3) .. controls +(0, 0.3) and +(0, 0.3) .. (0.7, 0.3);
\draw (-0.1, 0.3) .. controls +(0, 0.5) and +(0, 0.5) .. (0.9, 0.3);
\draw (0.1,- 0.3) .. controls +(0, -0.3) and +(0, -0.3) .. (0.7, -0.3);
\draw (-0.1, -0.3) .. controls +(0,- 0.5) and +(0,- 0.5) .. (0.9,- 0.3);
\end{tikzpicture}}=
\raisebox{-1.1cm}{
\begin{tikzpicture}
\path [fill=gray!40] (-0.1, 1.1) .. controls +(0, 0.5) and +(0, 0.5) .. (0.9,1.1)--(0.9,-0.3)..
 controls +(0, -0.5) and +(0, -0.5) .. (-0.1, -0.3);
 \path [fill=white](0.1, 1.1) .. controls +(0, 0.3) and +(0, 0.3) .. (0.7,1.1)--(0.7,-0.3).. controls +(0, -0.3) and +(0, -0.3)..(0.1, -0.3);
\draw(0.7,1.1)--(0.7,-0.3) (0.9,1.1)--(0.9,-0.3) (0.1,1.1)--(0.1,-0.3)(-0.1,1.1)--(-0.1,-0.3);
\draw [blue, fill=white] (-0.3,-0.3) rectangle (0.3,0.3);
\node at (0,0) {$y$};
\begin{scope}[shift={(0,0.8)}]
\draw [blue, fill=white] (-0.3,-0.3) rectangle (0.3,0.3);
\node at (0,0) {$x$};
\end{scope}
\draw (0.1, 1.1) .. controls +(0, 0.3) and +(0, 0.3) .. (0.7, 1.1);
\draw (-0.1, 1.1) .. controls +(0, 0.5) and +(0, 0.5) .. (0.9, 1.1);
\draw (0.1, -0.3) .. controls +(0, -0.3) and +(0, -0.3) .. (0.7, -0.3);
\draw (-0.1, -0.3) .. controls +(0, -0.5) and +(0, -0.5) .. (0.9, -0.3);
\end{tikzpicture}}=tr_{2,+}(yx)\;,
\end{align}
where $\overline {y}$ denotes the $180^\circ$ rotation of $y$.
The Fourier transform $\mathfrak{F}$: $\mathscr{P}_{2,+}\to\mathscr{P}_{2,-}$ is a $90^\circ$ rotation:
\begin{align}\label{eq:Fourier transform}
\mathfrak{F}(x):=\raisebox{-0.9cm}{
\begin{tikzpicture}[scale=1.5]
\path [fill=gray!40] (-0.3, -0.4) rectangle (0.8, 0.9);
\path [fill=white] (0.35, -0.4)--(0.35, 0.5) .. controls +(0, 0.3) and +(0, 0.3) .. (0.65, 0.5)--(0.65, -0.4);
\path[fill=white] (0.15, 0.9) -- (0.15, 0) .. controls +(0, -0.3) and +(0, -0.3) .. (-0.15, 0)--(-0.15, 0.9);
\draw [blue, fill=white] (0,0) rectangle (0.5, 0.5);
\node at (0.25, 0.25) {$x$};
\draw (0.35, 0)--(0.35, -0.4) (0.15, 0.5)--(0.15, 0.9);
\draw (0.35, 0.5) .. controls +(0, 0.3) and +(0, 0.3) .. (0.65, 0.5)--(0.65, -0.4);
\draw (0.15, 0) .. controls +(0, -0.3) and +(0, -0.3) .. (-0.15, 0)--(-0.15, 0.9);
\end{tikzpicture}}
\;.
\end{align}
The Fourier transform  induces the convolution on $\mathscr{P}_{2,+}$:
\begin{align}\label{eq:convolution1}
   x\ast y:=\mathfrak{F}^{-1}(\mathfrak{F}(y)\mathfrak{F}(x))=\raisebox{-0.9cm}{
\begin{tikzpicture}[scale=1.5]
\path [fill=gray!40] (0.15, -0.4) rectangle (1.05, 0.9);
\path [fill=white] (0.35, 0.5) .. controls +(0, 0.3) and +(0, 0.3) .. (0.85, 0.5)-- (0.85, 0) .. controls +(0, -0.3) and +(0, -0.3) .. (0.35, 0)--(0.35, 0.5);
\draw [blue, fill=white] (0,0) rectangle (0.5, 0.5);
\node at (0.25, 0.25) {$x$};
\draw (0.15, 0.5)--(0.15, 0.9) (0.15, 0)--(0.15, -0.4);
\begin{scope}[shift={(0.7, 0)}]
\draw [blue, fill=white] (0,0) rectangle (0.5, 0.5);
\node at (0.25, 0.25) {$y$};
\draw (0.35, 0.5)--(0.35, 0.9)  (0.35, 0)--(0.35, -0.4);
\end{scope}
\draw (0.35, 0.5) .. controls +(0, 0.3) and +(0, 0.3) .. (0.85, 0.5);
\draw (0.35, 0) .. controls +(0, -0.3) and +(0, -0.3) .. (0.85, 0);
\end{tikzpicture}}
\;.
\end{align}

\begin{notation}
We use $\widehat{x}$ to denote $\mathfrak{F}(x)$ in the subsequent sections for simplicity. 
\end{notation}
We next introduce a typical example of $C^*$-planar algebra.
 Let $\Gamma$ be the bipartite graph with one white vertex and $n$ black vertexes, $n\in \mathbb{N}$, and there is exactly one edge connecting the white and black vertexes:
\begin{align*}
\begin{tikzpicture}
\draw (0, 0) --(1, 0.4) (0, 0)--(1, -0.4) (0, 0)--(1, 0.8) (0,0)--(1, -0.8);
\node at (0.9, 0.1) {$\vdots$};
\begin{scope}
\draw [fill=white] (0, 0) circle (0.07);
\end{scope}
\begin{scope}[shift={(1, 0.4)}]
\draw [fill=black] (0, 0) circle (0.07);
\end{scope}
\begin{scope}[shift={(1, -0.4)}]
\draw [fill=black] (0, 0) circle (0.07);
\end{scope}
\begin{scope}[shift={(1, 0.8)}]
\draw [fill=black] (0, 0) circle (0.07);
\end{scope}
\begin{scope}[shift={(1, -0.8)}]
\draw [fill=black] (0, 0) circle (0.07);
\end{scope}
\end{tikzpicture}.
\end{align*}
A spin model is a bipartite graph planar algebra \cite{Jon99} associated to $\Gamma$, denoted by $Spin$.
Then $Spin_{0, +}=\mathbb{C}$ and $Spin_{0, -}$ is spanned by black vertexes, i.e. $Spin_{0, -}$ is of dimension $n$. Moreover,
\begin{align*}
Spin_{1, +}=& \text{span}\left\{
\raisebox{-0.3cm}{\begin{tikzpicture}[scale=1.5]
\path [fill=white] (0.25, 0) rectangle (0.05, 0.5);
\path [fill=gray!40] (0.25, 0) rectangle (0.55, 0.5);
\draw (0.25, 0)--(0.25, 0.5);
\node [right] at (0.25, 0.25) {$j$};
\end{tikzpicture}} : j=1, \ldots, n\right\},\
 Spin_{2, +} = \text{span}\left\{
\raisebox{-0.6cm}{
\begin{tikzpicture}[scale=1.5]
\path [fill=gray!40] (0.15, 0.9) .. controls +(0, -0.4) and +(0, -0.4) .. (0.75, 0.9);
\path [fill=gray!40] (0.75, 0).. controls +(0, 0.4) and +(0, 0.45) ..(0.15, 0);
\draw (0.15, 0.9) .. controls +(0, -0.4) and +(0, -0.4) .. (0.75, 0.9);
\draw (0.75, 0).. controls +(0, 0.4) and +(0, 0.45) ..(0.15, 0);
\node [below] at (0.45, 1) {$j$};
\node [above] at (0.45, 0) {$i$};
\end{tikzpicture}}: i, j =1, \ldots, n\right\}, \\
 Spin_{2, -} = & \text{span}\left\{
\raisebox{-0.6cm}{
\begin{tikzpicture}[scale=1.5]
\path [fill=gray!40] (0.15, 0.9) rectangle (0.35, 0);
\path [fill=gray!40] (0.55, 0.9) rectangle (0.75, 0);
\draw (0.35, 0)--(0.35, 0.9) (0.55, 0)--(0.55, 0.9);
\node [left] at (0.35, 0.45) {$j$};
\node [right] at (0.55, 0.45) {$i$};
\end{tikzpicture}}: i, j =1, \ldots, n\right\},
\end{align*}
and
\begin{align*}
\raisebox{-0.2cm}{
\begin{tikzpicture}
\draw (0,0) circle (0.3);
\end{tikzpicture} }=n^{\frac{1}2},
\quad
\raisebox{-0.2cm}{
\begin{tikzpicture}
\draw[fill=gray!40] (0,0) circle (0.3);
\node at (0,0) {$\substack{j\\i}$};
\end{tikzpicture}}=\begin{cases} 1/n^{\frac{1}2}&i=j\\
0 &i\neq j.
\end{cases}
\end{align*}
We see that $Spin_{2, +}$ is the matrix algebra $M_n(\mathbb{C})$ with basis
\begin{align}\label{eq:matrix unit}
    E_{ij}=n^{\frac{1}2}\raisebox{-0.6cm}{
\begin{tikzpicture}[scale=1.5]
\path [fill=gray!40] (0.15, 0.9) .. controls +(0, -0.4) and +(0, -0.4) .. (0.75, 0.9);
\path [fill=gray!40] (0.75, 0).. controls +(0, 0.4) and +(0, 0.45) ..(0.15, 0);
\draw (0.15, 0.9) .. controls +(0, -0.4) and +(0, -0.4) .. (0.75, 0.9);
\draw (0.75, 0).. controls +(0, 0.4) and +(0, 0.45) ..(0.15, 0);
\node [below] at (0.45, 1) {$j$};
\node [above] at (0.45, 0) {$i$};
\end{tikzpicture}},\ i, j =1, \ldots, n
\end{align}
Dually, $Spin_{2,-}$  is the commutative $C^*$-algebra
of functions on $n\times n$ points, denoted also by $M_n(\mathbb{C})$.
We also see that $Spin_{1,+}$ is a $n$-dimensional vector space. The following is the corresponding between matrix units in spin model to that in tensor networks:
\begin{align}
    n^{\frac{1}2}\raisebox{-0.6cm}{
\begin{tikzpicture}[scale=1.5]
\path [fill=gray!40] (0.15, 0.9) .. controls +(0, -0.4) and +(0, -0.4) .. (0.75, 0.9);
\path [fill=gray!40] (0.75, 0).. controls +(0, 0.4) and +(0, 0.45) ..(0.15, 0);
\draw (0.15, 0.9) .. controls +(0, -0.4) and +(0, -0.4) .. (0.75, 0.9);
\draw (0.75, 0).. controls +(0, 0.4) and +(0, 0.45) ..(0.15, 0);
\node [below] at (0.45, 1) {$j$};
\node [above] at (0.45, 0) {$i$};
\end{tikzpicture}} \longleftrightarrow \raisebox{-0.9cm}{
\begin{tikzpicture}
\draw (0,-1)--(0,1);
\draw [blue, fill=white] (-0.3,-0.3) rectangle (0.3,0.3);
\node at (0,0) {$E_{ij}$};
\node at (0.2,0.65) {$j$};
\node at (0.2,-0.65) {$i$};
\end{tikzpicture}}\;,\quad 
 n^{\frac{1}2}  \raisebox{-0.6cm}{
\begin{tikzpicture}[scale=1.5]
\path [fill=gray!40] (0.15, 0.9) rectangle (0.35, 0);
\path [fill=gray!40] (0.55, 0.9) rectangle (0.75, 0);
\draw (0.35, 0)--(0.35, 0.9) (0.55, 0)--(0.55, 0.9);
\node [left] at (0.35, 0.45) {$j$};
\node [right] at (0.55, 0.45) {$i$};
\end{tikzpicture}} \longleftrightarrow \raisebox{-0.2cm}{
\begin{tikzpicture}[rotate=90]
\draw (0,-1)--(0,1);
\draw [blue, fill=white] (-0.3,-0.3) rectangle (0.3,0.3);
\node at (0,0) {$E_{ij}$};
\node at (0.2,0.65) {$j$};
\node at (0.2,-0.65) {$i$};
\end{tikzpicture}} .
\end{align}
 $\{Spin_{2n,+}\}_{n\geq 0}$ is an unshaded planar algebra and a completely positive map could be realized by an element in $Spin_{4,+}$, which has a pictorial representation as
Equation \eqref{eq:4 tensor action}.

We have known that a matrix $A\in M_n(\mathbb{C})$ is positive if and only if the Fourier transform
$\widehat{\Phi_A}$ is positive-semidefinite. 
This means that we should interpret the matrix as pictures in $Spin_{2, -}$ and the multiplication is the horizontal composition, the positivity is read vertically.
Furthermore, a linear map $\Phi\in \text{Hom}(M_n(\mathbb{C}),M_n(\mathbb{C}) )$ is completely positive if and only if the Fourier transform $\widehat{\Phi}$ is positive-semidefinite. This fact inspires us to introduce the notion of 
$\mathfrak{F}$-positivity for any 2-box $x\in\mathscr{P}_{2,\pm}$.

Suppose $\mathfrak{A}$ is a $C^*$-algebra. Recall that an element $x\in\mathfrak{A}$ is called \textbf{positive}
if $x=yy^*$ for some $y\in\mathfrak{A}$; is called \textbf{strictly positive} if $x$ is positive and invertible,  i.e. there exists $\lambda>0$ such that $x \geq \lambda$.

\begin{definition}\label{def:F positive}
Suppose $\mathscr{P}$ is a C$^*$-planar algebra. An element $x\in \sP_{2, \pm}$ is called
 \textbf{(strictly) $\mathfrak{F}$-positive} if $\widehat{x}$ is a (strictly) positive element in $\mathscr{P}_{2,\mp}$.
\end{definition}

\begin{remark}
Suppose $\mathscr{P}$ is a  C$^*$-planar algebra. An element $x\in \sP_{2, \pm}$ is 
 (strictly) $\mathfrak{F}$-positive if and only if $x=\widehat{y}$ for some (strictly) positive element $y\in\mathscr{P}_{2,\mp}$.
\end{remark}

  \begin{remark}
       The positivity of an element in a $C^*$-algebra is different from the positivity of a matrix. If we consider $M_n(\mathbb{C})$ as a $C^*$-algebra, then an element $A$ in this $C^*$-algebra
is positive  means that $A$ is a  positive-semidefinite matrix.  
If we consider $M_n(\mathbb{C})$ as the two box space $Spin_{2, +}$ of the spin model, then an element $A\in Spin_{2, +}$ is $\mathfrak{F}$-positive
 means that $A$ is a positive matrix.       
    \end{remark}

\subsection{Quon Matrices}\label{subsec:quon}
We introduce a 4-string model for  matrices based on the Quon language~\cite{LWJ17}. We use the notation of planar para algebras for annotated planar algebras~\cite{JL17,JaLW18}, where lines in a diagram are annotated by an element of $\mathbb{Z}_d$ that interpret as a ``charge.'' In Quon, isotopy of diagrams in planar algebra is generalized to ``para-isotopy,'' in which a diagram obtains a phase when two charges pass each other vertically. 

The standard matrix basis for $d\times d$ matrices in Quon is given by a neutral pair in the tensor product of matrix algebra elements annotated by charges $k\in\Z_{d}$.  An elementary  Quon is a neutral pair of charges $\pm k\in\mathbb{Z} _{d}$, where $k$ takes values $0,1,\ldots,d-1$, and $k\equiv k\pm d$.  Thus a Quon matrix is an element of the $d^{2}$-dimensional neutral subspace of the $d^{2}\times d^{2}=d^{4}$-dimensional algebra  generated by 4 parafermions.

\textbf{Standard Basis}:
We choose as the standard neutral basis the product state,
	\begin{align}
	    \boldsymbol{\ket {k}}
	= \frac{1}{d^{\frac12}}
	\raisebox{-.3cm}
	{\begin{tikzpicture}
    \draw (0, -.5) -- (0, 0) arc [radius=.5, start angle=180, end angle=0]--(1, -.5);
    \node at (0, -.25)[left] {$k$};
    \draw (0+1.25, -.5) -- (0+1.25, 0) arc [radius=.5, start angle=180, end angle=0]--(1+1.25, -.5);
    \node at (1+1.25, -.25)[left] {$-k$};
\end{tikzpicture}
}\;, \quad\text{for}\quad k\in \mathbb{Z} _{d}\;.
	\end{align}
We use a  bold font to indicate the Quon basis from the standard basis in \eqref{eq:matrix unit}. The vector $\boldsymbol{\ket{k}}$ can also be regarded as the twisted tensor product of the two-string $\ket{k}$ with $\ket{-k}$, namely $\zeta^{k^{2}}\ket{k}\otimes \ket{-k}$, graphically,
\begin{align}\label{eq:Cap Fourier Relation}
 \raisebox{-.3cm}
	{\begin{tikzpicture}
    \draw (0, -.5) -- (0, 0) arc [radius=.5, start angle=180, end angle=0]--(1, -.5);
    \node at (0, -.25)[left] {$k$};
    \draw (0+1.25, -.5) -- (0+1.25, 0) arc [radius=.5, start angle=180, end angle=0]--(1+1.25, -.5);
    \node at (1+1.25, -.25)[left] {$-k$};
\end{tikzpicture}
}=\zeta^{k^2}\ 
\raisebox{-.3cm}
	{\begin{tikzpicture}
    \draw (0, -.5) -- (0, 0) arc [radius=.5, start angle=180, end angle=0]--(1, -.5);
    \node at (1, -.25)[left] {$k$};
    \draw (0+1.25, -.5) -- (0+1.25, 0) arc [radius=.5, start angle=180, end angle=0]--(1+1.25, -.5);
    \node at (1+1.25, -.25)[left] {$-k$};
\end{tikzpicture}
}\;,
\end{align}
where $\zeta$ is a square root of $e^{\frac{2\pi i}{d}}$ such that $\zeta^{d^2}=1$; this led to the pictorial interpretation of the interpolation of two possible tensor products, introduced in~\cite{JL17}.   The matrix units $\boldsymbol{M_{ k\ell}}$ are then 

	\begin{align} \label{QuonMatrixUnit}
	    \boldsymbol{M_{ k\ell}}=\boldsymbol{\ket{k}\bra{\ell}}
	= \frac{1}{d} \hskip-.2cm
	\raisebox{-1cm}
	 {\begin{tikzpicture}
    \draw (0, .5) -- (0, 0) arc [radius=.5, start angle=180, end angle=360]--(1, .5);
    \node at (0, .25)[left] {$-\ell$};
    \draw (0, -1.75)--(0, -1.25) arc[radius=.5, start angle = 180, end angle=0]--(1, -1.75);
    \node at (0, -1.5) [left] {$k$};
    \draw (0+1.25, .5) -- (0+1.25, 0) arc [radius=.5, start angle=180, end angle=360]--(1+1.25, .5);
    \node at (1+1.25, .25)[left] {$\ell$};
    \draw (0+1.25, -1.75)--(0+1.25, -1.25) arc[radius=.5, start angle = 180, end angle=0]--(1+1.25, -1.75);
    \node at (1+1.25, -1.5) [left] {$-k$};
	\end{tikzpicture}}
	\end{align}
 The adjoint $\Star$ of the unit element is given by
	\begin{align}
	\Star\  	
	\raisebox{-1cm} {\begin{tikzpicture}
    \draw (0, .5) -- (0, 0) arc [radius=.5, start angle=180, end angle=360]--(1, .5);
    \node at (0, .25)[left] {$-\ell$};
    \draw (0, -1.75)--(0, -1.25) arc[radius=.5, start angle = 180, end angle=0]--(1, -1.75);
    \node at (0, -1.5) [left] {$k$};
    \draw (0+1.25, .5) -- (0+1.25, 0) arc [radius=.5, start angle=180, end angle=360]--(1+1.25, .5);
    \node at (1+1.25, .25)[left] {$\ell$};
    \draw (0+1.25, -1.75)--(0+1.25, -1.25) arc[radius=.5, start angle = 180, end angle=0]--(1+1.25, -1.75);
    \node at (1+1.25, -1.5) [left] {$-k$};
	\end{tikzpicture}}
	\ \ =
		\raisebox{-1cm} {\begin{tikzpicture}
    \draw (0, .5) -- (0, 0) arc [radius=.5, start angle=180, end angle=360]--(1, .5);
    \node at (0, .25)[left] {$-k$};
    \draw (0, -1.75)--(0, -1.25) arc[radius=.5, start angle = 180, end angle=0]--(1, -1.75);
    \node at (0, -1.5) [left] {$\ell$};
    \draw (0+1.25, .5) -- (0+1.25, 0) arc [radius=.5, start angle=180, end angle=360]--(1+1.25, .5);
    \node at (1+1.25, .25)[left] {$k$};
    \draw (0+1.25, -1.75)--(0+1.25, -1.25) arc[radius=.5, start angle = 180, end angle=0]--(1+1.25, -1.75);
    \node at (1+1.25, -1.5) [left] {$-\ell$};
	\end{tikzpicture}}\;.
	\end{align}
The representation of the $d\times d$ matrix $T=(t_{k\ell})_{k,\ell\in\Z_{d}}$ 
 is $T=\sum_{k,\ell\in\Z_{d}} T_{k\ell} \boldsymbol{\ket{k}\bra{\ell}}$, where $T_{k\ell}=dt_{k\ell}$,
 pictorially, 
\begin{align} \label{QuonMatrixT}
    {T= \frac{1}{d} \sum_{k,\ell} T_{k\ell} 
		\raisebox{-1cm}
	 {\begin{tikzpicture}
    \draw (0, .5) -- (0, 0) arc [radius=.5, start angle=180, end angle=360]--(1, .5);
    \node at (0, .25)[left] {$-\ell$};
    \draw (0, -1.75)--(0, -1.25) arc[radius=.5, start angle = 180, end angle=0]--(1, -1.75);
    \node at (0, -1.5) [left] {$k$};
    \draw (0+1.25, .5) -- (0+1.25, 0) arc [radius=.5, start angle=180, end angle=360]--(1+1.25, .5);
    \node at (1+1.25, .25)[left] {$\ell$};
    \draw (0+1.25, -1.75)--(0+1.25, -1.25) arc[radius=.5, start angle = 180, end angle=0]--(1+1.25, -1.75);
    \node at (1+1.25, -1.5) [left] {$-k$};
	\end{tikzpicture}}}\;.
\end{align}
We deﬁne a non-neutral circle to have value zero, while a neutral circle has a positive value:
\be\label{eq:neutral circle}
	 \raisebox{-1.25cm}
	 {
	\begin{tikzpicture}
    \draw (0, -.5) -- (0, 0) arc [radius=.8, start angle=180, end angle=0]--(1.6, -.5);
    \node at (0+1.6, 0)[left] {$\scriptstyle \ell$};
    \draw (0, .5-1.1) -- (0, 0-1.1) arc [radius=.8, start angle=180, end angle=360]--(1.6, .5-1.1);
    \node at (0+1.6, -1)[left] {$\scriptstyle -k$};

\end{tikzpicture}}
= 
	 \raisebox{-.65cm}{\begin{tikzpicture}
    \draw (0, 0) circle [radius = .8];
    \node at (.8,0) [left] {$\scriptstyle \ell-k$};
\end{tikzpicture}}
= d^{\frac12}\,\delta_{\ell, k} \;.
	\ee
Equations \eqref{eq:Cap Fourier Relation}, \eqref{QuonMatrixT} and \eqref{eq:neutral circle}  give rise to the multiplication 
	\begin{align}
	\raisebox{-1.2cm}{\begin{tikzpicture}
    \draw [blue] (0,0) rectangle (1, .6);
    \draw (.2, -1)--(.2,0);
    \draw (.2, .6)--(.2,1.7);
    \draw (.4, -1)--(.4,0);
    \draw (.4, .6)--(.4,1.7);
    \draw (.6, -1)--(.6,0);
    \draw (.6, .6)--(.6,1.7);
    \draw (.8, -1)--(.8,0);
    \draw (.8, .6)--(.8,1.7);
    \node at (.5, .3) {$TS$};
\end{tikzpicture}}	
	= 
	\raisebox{-1.2cm}{\begin{tikzpicture}
    \draw (0,0)[blue] rectangle (1, .6);
    \draw (0,1.1)[blue] rectangle (1, 1.7);
    \draw (.2, -.5)--(.2,0);
    \draw (.2, .6)--(.2,1.1);
    \draw (.2, 1.7)--(.2, 2.2);
     \draw (.2+.2, -.5)--(.2+.2,0);
    \draw (.2+.2, .6)--(.2+.2,1.1);
    \draw (.2+.2, 1.7)--(.2+.2, 2.2);
     \draw (.2+.4, -.5)--(.2+.4,0);
    \draw (.2+.4, .6)--(.2+.4,1.1);
    \draw (.2+.4, 1.7)--(.2+.4, 2.2);
     \draw (.2+.6, -.5)--(.2+.6,0);
    \draw (.2+.6, .6)--(.2+.6,1.1);
    \draw (.2+.6, 1.7)--(.2+.6, 2.2);
    \node at (.5, .3) {$T$};
    \node at (.5, 1.4) {$S$};
\end{tikzpicture}
}\;.
\end{align}

\paragraph{\textbf{Fourier Transform of a Quon Matrix:}}
The natural Fourier transform in the Quon model is 
$
\mbff\,= \mathfrak{F}^{2}
$,
where $\mff$ is the Fourier transform \eqref{eq:Fourier transform}.  The Fourier transform on Quon matrix $T$ is  
	\be\label{FourierTransformQuon}
	\mbff (T) = 
	\raisebox{-1cm}{
	\begin{tikzpicture}
    \draw(0,0)[blue] rectangle(1, .6);
    \node at (.5, .3) {$T$};
    \draw (.2, 0) arc[radius=.2, start angle=360, end angle=180]--(-.2,1.4);
    \draw (.4, 0) arc[radius=.4, start angle=360, end angle=180]--(-.4,1.4);
    \draw (.6, .6) arc[radius=.4, start angle=180, end angle=0]--(1.4, -.8);
    \draw (.8, .6) arc[radius=.2, start angle=180, end angle=0]--(1.2, -.8);
    \draw (.2, .6) -- (.2, 1.4);
    \draw (.4, .6) -- (.4, 1.4);
    \draw (.6, 0) -- (.6, -.8);
    \draw (.8, 0) -- (.8, -.8);
\end{tikzpicture}
	} =\frac{1}{d} \sum_{k,\ell} T_{k\ell} 
		\raisebox{-1cm}
	 {\begin{tikzpicture}
    \draw (0, .5) -- (0, 0) arc [radius=.5, start angle=180, end angle=360]--(1, .5);
    \node at (0, .25)[left] {$k$};
    \draw (0, -1.75)--(0, -1.25) arc[radius=.5, start angle = 180, end angle=0]--(1, -1.75);
    \node at (0, -1.5) [left] {$-k$};
    \draw (0+1.25, .5) -- (0+1.25, 0) arc [radius=.5, start angle=180, end angle=360]--(1+1.25, .5);
    \node at (1+1.25, .25)[left] {$-\ell$};
    \draw (0+1.25, -1.75)--(0+1.25, -1.25) arc[radius=.5, start angle = 180, end angle=0]--(1+1.25, -1.75);
    \node at (1+1.25, -1.5) [left] {$\ell$};
	\end{tikzpicture}} \;.
	\ee
In more detail, the matrix units $\bs{M_{\ell k}}$ are transformed into rank-one projections $\bs{P_{-\ell, k}}$, the coefficients of $T_{k\ell}$ in \eqref{FourierTransformQuon}:
\be
{
\bs{P_{-\ell, k}} = \frac{1}{d}
		\raisebox{-1cm}
	 {\begin{tikzpicture}
    \draw (0, .5) -- (0, 0) arc [radius=.5, start angle=180, end angle=360]--(1, .5);
    \node at (0, .25)[left] {$k$};
    \draw (0, -1.75)--(0, -1.25) arc[radius=.5, start angle = 180, end angle=0]--(1, -1.75);
    \node at (0, -1.5) [left] {$-k$};
    \draw (0+1.25, .5) -- (0+1.25, 0) arc [radius=.5, start angle=180, end angle=360]--(1+1.25, .5);
    \node at (1+1.25, .25)[left] {$-\ell$};
    \draw (0+1.25, -1.75)--(0+1.25, -1.25) arc[radius=.5, start angle = 180, end angle=0]--(1+1.25, -1.75);
    \node at (1+1.25, -1.5) [left] {$\ell$};
	\end{tikzpicture}}
}\;.
\ee

In the two string model we had a natural definition of a convolution product, namely  \eqref{eq:convolution1}.  For Quon matrices, there is also the possibility to give two definitions to the convolution of two Quon matrices. Now with the Fourier transform $\mbff$ replacing $\mathfrak{F}$, one can  define 
\begin{align}
   \label{Defn:QuonConvolution}
	T\bs{\circ_{1}}S = \mbff (\mbff^{-1}(T)\, \mbff^{-1}(S) )\;,
	\qquad\text{and}\quad
	T\bs{\circ_{2}}S=  \mbff^{-1} (\mbff(T)\, \mbff(S) ) \;. 
\end{align}	
Note that the Fourier transform $\mbff T$ commutes with $\mbff S$;  so both convolution products are abelian. Moreover,	 the two  convolutions  $T\bs{\circ_{1}}S $ and $T\bs{\circ_{2}}S$ agree, and we denote them by $T\bs{\circ}S$. 
Pictorially,
\begin{align}
    \label{QuonConvolutionDiagram}
T\bs{\circ}S 
	= 
	\raisebox{-1cm}{
	\begin{tikzpicture}
    \draw(0,0) [blue]rectangle(1, .6);
    \node at (.5, .3) {$T$};
    \draw(1.5,0)[blue] rectangle(2.5, .6);
    \node at (2, .3) {$S$};
    \draw (.2, .6) -- (.2, 1.4);
    \draw (.4, .6) -- (.4, 1.4);
    \draw (.2, 0) -- (.2, -.8);
    \draw (.4, 0) -- (.4, -.8);
    \draw (.8, 0) arc[radius=.45, start angle=180, end angle=360];
    \draw (.6, 0) arc[radius=.65, start angle=180, end angle=360];
    \draw (.8, .6) arc[radius=.45, start angle=180, end angle=0];
    \draw (.6, .6) arc[radius=.65, start angle=180, end angle=0];
    \draw (2.3, .6) -- (2.3, 1.4);
    \draw (2.1, .6) -- (2.1, 1.4);
    \draw (2.3, 0) -- (2.3, -.8);
    \draw (2.1, 0) -- (2.1, -.8);
\end{tikzpicture}
}
=	\raisebox{-1cm}{
	\begin{tikzpicture}
    \draw(0,0)[blue] rectangle(1, .6);
    \node at (.5, .3) {$S$};
    \draw(1.5,0)[blue] rectangle(2.5, .6);
    \node at (2, .3) {$T$};
    \draw (.2, .6) -- (.2, 1.4);
    \draw (.4, .6) -- (.4, 1.4);
    \draw (.2, 0) -- (.2, -.8);
    \draw (.4, 0) -- (.4, -.8);
    \draw (.8, 0) arc[radius=.45, start angle=180, end angle=360];
    \draw (.6, 0) arc[radius=.65, start angle=180, end angle=360];
    \draw (.8, .6) arc[radius=.45, start angle=180, end angle=0];
    \draw (.6, .6) arc[radius=.65, start angle=180, end angle=0];
    \draw (2.3, .6) -- (2.3, 1.4);
    \draw (2.1, .6) -- (2.1, 1.4);
    \draw (2.3, 0) -- (2.3, -.8);
    \draw (2.1, 0) -- (2.1, -.8);
\end{tikzpicture}
}
= S\circ T\;.
\end{align}
    The Hadamard product  $T \times S$ of two $d\times d$ matrices is 
$(T\times S)_{k\ell}=t_{k\ell}s_{k\ell}$. The convolution of two Quon  matrices  equals their Hadamard product: $T\times S=T\bs{\circ} S$.

One commonly encounters two types of positivity for matrices.  The first positivity is basis dependent, and here we study the Quon basis. Let us say that a matrix $T$ with entries $t_{k\ell}$ is ``point-wise''\footnote{We use point-wise is the sense that a function $f(t)$ is said to be ``point-wise'' non-negative if $f(t)\geq0$.} non-negative if  $t_{k\ell}\geq0$, for all $k,\ell$. Analogously one says that $T$ is point-wise \textit{strictly} positive if $t_{k\ell}>0$, for all $k,\ell$. 

The second type of positivity is independent of basis, and refers to the spectrum of $T$, namely the set of numbers $z$ in the complex plane for which $z-T$ is not invertible. For an hermitian matrix the eigenvalues are real, but the representation of matrices in spin model shows that the matrix has a reflection-positive representation, if and only if the eigenvalues are non-negative (See \S \ref{subsec:planar algebra}).

\section{PF Theory for Planar Algebras}\label{sec:PF theorem}
 In this section, we obtain the PF theorem for $C^*$-planar algebras. Moreover,
  we provide many equivalent conditions for the irreducibility and planar algebra structures play 
  an essential role.   We also prove the quantum Collatz-Wielandt formula.

  We shall recall the quantum Schur product theorem~\cite{Liu16} in $C^*$-planar algebras: For two (strictly) positive elements $x,y\in\mathscr{P}_{2,\pm}$, their convolution $x\ast y$ is (strictly) positive. This fact has been deeply applied in classification of subfactor planar algebras~\cite{Liu16} and as an analytic obstruction of unitary categorification of fusion rings \cite{LPW21}. Here, we would use it again to show the PF theorem.
  
 \subsection{The PF Theorem for $C^*$-Planar Algebras}
We give a proof of the PF theorem in $C^*$-planar algebras.
\begin{proposition}  \label{prop:PF for planar algebra}
Let $x$ be a $\mathfrak{F}$-positive element of the $2$-box space $\mathscr{P}_{2,\pm}$ of a $C^*$-planar algebra.
Then there exists a non-zero, $\mathfrak{F}$-positive  $y\in\mathscr{P}_{2,\pm}$ such that
\begin{align}
  \raisebox{-1.4cm}{
\begin{tikzpicture}
\path  (-0.1, -2) rectangle (0.1, 1);
\draw (-0.1,-2)--(-0.1,1);
\draw (0.1,-2)--(0.1,1);
\draw [blue, fill=white] (-0.3,-0.3) rectangle (0.3,0.3);
\node at (0,0) {$y$};
\begin{scope}[shift={(0, -1)}]
\draw [blue, fill=white] (-0.3,-0.3) rectangle (0.3,0.3);
\node at (0,0) {$x$};
\end{scope}
\end{tikzpicture}}=  \raisebox{-1.4cm}{
\begin{tikzpicture}
\path  (-0.1, -2) rectangle (0.1, 1);
\draw (-0.1,-2)--(-0.1,1);
\draw (0.1,-2)--(0.1,1);
\draw [blue, fill=white] (-0.3,-0.3) rectangle (0.3,0.3);
\node at (0,0) {$x$};
\begin{scope}[shift={(0, -1)}]
\draw [blue, fill=white] (-0.3,-0.3) rectangle (0.3,0.3);
\node at (0,0) {$y$};
\end{scope}
\end{tikzpicture}}=r\raisebox{-0.9cm}{
\begin{tikzpicture}
\path (-0.1, -1) rectangle (0.1, 1);
\draw (-0.1,-1)--(-0.1,1);
\draw (0.1,-1)--(0.1,1);
\draw [blue, fill=white] (-0.3,-0.3) rectangle (0.3,0.3);
\node at (0,0) {$y$};
\end{tikzpicture}}\;,
\end{align}
where $r$ is the spectral radius of $x$.
Equivalently, 
\begin{align}\label{eq:perron operator Fourier}
\raisebox{-0.9cm}{
\begin{tikzpicture}[scale=1.5]
\path  (0.15, -0.4) rectangle (1.05, 0.9);
\path [fill=white] (0.35, 0.5) .. controls +(0, 0.3) and +(0, 0.3) .. (0.85, 0.5)-- (0.85, 0) .. controls +(0, -0.3) and +(0, -0.3) .. (0.35, 0)--(0.35, 0.5);
\draw [blue, fill=white] (0,0) rectangle (0.5, 0.5);
\node at (0.25, 0.25) {$\widehat{x}$};
\draw (0.15, 0.5)--(0.15, 0.9) (0.15, 0)--(0.15, -0.4);
\begin{scope}[shift={(0.7, 0)}]
\draw [blue, fill=white] (0,0) rectangle (0.5, 0.5);
\node at (0.25, 0.25) {$\widehat{y}$};
\draw (0.35, 0.5)--(0.35, 0.9)  (0.35, 0)--(0.35, -0.4);
\end{scope}
\draw (0.35, 0.5) .. controls +(0, 0.3) and +(0, 0.3) .. (0.85, 0.5);
\draw (0.35, 0) .. controls +(0, -0.3) and +(0, -0.3) .. (0.85, 0);
\end{tikzpicture}}=\raisebox{-0.9cm}{
\begin{tikzpicture}[scale=1.5]
\path  (0.15, -0.4) rectangle (1.05, 0.9);
\path [fill=white] (0.35, 0.5) .. controls +(0, 0.3) and +(0, 0.3) .. (0.85, 0.5)-- (0.85, 0) .. controls +(0, -0.3) and +(0, -0.3) .. (0.35, 0)--(0.35, 0.5);
\draw [blue, fill=white] (0,0) rectangle (0.5, 0.5);
\node at (0.25, 0.25) {$\widehat{y}$};
\draw (0.15, 0.5)--(0.15, 0.9) (0.15, 0)--(0.15, -0.4);
\begin{scope}[shift={(0.7, 0)}]
\draw [blue, fill=white] (0,0) rectangle (0.5, 0.5);
\node at (0.25, 0.25) {$\widehat{x}$};
\draw (0.35, 0.5)--(0.35, 0.9)  (0.35, 0)--(0.35, -0.4);
\end{scope}
\draw (0.35, 0.5) .. controls +(0, 0.3) and +(0, 0.3) .. (0.85, 0.5);
\draw (0.35, 0) .. controls +(0, -0.3) and +(0, -0.3) .. (0.85, 0);
\end{tikzpicture}}=r\raisebox{-0.9cm}{
\begin{tikzpicture}
\path (-0.1, -1) rectangle (0.1, 1);
\draw (-0.1,-1)--(-0.1,1);
\draw (0.1,-1)--(0.1,1);
\draw [blue, fill=white] (-0.3,-0.3) rectangle (0.3,0.3);
\node at (0,0) {$\widehat{y}$};
\end{tikzpicture}}
\;.
\end{align}
\end{proposition}


\begin{proof}
We first show that $r\in \sigma(x)$, the spectrum of $x$. For any $\lambda\in \rho(x)$, let $R(\lambda, x)=(\lambda-x)^{-1}$ be the resolvent of $x$, which defines an element in $\mathscr{P}_{2,\pm}$.
We have that $R(\lambda, x)$ is analytic in $\rho(x)$.
For any $|\lambda|> r$, we have that
\begin{align}\label{eq:resolvent}
R(\lambda, x)=\sum_{j=0}^\infty \frac{x^j}{\lambda^{j+1}}.
\end{align}
We assume that $r\notin \sigma(x)$.
By the fact that $R(\lambda, x)$ is analytic outside of the spectrum $\sigma(x)$, we see that $\displaystyle \lim_{\lambda\to r} R(\lambda, x)=R(r, x)$.
By the continuity of  Fourier transform, 
\begin{align}\label{paeq:eigen1}
\lim_{z\to r} \widehat{R(\lambda, x)}=\widehat{R(r, x)}.
\end{align}
By the quantum Schur product theorem, we have that $\widehat{x}^{(*j)}\geq 0$ for any $j \geq 1$, where $\widehat{x}^{(*j)}$ is $\underbrace{\widehat{x}*\cdots * \widehat{x}}_{j}$.
Therefore, for any real number $\lambda>r$,
\begin{align*}
\widehat{R(\lambda, x)}=\sum_{j=0}^\infty \frac{\widehat{x}^{(*j)}}{\lambda^{j+1}}\geq 0.
\end{align*}
By Equation (\ref{paeq:eigen1}), 
\begin{align*}
\widehat{R(r, x)}=\sum_{j=0}^\infty \frac{\widehat{x}^{(*j)}}{r^{j+1}}\geq 0.
\end{align*}
For any positive linear functional $\phi$ on $\mathscr{P}_{2,\pm}$, 
\begin{align*}
\phi(\widehat{R(r, x)})=\sum_{j=0}^\infty \frac{\phi(\widehat{x}^{(*j)})}{r^{j+1}}< \infty.
\end{align*}
Hence for any $|\lambda|=r$, 
\begin{align*}
\left|\sum_{j=0}^\infty \frac{\phi\left(\widehat{x}^{(*j)}\right)}{\lambda^{j+1}} \right|< \infty.
\end{align*}
By the uniform boundedness principle, we have that $\displaystyle \sum_{j=0}^\infty \frac{\phi(\widehat{x}^{(*j)})}{\lambda^{j+1}}$ is bounded for any $|\lambda|=r$.
This implies that $\{\lambda\in \bC: |\lambda|=r\}\subset \bC\backslash \sigma(x)$, which is contradict to the fact that $r$ is the spectral radius of $x$. So $r\in\sigma(x)$.

Next we show that there exists a nonzero $\mathfrak{F}$-positive element $y\in\mathscr{P}_{2,\pm}$ such that $xy=yx=ry$.
Note that the algebraic multiplicity of $r$ for $x$ is finite, we see that the order of the pole $r$ for $R(\lambda, x)$ is also finite.
Let $\ell$ be the order of the pole $r$ for $R(\lambda, x)$.
We take the Laurent series of $R(\lambda, x)$ at $\lambda=r$ as follows:
\begin{align*}
R(\lambda, x)=\sum_{j=-\ell}^\infty x_{(j)} (\lambda-r)^j,
\end{align*}
where
\begin{align}\label{paeq:coeff1}
x_{(j)}=\frac{1}{2\pi i} \int_C \frac{R(\lambda,x)}{(\lambda-r)^{j+1}} d\lambda,
\end{align}
and $C$ is a circle centered at $r$ with radius small enough such that $r$ is the only pole.
For any $k, m \geq -\ell$, we have
\begin{align*}
x_{(k)} x_{(m)}= & \frac{1}{2\pi i} \int_{C_1} \frac{R(\lambda_1, x)}{(\lambda_1-r)^{k+1}} d\lambda_1 \frac{1}{2\pi i}
\int_{C_2} \frac{R(\lambda_2, x)}{(\lambda_2-r)^{m+1}} d\lambda_2 \\
=&- \frac{1}{4\pi^2 } \int_{C_1} \int_{C_2} \frac{R(\lambda_1, x) R(\lambda_2, x)}{(\lambda_1-r)^{k+1} (\lambda_2-r)^{m+1}} d\lambda_1d\lambda_2 \\
=& - \frac{1}{4\pi^2 } \int_{C_1} \int_{C_2} \frac{R(\lambda_1, x) -R(\lambda_2, x)}{(\lambda_2-\lambda_1)(\lambda_1-r)^{k+1} (\lambda_2-r)^{m+1}} d\lambda_1d\lambda_2 \\
= & \left\{\begin{array}{cc}
-x_{(k+m+1)} & k, m \geq 0 \\
0 & k<0, m \geq 0 \ or\ k\geq0,m<0\\
x_{(k+m+1)} & k,m <0
\end{array} \right.
\end{align*}
where $C_1, C_2$ are circles centered at $r$ with different radius such that $C_1$ is contained in $C_2$. Then
\begin{align*}
x_{(-\ell)}R(\lambda,x)  =\sum_{j=-\ell}^\infty (\lambda-r)^j x_{(-\ell)} x_{(j)} =(\lambda-r)^{-1} x_{(-\ell)} x_{(-1)}=\frac{x_{(-\ell)}}{\lambda-r}.
\end{align*}
Thus $(\lambda-r)x_{(-\ell)} R(\lambda,x) =x_{(-\ell)}$.
By Equation \eqref{eq:resolvent} and taking the coefficient of $\lambda^{-2}$, we have
\begin{align}\label{eq:x-l}
x_{(-\ell)}x =rx_{(-\ell)}.
\end{align}
Note that
\begin{align}\label{eq:perron operator}
\widehat{x_{(-\ell)}}=\lim_{\mathbb{R}\ni \lambda\to r^+}(\lambda-r)^\ell \widehat{R(\lambda,x)} =\lim_{\mathbb{R}\ni \lambda\to r^+}(\lambda-r)^\ell\sum_{j=0}^\infty \frac{\widehat{x}^{(\ast j)} }{\lambda^{j+1}}\geq0.
\end{align}
So $x_{(-\ell)}$ is $\mathfrak{F}$-positive. Similarly, we have $xx_{(-\ell)} =rx_{(-\ell)}$. Let $y=x_{(-\ell)}$, then the conclusion holds.
\end{proof}

\begin{remark}\label{rem:one box PF}
In Proposition \ref{prop:PF for planar algebra}, we let
\begin{align*}
y_1
=
\raisebox{-0.7cm}{
\begin{tikzpicture}[scale=1.5]
\draw [blue] (0,0) rectangle (0.7, 0.5);
\node at (0.35, 0.25) {$\widehat{y}$};
\draw (-0.15, 0)--(-0.15, 0.5) (0.55, 0)--(0.55, -0.2) (0.55, 0.5)--(0.55, 0.7);
\draw (0.15, 0.5) .. controls +(0, 0.3) and +(0, 0.3) .. (-0.15, 0.5);
\draw (0.15, 0) .. controls +(0, -0.3) and +(0, -0.3) .. (-0.15, 0);
\end{tikzpicture}},
\quad
y_2=
\raisebox{-0.7cm}{
\begin{tikzpicture}[scale=1.5]
\draw [blue] (0,0) rectangle (0.7, 0.5);
\node at (0.35, 0.25) {$\widehat{y}$};
\draw (0.85, 0)--(0.85, 0.5) (0.15, 0)--(0.15, -0.2) (0.15, 0.5)--(0.15, 0.7);
\draw (0.55, 0.5) .. controls +(0, 0.3) and +(0, 0.3) .. (0.85, 0.5);
\draw (0.55, 0) .. controls +(0, -0.3) and +(0, -0.3) .. (0.85, 0);
\end{tikzpicture}}.
\end{align*}
We have that for any $\mathfrak{F}$-positive element  $x\in\mathscr{P}_{2,\pm}$, there exist nonzero positive elements $y_1\in\mathscr{P}_{1,\pm}$ and $y_2\in\mathscr{P}_{1,\mp}$
such that
\begin{align*}
y_1\ast \widehat{x}:=\raisebox{-0.9cm}{
\begin{tikzpicture}[scale=1.5]
\begin{scope}[shift={(-0.6, 0)}]
\draw [blue] (0,0) rectangle (0.35, 0.5);
\node at (0.175, 0.25) {$y_1$};
\end{scope}
\draw [blue] (0,0) rectangle (0.7, 0.5);
\node at (0.35, 0.25) {$\widehat{x}$};
\draw  (0.5, 0)--(0.5, -0.4) (0.5, 0.5)--(0.5, 0.9);
\draw (0.2, 0.5) .. controls +(0, 0.4) and +(0, 0.4) .. (-0.425, 0.5);
\draw (0.2, 0) .. controls +(0, -0.4) and +(0, -0.4) .. (-0.425, 0);
\end{tikzpicture}}
=r
\raisebox{-0.9cm}{
\begin{tikzpicture}[scale=1.5]
\begin{scope}[shift={(1.5, 0)}]
\draw [blue] (0,0) rectangle (0.35, 0.5);
\node at (0.175, 0.25) {$y_1$};
\draw  (0.175, 0)--(0.175, -0.4) (0.175, 0.5)--(0.175, 0.9);
\end{scope}
\end{tikzpicture}},\quad
\widehat{x}\ast y_2:=\raisebox{-0.9cm}{
\begin{tikzpicture}[scale=1.5]
\begin{scope}[shift={(0.95, 0)}]
\draw [blue] (0,0) rectangle (0.35, 0.5);
\node at (0.175, 0.25) {$y_2$};
\end{scope}
\draw [blue] (0,0) rectangle (0.7, 0.5);
\node at (0.35, 0.25) {$\widehat{x}$};
\draw  (0.2, 0)--(0.2, -0.4) (0.2, 0.5)--(0.2, 0.9);
\draw (0.5, 0.5) .. controls +(0, 0.4) and +(0, 0.4) .. (1.125, 0.5);
\draw (0.5, 0) .. controls +(0, -0.4) and +(0, -0.4) .. (1.125, 0);
\end{tikzpicture}}
=r
\raisebox{-0.9cm}{
\begin{tikzpicture}[scale=1.5]
\begin{scope}[shift={(1.5, 0)}]
\draw [blue] (0,0) rectangle (0.35, 0.5);
\node at (0.175, 0.25) {$y_2$};
\draw  (0.175, 0)--(0.175, -0.4) (0.175, 0.5)--(0.175, 0.9);
\end{scope}
\end{tikzpicture}}
\end{align*}
where $r$ is the spectral radius of $x$.
\end{remark}

\begin{remark}
Suppose $\mathscr{P}$ is a  $C^*$-planar algebra. We call a $2n$-box $x\in\mathscr{P}_{2n,\pm}$ 
 is $\mathfrak{F}$-positive if $\mathfrak{F}^n(x)\geq0$. Here $\mathfrak{F}$ is the one-click rotation.
Proposition \ref{prop:PF for planar algebra} also holds for $2n$-box:  Let $x$ be a $\mathfrak{F}$-positive element of the $2n$-box space $\mathscr{P}_{2n,\pm}$ of a $C^*$-planar algebra.
Then there exists a non-zero, $\mathfrak{F}$-positive  $y\in\mathscr{P}_{2n,\pm}$ such that
\begin{align}
  \raisebox{-1.4cm}{
\begin{tikzpicture}
\path  (-0.1, -2) rectangle (0.1, 1);
\draw [line width=0.1cm](-0.1,-2)--(-0.1,1);
\draw [line width=0.1cm](0.1,-2)--(0.1,1);
\draw [blue, fill=white] (-0.3,-0.3) rectangle (0.3,0.3);
\node at (0,0) {$y$};
\begin{scope}[shift={(0, -1)}]
\draw [blue, fill=white] (-0.3,-0.3) rectangle (0.3,0.3);
\node at (0,0) {$x$};
\end{scope}
\end{tikzpicture}}=  \raisebox{-1.4cm}{
\begin{tikzpicture}
\path  (-0.1, -2) rectangle (0.1, 1);
\draw[line width=0.1cm] (-0.1,-2)--(-0.1,1);
\draw [line width=0.1cm](0.1,-2)--(0.1,1);
\draw [blue, fill=white] (-0.3,-0.3) rectangle (0.3,0.3);
\node at (0,0) {$x$};
\begin{scope}[shift={(0, -1)}]
\draw [blue, fill=white] (-0.3,-0.3) rectangle (0.3,0.3);
\node at (0,0) {$y$};
\end{scope}
\end{tikzpicture}}=r\raisebox{-0.9cm}{
\begin{tikzpicture}
\path (-0.1, -1) rectangle (0.1, 1);
\draw [line width=0.1cm] (-0.1,-1)--(-0.1,1);
\draw[line width=0.1cm] (0.1,-1)--(0.1,1);
\draw [blue, fill=white] (-0.3,-0.3) rectangle (0.3,0.3);
\node at (0,0) {$y$};
\end{tikzpicture}}\;,
\end{align}
where $r$ is the spectral radius of $x$ and $\raisebox{-0.5cm}{
\begin{tikzpicture}
\draw [line width=0.1cm] (0,-0.5)--(0,0.5);
\end{tikzpicture}}$ represents $n$ strings.
\end{remark}

\subsection{Uniqueness of PF Eigenvector}
We introduce the irreduciblity for $\mathfrak{F}$-positive elements  in planar algebras, which is analoug to irreducible matrices, and use it to study the uniqueness of PF eigenvector.
\begin{definition}\label{def:irreducible}
Suppose $\mathscr{P}$ is a $C^*$-planar algebra and $x\in\mathscr{P}_{2,\pm}$ is $\mathfrak{F}$-positive. We say $x$ is \textbf{irreducible} on $\mathscr{P}_{1,\pm}$, if for any projection $p\in\mathscr{P}_{1,\pm}$, the inequality  $p\ast \widehat{x}\leq \lambda p$, with $\lambda>0$, ensures that  $p=0$ or $I$.
\end{definition}

In the following, we present some equivalent descriptions for irreducibility of a $\mathfrak{F}$-positive element. Moreover, we also have several
 easy-to-check,  sufficient conditions for irreducibility. 

\begin{theorem}\label{thm:equivirr}
Suppose $\sP$ is a $C^*$-planar algebra and $x\in \sP_{2, +}$ is $\mathfrak{F}$-positive.
Let $\displaystyle B=\bigvee_{k=0}^\infty \mathcal{R}\left(\widehat{x}^{(*k)}\right)$.
Then the following statements satisfy 
$$(1)\Rightarrow (2)\Rightarrow (3)\Leftrightarrow (4) \Leftrightarrow (5) \Leftrightarrow (6) \Leftrightarrow (7) \Leftrightarrow (8) \Leftrightarrow (9)\;.$$
\begin{enumerate}
\item $B=1$ and $\dim\mathscr{P}_{0,+}=1$.
\item $B=1$ and $tr_{0,+}\left(\raisebox{-0.5cm}{
\begin{tikzpicture}
\draw [blue, fill=white] (-0.3,-0.3) rectangle (0.3,0.3);
\node at (0,0) {$p_1$};
\draw (0, 0.3) .. controls +(0, 0.3) and +(0, 0.3) .. (0.6, 0.3);
\draw (0, -0.3) .. controls +(0, -0.3) and +(0, -0.3) .. (0.6, -0.3);
\draw(0.6,0.3)--(0.6,-0.3);
\end{tikzpicture}}\raisebox{-0.5cm}{
\begin{tikzpicture}
\draw [blue, fill=white] (-0.3,-0.3) rectangle (0.3,0.3);
\node at (0,0) {$p_2$};
\draw (0, 0.3) .. controls +(0, 0.3) and +(0, 0.3) .. (0.6, 0.3);
\draw (0, -0.3) .. controls +(0, -0.3) and +(0, -0.3) .. (0.6, -0.3);
\draw(0.6,0.3)--(0.6,-0.3);
\end{tikzpicture}}\right)\neq0$ for any minimal projections $p_1, p_2$ in $\mathscr{P}_{1,+}$.
\item $B(\overline{ p_1}\otimes  p_2)\neq 0$ for any minimal projections $p_1, p_2$ in $\mathscr{P}_{1,+}$.
\item $tr_{2,-}(B(\overline{ p_1}\otimes  p_2))\neq 0$ for any minimal projections $p_1, p_2$ in $\mathscr{P}_{1,+}$.
\item $\mathcal{R}(p*B)=I$ for any nonzero projection $p\in \mathscr{P}_{1,+}$.
\item $x$ is irreducible on $\mathscr{P}_{1,+}$.
\item For any non-zero projection $p\in \mathscr{P}_{1,+}$,
\begin{align*}
p*(e_1+\widehat{x})^{(*(d-1))} >0
\end{align*}
where $d=\dim\mathscr{P}_{1,+}$, $e_1$ is the Jones projection.
\item For any positive elements $y_1, y_2\in \mathscr{P}_{1,+}$ with $\langle y_1, y_2 \rangle =0$, there exists $k_0\in \bN$, $k_0<d$, such that
\begin{align*}
\left\langle y_1*\widehat{x}^{(*k_0)} , y_2\right\rangle >0.
\end{align*}
\item For any projection $q\in\mathscr{P}_{1,-}$, the inequality  $ \widehat{x}\ast q\leq \lambda q$, with $\lambda>0$, ensures that  $q=0$ or $I$.
\end{enumerate}

\end{theorem}
\begin{proof}
$(1)\Rightarrow (2)$
Suppose that $p_1, p_2$ are minimal projections in $\mathscr{P}_{1,+}$ such that $tr_{0}\left(\raisebox{-0.5cm}{
\begin{tikzpicture}
\draw [blue, fill=white] (-0.3,-0.3) rectangle (0.3,0.3);
\node at (0,0) {$p_1$};
\draw (0, 0.3) .. controls +(0, 0.3) and +(0, 0.3) .. (0.6, 0.3);
\draw (0, -0.3) .. controls +(0, -0.3) and +(0, -0.3) .. (0.6, -0.3);
\draw(0.6,0.3)--(0.6,-0.3);
\end{tikzpicture}}\raisebox{-0.5cm}{
\begin{tikzpicture}
\draw [blue, fill=white] (-0.3,-0.3) rectangle (0.3,0.3);
\node at (0,0) {$p_2$};
\draw (0, 0.3) .. controls +(0, 0.3) and +(0, 0.3) .. (0.6, 0.3);
\draw (0, -0.3) .. controls +(0, -0.3) and +(0, -0.3) .. (0.6, -0.3);
\draw(0.6,0.3)--(0.6,-0.3);
\end{tikzpicture}}\right)=0$.
Then
\begin{align*}
    \raisebox{-0.5cm}{
\begin{tikzpicture}
\draw [blue, fill=white] (-0.3,-0.3) rectangle (0.3,0.3);
\node at (0,0) {$p_1$};
\draw (0, 0.3) .. controls +(0, 0.3) and +(0, 0.3) .. (0.6, 0.3);
\draw (0, -0.3) .. controls +(0, -0.3) and +(0, -0.3) .. (0.6, -0.3);
\draw(0.6,0.3)--(0.6,-0.3);
\end{tikzpicture}}\neq 0,\quad  \raisebox{-0.5cm}{
\begin{tikzpicture}
\draw [blue, fill=white] (-0.3,-0.3) rectangle (0.3,0.3);
\node at (0,0) {$p_2$};
\draw (0, 0.3) .. controls +(0, 0.3) and +(0, 0.3) .. (0.6, 0.3);
\draw (0, -0.3) .. controls +(0, -0.3) and +(0, -0.3) .. (0.6, -0.3);
\draw(0.6,0.3)--(0.6,-0.3);
\end{tikzpicture}}\neq 0.
\end{align*}
and
\begin{align*}
\raisebox{-0.5cm}{
\begin{tikzpicture}
\draw [blue, fill=white] (-0.3,-0.3) rectangle (0.3,0.3);
\node at (0,0) {$p_1$};
\draw (0, 0.3) .. controls +(0, 0.3) and +(0, 0.3) .. (0.6, 0.3);
\draw (0, -0.3) .. controls +(0, -0.3) and +(0, -0.3) .. (0.6, -0.3);
\draw(0.6,0.3)--(0.6,-0.3);
\end{tikzpicture}}\raisebox{-0.5cm}{
\begin{tikzpicture}
\draw [blue, fill=white] (-0.3,-0.3) rectangle (0.3,0.3);
\node at (0,0) {$p_2$};
\draw (0, 0.3) .. controls +(0, 0.3) and +(0, 0.3) .. (0.6, 0.3);
\draw (0, -0.3) .. controls +(0, -0.3) and +(0, -0.3) .. (0.6, -0.3);
\draw(0.6,0.3)--(0.6,-0.3);
\end{tikzpicture}}=0.
\end{align*}
This implies that $\raisebox{-0.5cm}{
\begin{tikzpicture}
\draw [blue, fill=white] (-0.3,-0.3) rectangle (0.3,0.3);
\node at (0,0) {$p_1$};
\draw (0, 0.3) .. controls +(0, 0.3) and +(0, 0.3) .. (0.6, 0.3);
\draw (0, -0.3) .. controls +(0, -0.3) and +(0, -0.3) .. (0.6, -0.3);
\draw(0.6,0.3)--(0.6,-0.3);
\end{tikzpicture}},\raisebox{-0.5cm}{
\begin{tikzpicture}
\draw [blue, fill=white] (-0.3,-0.3) rectangle (0.3,0.3);
\node at (0,0) {$p_2$};
\draw (0, 0.3) .. controls +(0, 0.3) and +(0, 0.3) .. (0.6, 0.3);
\draw (0, -0.3) .. controls +(0, -0.3) and +(0, -0.3) .. (0.6, -0.3);
\draw(0.6,0.3)--(0.6,-0.3);
\end{tikzpicture}}$ are linearly independent in $\mathscr{P}_{0,+}$, which leads to the contradiction.

$(2)\Rightarrow (3)\Leftrightarrow (4)$ They are not hard to check.

$(4)\Rightarrow (5)$ Suppose that $p_1, p_2$ are minimal projections in $\mathscr{P}_{1,+}$ such that $p_2(p_1*B)=0$.
Then
\begin{align*}
tr_1(p_1(p_2*B))=tr_2(B(\overline{p_1}\otimes p_2))=0,
\end{align*}
which leads a contradiction.

$(5)\Rightarrow (4)$ Suppose that $p_1, p_2$ are minimal projections in $\mathscr{P}_{1,+}$ such that $tr_{2,-}(B( \overline{p_1}\otimes p_2))=0$.
Then $tr_{1,+}(p_1(p_2*B))=0$, i.e. $p_1(p_2*B)=0$.
We the see that $\mathcal{R}(p_2*B)\leq I-p_1$.
This leads a contradiction.
Hence $(4)$ is true.

$(5)\Rightarrow (6)$
Suppose there exists a nonzero projection $p\in\mathscr{P}_{1,+}$ such that
\begin{align*}
p*\widehat{x}\leq\lambda p
\end{align*}
 for some $\lambda>0$.
By the assumption, we see that there exists $n_0>0$ such that
\begin{align*}
 \sum_{k=0}^{n_0}\widehat{x}^{(*k)} \geq \lambda_1 B, \quad
p*B \geq \lambda_0 I
\end{align*}
for some $\lambda_0, \lambda_1>0$.
Hence
\begin{align*}
\left(\sum_{j=0}^{n_0} \lambda^j \right) p
\geq \sum_{j=0}^{n_0} p * \widehat{x}^{(*j)}
\geq \lambda_1 (p*B)
\geq \lambda_0 \lambda_1 I.
\end{align*}
Finally, we obtain that $p=I$ and $x$ is irreducible on $\mathscr{P}_{1,+}$.

$(6)\Rightarrow(7), (5)$
Suppose that $x$ is irreducible on $\mathscr{P}_{1,+}$.
Let $p$ be a nonzero projection in $\mathscr{P}_{1,+}$ and $y=p+p* \widehat{x} \in\mathscr{P}_{1,+}$.
Then $\mathcal{R}(y)= p \vee\mathcal{R}(p* \widehat{x})$.
If $\mathcal{R}(y)=p$, then $\mathcal{R}(p* \widehat{x}) =p$ and $p=I$ by the irreducibility of $x$ on $\mathscr{P}_{1,+}$.
If $\mathcal{R}(y)\neq p$, then we consider $p*\left(e_{1}+\widehat{x}\right)^{(* (k))}$ for $k \geq 2$ and its range space is increasing.
Note that the dimension of $\mathscr{P}_{1,+}$ is $d$.
 We have that
\begin{align*}
\mathcal{R}\left( p*\left(e_{1}+\widehat{x}\right)^{(* (d-1))}\right)=\mathcal{R}\left( p*\left(e_{1}+\widehat{x}\right)^{(* d)}\right)=I.
\end{align*}
Hence we see that $p*\left(e_{1}+\widehat{x}\right)^{(* (d-1))}>0$ and $\mathcal{R}(p*B)=I$ for any nonzero projection $p$ in $\mathscr{P}_{1,+}$.
Moreover, we have that $y*\left(e_{1}+\widehat{x}\right)^{(* (d-1))}>0$ for any nonzero positive element $y$ in $\mathscr{P}_{1,+}$.

(7)$\Rightarrow$(8)
Note that for any positive elements $y_1, y_2\in \mathscr{P}_{1,+}$ with $\langle y_1, y_2 \rangle =0$, we have
\begin{align*}
\left\langle y_1*\left(e_{1}+\widehat{x}\right)^{(* (d-1))} , y_2 \right\rangle  >0
\end{align*}
By the expansion, we see that there exists $k_0< d$ such that $ \left\langle y_1*\widehat{x}^{(*k_0)}, y_2 \right\rangle >0$.

(8)$\Rightarrow$(6)
Suppose $x$ is not irreducible on $\mathscr{P}_{1,+}$.
Then there exists a nonzero projection $I\neq p\in\mathscr{P}_{1,+}$ such that $p*\widehat{x}\leq\lambda p$ for some $\lambda>0$.
We see that for any $k\in \bN$, $\left\langle p*\widehat{x}^{(*k)},I-p\right\rangle=0$ which leads a contradiction.
Hence $x$ is irreducible on $\mathscr{P}_{1,+}$.

(9) $\Leftrightarrow$ (8) Since (6) $\Leftrightarrow$ (8), we have similarly that (9) is equivalent to the following statement: For any $z_1,z_2\in\mathscr{P}_{1,-}$,  with $\langle z_1, z_2 \rangle =0$, there exists $k_0\in \bN$, $k_0<d$, such that
\begin{align*}
     \left\langle z_1 ,\widehat{x}^{(*k_0)}\ast z_2\right\rangle>0.
\end{align*}
Note that $
  \left\langle z_1 ,\widehat{x}^{(*k_0)}\ast z_2\right\rangle = \left\langle \overline{z_1} \ast\widehat{x}^{(*k_0)} , \overline{z_2}\right\rangle.$ So this statement is equivalent to (8), 
  and (9) $\Leftrightarrow$ (8).
\end{proof}
\begin{remark}
The minimal projections in the statements (2), (3), (4) in Theorem \ref{thm:equivirr} can be
replaced by projections or positive elements in $\mathscr{P}_{1,+}$.
\end{remark}
\begin{remark}
The statement (1) in Theorem \ref{thm:equivirr} is stronger than the irreducibility. 
Let $\mathscr{P}$ be an irreducible subfactor planar algebra with a non-trivial (not indentity) biprojection $q\in \mathscr{P}_{2,-}$.
Take $x$ such that $\mathcal{R}(\widehat{x})\leq q$ in Theorem \ref{thm:equivirr}, then $x$ satisfies 
the statement (3) since $\mathscr{P}_{1,+}=\mathbb{C}I$, so $x$ is irreducible. On the other hand, the statement $(1)$ does not hold since $B\leq q<I$. 
\end{remark}
\begin{remark}
    For  spin model,  the  irreduciblity in Definition  \ref{def:irreducible} is the same is the irreduciblity for matrices.   Moreover, the statements (1) and (2) are equivalent to irreducibility.
\end{remark}
\begin{remark}
Suppose that $\mathscr{P}$ is an irreducible subfactor planar algebra, then  $\displaystyle B=\bigvee_{k=1}^\infty \mathcal{R}\left(\widehat{x}^{(*k)}\right)$ is a biprojection (See Theorem 4.12 in \cite{Liu16}).
\end{remark}

We obtain the uniqueness result under the irreducibility.
\begin{theorem}[\bf Uniqueness of PF eigenvector]\label{thm: unique PF}
Suppose $\mathscr{P}$ is a $C^*$-planar algebra, $x\in\mathscr{P}_{2,\pm}$ is $\mathfrak{F}$-positive and irreducible on $\mathscr{P}_{1,\pm}$.
Then there exists a unique (up to a scalar) strictly positive element $y\in\mathscr{P}_{1,\pm}$ such that
\begin{align}\label{eq:unique left}
    \raisebox{-0.9cm}{
\begin{tikzpicture}[scale=1.5]
\begin{scope}[shift={(-0.6, 0)}]
\draw [blue] (0,0) rectangle (0.35, 0.5);
\node at (0.175, 0.25) {$y$};
\end{scope}
\draw [blue] (0,0) rectangle (0.7, 0.5);
\node at (0.35, 0.25) {$\widehat{x}$};
\draw  (0.5, 0)--(0.5, -0.4) (0.5, 0.5)--(0.5, 0.9);
\draw (0.2, 0.5) .. controls +(0, 0.4) and +(0, 0.4) .. (-0.425, 0.5);
\draw (0.2, 0) .. controls +(0, -0.4) and +(0, -0.4) .. (-0.425, 0);
\end{tikzpicture}}
=r
\raisebox{-0.9cm}{
\begin{tikzpicture}[scale=1.5]
\begin{scope}[shift={(1.5, 0)}]
\draw [blue] (0,0) rectangle (0.35, 0.5);
\node at (0.175, 0.25) {$y$};
\draw  (0.175, 0)--(0.175, -0.4) (0.175, 0.5)--(0.175, 0.9);
\end{scope}
\end{tikzpicture}},
\end{align}
where $r$ is the spectral radius of $x$. Moreover, if there exists a (not necessary positive) element $y'\in \mathscr{P}_{1,\pm}$ such that $y'*\widehat{x}=ry'$, then $y'$ is a multiple of $y$.
\end{theorem}
\begin{proof}
 From Proposition \ref{prop:PF for planar algebra}, there exists a nonzero positive element $y\in\mathscr{P}_{1,\pm}$ such that $y\ast \widehat{x}=ry$. Thus $\mathcal{R}(y)\ast \widehat{x}\leq \lambda\mathcal{R}(y)$ for some $\lambda>0$. Since $x$ is irreducible on $\mathscr{P}_{1,\pm}$, so $\mathcal{R}(y)=I$ and $y$ is invertible.
If $y_1$ is another strictly positive element in $\mathscr{P}_{1,\pm}$ such that $y_1\ast \widehat{x}=ry_1$. We could take $y_2=y-ty_1\geq0$ for some $t>0$ such that $0<\mathcal{R}(y_2)<I$ by a continuity argument. Note that $y_2\ast \widehat{x}=ry_2$, so $\mathcal{R}(y_2)=I$, which leads to a contradiction.

Since $\widehat{x}$ is positive, we may assume that $y'$ is self-adjoint. There exists $\lambda>0$ such that $y'+\lambda y>0$.
Thus, $y'+\lambda y$ is a multiple of $y$.
Therefore, $y'$ is a multiple of $y$.
\end{proof}

\begin{corollary}
Let $\mathscr{P}$ be a  $C^*$-planar algebra. Suppose that $x\in\mathscr{P}_{2,\pm}$ is $\mathfrak{F}$-positive and  irreducible on $\mathscr{P}_{1,\pm}$.
If $\displaystyle \bigvee_{j=0}^\infty \mathcal{R}(\widehat{x}^{(*j)})=1$ and $\dim \mathscr{P}_{1,\pm}=1$,
then there exists a unique strictly positive element $y \in\mathscr{P}_{1,\pm}$ such that $y*\widehat{x}=ry$.
\end{corollary}
\begin{proof}
It follows from Theorems \ref{thm:equivirr} and \ref{thm: unique PF}.
\end{proof}
We now consider the space of PF eigenvectors.
\begin{notation}\label{notation:eigenspace}
Suppose $\mathscr{P}$ is a $C^*$-planar algebra and $x\in\mathscr{P}_{2,\pm}$ is $\mathfrak{F}$-positive.
Let
\begin{align*}
\mathcal{E}=\{y\in \mathscr{P}_{1,\pm}: y*\widehat{x}=ry\} \subseteq \mathscr{P}_{1,\pm}
\end{align*}
be the eigenspace of $r$ for $x$ in $\mathscr{P}_{1,\pm}$.
We denote by $\mathcal{E}^+$ the set of all positive elements in $\mathcal{E}$ and $(\mathcal{E}^+)_1$ the set of all trace-one elements in $\mathcal{E}^+$. Proposition \ref{prop:PF for planar algebra} indicates that 
$(\mathcal{E}^+)_1$ is a nonempty compact convex set.
Let
\begin{align*}
\mathscr{Q}=\{\mathcal{R}(y):\ y\in \mathcal{E}^+ \}
\end{align*}
be the set of all range projections of positive eigenvectors.
For any $p,q\in\mathscr{Q}$, there exists $y_1,y_2\in\mathcal{E}^+$ such that $\mathcal{R}(y_1)=p$
and $\mathcal{R}(y_2)=1$. So $\mathcal{R}(y_1+y_2)=p\vee q$ and $p\vee q\in\mathscr{Q}$.
Let $p_{\max}$ be the maximal projection in $\mathscr{Q}$. Then $p_{\max}=\bigvee_{p\in\mathscr{Q}}p$.
\end{notation}

\begin{proposition}\label{prop: minimal projection and extremal point 1} 
Suppose $\mathscr{P}$ is a $C^*$-planar algebra and $x\in\mathscr{P}_{2,\pm}$ is $\mathfrak{F}$-positive.
Then  $y\in(\mathcal{E}^+)_1$ is an extreme point
if and only if $\mathcal{R}(y)$ is a minimal projection in
$\mathscr{Q}$. Moreover, if there exist $y_1,y_2 \in (\mathcal{E}^+)_1$ such that $y_1$ is an extreme point and $\mathcal{R}(y_1)=\mathcal{R}(y_2)$, then $y_1=y_2$.
\end{proposition}
\begin{proof}
Suppose   $y$ is an extremal point. If $\mathcal{R}(y)$ is not minimal, then there exists $z\in (\mathcal{E}^+)_1$ such that $\mathcal{R}(z)<\mathcal{R}(y)$. Take $\omega=y-tz\geq0$ for some $0<t<1$. Then $y=(1-t)\dfrac{\omega}{1-t}+t z$, which is a contradiction.

On the other hand, if $\mathcal{R}(y)$ is minimal but $y$ is not extremal. Then $y$ can be expressed as a convex combination of two different elements, i.e., $y=\lambda y'+(1-\lambda)y''$, $0<\lambda<1$,  $y',y''\in (\mathcal{E}^+)_1$. It is easy to see $\mathcal{R}(y')\leq\mathcal{R}(y)$. Thus $\mathcal{R}(y')=\mathcal{R}(y)$ from the minimality of $\mathcal{R}(y)$. We can choose $0<t<1$ such that $y-ty'\geq0$ and $0<\mathcal{R}(y-ty')<\mathcal{R}(y)$ by a continuity argument. Note that $\dfrac{y-ty'}{1-t}\in(\mathcal{E}^+)_1$. So we find another eigenvector with a smaller support, which leads to a contradiction.
\end{proof}
In the end of this section, we present the Collatz-Wielandt formula for planar algebras.
\begin{theorem}[\bf Collatz-Wielandt Formula]\label{thm:cwF}
Suppose $\mathscr{P}$ is a $C^*$-planar algebra, $x\in\mathscr{P}_{2,\pm}$ is $\mathfrak{F}$-positive. Let $r$ be the spectral radius of $x$.
We have the following statements are true:
\begin{enumerate}
\item If there exists a non-zero positive element $z \in\sP_{1, \mp}$ such that
\[ \widehat{x}*z \leq r z \quad \text{and}\quad \mathcal{R}(z)\leq \overline{p_{\max}},\]
then $ \widehat{x}*z = r z$.
\item  If there exists a non-zero positive element $z \in\sP_{1, \mp}$ such that
\[  r z\leq \widehat{x}*z  \quad \text{and}\quad \mathcal{R}(\widehat{x}*z )\leq\overline{p_{\max}},\]
then $ \widehat{x}*z = r z$.
\item If there exists a non-zero positive element $z\in\sP_{1, \mp}$ such that $ \widehat{x}*z = \widetilde{r} z$ and $\mathcal{R}(z)\leq \overline{p_{\max}}$, then $\widetilde{r}=r$.
\end{enumerate}
\end{theorem}
\begin{proof}
$(1)$ Let $y\in\mathscr{P}_{1,\pm}$ be the positive element such that $y*\widehat{x}=ry$ and $\mathcal{R}(y)=p_{\max}$.
Then we have
\begin{align*}
ry*z = (y*\widehat{x})*z=y*(\widehat{x}*z) \leq  r y*z.
\end{align*}
The last inequality is the quantum Schur product theorem.
Hence $y*(rz-\widehat{x}*z) =0$ and $p_{\max}*(rz-\widehat{x}*z) =0$.
We obtain that
\begin{align*}
\raisebox{-0.8cm}{\begin{tikzpicture}[scale=1.2]
\draw [blue, fill=white] (-0.85,0) rectangle (0.75, 0.5);
\node at (-0.05, 0.25) {$rz-\widehat{x}*z$};
\draw (-0.05, 0.5) .. controls +(0, 0.5) and +(0, 0.5) .. (-1.45, 0.5)--(-1.45, 0).. controls +(0, -0.5) and +(0, -0.5) ..(-0.05, 0);
\begin{scope}[shift={(-1.8, 0)}]
    \draw [blue,fill=white] (0,0) rectangle (0.7, 0.5);
\node at (0.35, 0.25) {$p_{\max}$};
\end{scope}
\end{tikzpicture}}=0.
\end{align*}
Hence $tr_{1,\mp}(rz-\widehat{x}*z)$=0. By the faithful of the trace, we obtain $\widehat{x}*z=rz$.
(2) The argument is similar to (1).

$(3)$ Note that $\widetilde{r}\in \sigma(x)$ and $\widetilde{r}\leq r$.
By the statement $(1)$, we see that $\widetilde{r}=r$.
\end{proof}

\begin{corollary}
Suppose $\mathscr{P}$ is a $C^*$-planar algebra, $x\in\mathscr{P}_{2,\pm}$ is $\mathfrak{F}$-positive and irreducible. Let $r$ be the spectral radius of $x$.
We have the following statements are true:
\begin{enumerate}
\item If there exists a non-zero positive element $z \in\sP_{1, \mp}$ such that
$\widehat{x}*z \leq r z$ or $\widehat{x}*z \geq rz$,
then $ \widehat{x}*z = r z$.
\item If there exists a non-zero positive element $z\in\sP_{1, \mp}$ such that $ \widehat{x}*z = \widetilde{r} z$, then $\widetilde{r}=r$.
\end{enumerate}
\end{corollary}
\begin{proof}
By Theorem \ref{thm: unique PF}, there exists a unique strictly positive element $y\in\mathscr{P}_{1,\pm}$ such that $y*\widehat{x}=ry$. So we have $p_{\max}=I$. Then by Theorem \ref{thm:cwF}, we could obtain the conclusion.
\end{proof}
\begin{remark}
 Evans and H{\o}egh-Krohn proved the Collatz-Wielandt formula for irreducible positive maps on finite-dimensional $C^*$-algebras (See Theorems 2.3 and 2.4  in  \cite{ChoEff77}). We give the  Collatz-Wielandt formula on
 a weaker condition for $x$ in Theorem \ref{thm:cwF}.
\end{remark}
\begin{remark}
    Suppose  $\mathscr{P}_{1,\pm}$ is abelian and $\{p_i\}_{i=1}^n$ are orthgonal minimal projections with $\sum_{i=1}^n p_i=I$, and $x\in\mathscr{P}_{2,\pm}$ is $\mathfrak{F}$-positive.
    We have another form of the Collatz-Wielandt formula: Let $z\in \mathscr{P}_{1,\mp}$ be a strictly positive element. Then
    \begin{align}\label{eq:CWF abelian}
    \min_{1\leq i\leq n}\frac{tr_{1,\mp}(\overline{p_i}[\widehat{x}\ast z])}{tr_{1,\mp}(\overline{p_i} z)}\leq   
    r(x)\leq \max_{1\leq i\leq n}\frac{tr_{1,\mp}(\overline{p_i}[\widehat{x}\ast z])}{tr_{1,\mp}(\overline{p_i} z)}.
    \end{align}
This could recover the Collatz-Wielandt formula for (not-necessarily irreducible) positive matrices. 
Suppose $\Gamma$ is a bipartite graph with simple edges. Let $\mathscr{P}(\Gamma)$ be the
 bipartite graph $C^*$-planar algebra of $\Gamma$. Then $\mathscr{P}(\Gamma)_{1,\pm}$ is abelian and Inequality \eqref{eq:CWF abelian} holds.
\end{remark}

\section{PF Theorem on Ordered Banach Spaces}\label{sec:PFT for Banach space}
For a  positive map on an (infinite-dimensional) ordered Banach space, there exists an 
approximate PF eigenvector with respective to the spectral radius. The space of  PF eigenvectors is at most one-dimensional if the positive map has ergodicity property. We give a self-contained proof of these results.
\begin{definition}\label{def:ordered banach space}
An \textbf{ordered Banach space} $\mathcal{X}$ is a Banach space with a closed positive cone $\mathcal{C}$  such that $\mathcal{C}\cap -\mathcal{C}=0$ and $\mathcal{C}+\mathcal{C}\subseteq \mathcal{C}$.
For any $x,y\in \mathcal{X}$, we define the order relation $x\leq y$ if $y-x\in\mathcal{C}$. 
Moreover, the positive cone $\mathcal{C}$ also satisfies the following properties:
\begin{enumerate}
\item If $0\leq x\leq y$, then $\|x\|\leq c_{\mathcal{X}}\|y\|$, where $c_{\mathcal{X}}$ is a positive constant independent of $x,y$;
\item Each $x\in \mathcal{X}$ is a linear sum of four positive elements with norm bounded by $d_{\mathcal{X}}\|x\|$, where $d_{\mathcal{X}}$ is a positive constant independent of $x$.
\end{enumerate}
\end{definition}
\begin{remark}
Property (1) is usually called normality. Property (2) holds if only if $\mathcal{C}$ is a generating cone, i.e., $\mathcal{X}=\mathcal{C}-\mathcal{C}$; see for instance \cite[Theorem 2.37 (1) and (3)]{AT07}. 
\end{remark}

\begin{remark}
    Suppose  $\mathcal{M}$ is a von Neumann algebra with a normal semifinite faithful weight. Let $L^p(\mathcal{M})$ be the noncommutative spatial  $L^p$ space \cite{Hil81}, $1\leq p\leq \infty$. Define $\mathcal{C}:=\{x\in L^p(\mathcal{M}): x\geq 0\}$. Then $(L^p(\mathcal{M}),\mathcal{C})$ is an ordered Banach space.
\end{remark}

\begin{definition}\label{def:positive map}
A bounded linear map $T$ between two ordered Banach spaces $(\mathcal{X}_1,\mathcal{C}_1)$ and  $(\mathcal{X}_2,\mathcal{C}_2)$ is \textbf{positive} if $T\mathcal{C}_1\subseteq\mathcal{C}_2$. 
\end{definition}

\begin{lemma}\label{lem:bound on the cone}
Suppose $T$ is a linear map from the ordered Banach space $(\mathcal{X},\mathcal{C})$
to a Banach space. 
We have that  $T$ is bounded if and only if $\displaystyle \sup_{x\in\mathcal{C},x\neq0}\dfrac{\|Tx\|}{\|x\|}<\infty$, i.e., $T$ is bounded on the cone $\mathcal{C}$.
\end{lemma}
\begin{proof}
    It could be proved by using the condition (2) in Definition \ref{def:ordered banach space}. 
\end{proof}

\begin{lemma}\label{lem:converge}
Let $(\mathcal{X},\mathcal{C})$ be an ordered Banach space. 
Suppose $\{x_n\}_{n=1}^\infty\subseteq\mathcal{C}$ is a sequence of positive elements such that $\displaystyle \sum_{n=1}^\infty x_n$ converges. Then for any  sequence of complex numbers
$\{w_n\}_{n=1}^\infty$ with modulus uniformly  bounded by a positive number $c$, we have $\displaystyle \sum_{n=1}^\infty w_n x_n$ converges and 
\begin{align*}
\left\|\sum_{n=1}^\infty w_n x_n \right\|\leq4cc_{\mathcal{X}} \left\|\sum_{n=1}^\infty x_n \right\|.
\end{align*}
\end{lemma}
\begin{proof}
For any $1\leq N<M$, we have 
\begin{align*}
\bigg\|\sum_{n=N}^M w_nx_n \bigg\|
=& \bigg\|\sum_{n=N}^M \Re(w_n)x_n+i\sum_{n=N}^M \Im(w_n)x_n\bigg\|\\
\leq & \bigg\|\sum_{n=N}^M \Re(w_n)x_n\bigg\|+\bigg\|\sum_{n=N}^M \Im(w_n)x_n\bigg\| \\
\leq & 4cc_{\mathcal{X}}\bigg\|\sum_{n=N}^M x_n\bigg\|.
\end{align*}
Thus $\displaystyle \sum_{n=1}^\infty w_n x_n$ converges. 
Let $N=1$, and $M\rightarrow\infty$, we have  
\begin{align*}
\left\|\sum_{n=1}^\infty w_n x_n \right\|\leq 4cc_{\mathcal{X}} \left\|\sum_{n=1}^\infty x_n\right\|.
\end{align*}
\end{proof}

\begin{theorem}[\bf Approximate Existence]\label{thm:existobs}
Let $T$ be a positive linear map on an ordered Banach spaces $(X,\mathcal{C})$. 
Then $r(T) \in \sigma(T)$ and there exists a sequence of positive elements $\displaystyle \{x_n\}_{n=1}^\infty\subseteq\mathcal{C}$ with $\|x_n\|=1$ such that
\begin{align*}
\lim_{n\rightarrow\infty}\|Tx_n-r(T)x_n \|=0,
\end{align*}
where $r(T)$ is the spectral radius of $T$.
\end{theorem}
\begin{proof}
Let $r=r(T)$ and 
\begin{align*}
R(\lambda, T)=(\lambda -T)^{-1}, \quad  \lambda\in \mathbb{C}\backslash \sigma(T),
\end{align*}
be the resolvent of $T$.
Then $R(\lambda, T)$ is analytic in $\mathbb{C}\backslash \sigma(T)$.
For any $|\lambda| > r$, we have that
\begin{align*}
R(\lambda, T)=\sum_{j=0}^\infty \frac{T^j}{\lambda^{j+1}}.
\end{align*}
If $r\notin \sigma(T)$, then $R(r, T)=\displaystyle \lim_{\lambda\to r^+}R(\lambda, T)$ is bounded. 
For any $x\in\mathcal{C}$, we have 
\begin{align*}
R(r, T)(x)=\lim_{\lambda \to r^+}R(\lambda, T)(x)=\sum_{j=0}^\infty \frac{T^j(x)}{r^{j+1}} \in \mathcal{X}.
\end{align*}
By Lemma \ref{lem:converge}, we have $\displaystyle \sum_{j=0}^\infty \dfrac{T^j(x)}{w^{j+1}}$ converges for any $|w|=r$ and
\begin{align*}
\left\|\sum_{j=0}^\infty \frac{T^j(x)}{w^{j+1}}\right\|
\leq \frac{4c_{\mathcal{X}}}{r}\left \| \sum_{j=0}^\infty\frac{T^j(x)}{r^{j+1}}  \right\|
\leq\frac{4c_{\mathcal{X}}\|R(r, T)\|}{r}\|x\|.
\end{align*}
Then we see  that $\displaystyle R(w, T)=\lim_{\lambda \to w}R(\lambda, T)$ is bounded on the cone, hence is bounded by Lemma \ref{lem:bound on the cone}. Thus
 $w\notin \sigma(T)$, which leads to a contradiction. 
Hence $r\in \sigma(T)$.

Let 
\begin{align*}
T_n=\left(r+\frac{1}{n}-T\right)^{-1}, \quad n=1,2,\ldots.
\end{align*}
Then $\{T_n\}_{n\in\bN}$ is an unbounded sequence  of positive linear maps. 
Thus there exists a sequence of positive elements $y_n\in\mathcal{C}$ such that $\|T_n y_n\|=1$ and $\displaystyle \lim_{n\to\infty}y_n=0$. 
Now take $x_n=T_n y_n$, we have $\displaystyle \lim_{n\rightarrow\infty}\|Tx_n-rx_n \|=0$.
\end{proof}

\begin{remark}
Gl\"{u}ck and Mironchenko proved the approximate existence of PF eigenvector for a positive map on an ordered Banach space in \cite[Lemma 3.5]{GM21}.
\end{remark}

We now assume that there exits an element $I\in\mathcal{C}$ such that $x\leq e_\mathcal{X}\|x\|I$ for any $x\in\mathcal{C}$, where $e_\mathcal{X}$ is a positive constant independent of $x$. We call $x\in\mathcal{C}$ is strictly positive, written $x>0$, if there exits $\lambda>0$ such that $\lambda I\leq x$.

\begin{definition}
Let $(X,\mathcal{C})$ be an ordered Banach space. Let $T$ be a positive linear map on $X$. We say $T$ is \textbf{ergodic} if for any non-zero $x\in\mathcal{C}$, there exists $N>0$ such that $\sum_{j=1}^N T^j(x)$ is strictly positive.
\end{definition}

\begin{remark}
Let $\mathcal{X}^*$ be the dual Banach space of $\mathcal{X}$. Define the dual cone $\mathcal{C}^*$ of $\mathcal{C}$ in $\mathcal{X}^*$ by
\begin{align*}
   \mathcal{C}^*:=\{x^*\in\mathcal{X}^*:\ \langle x,x^*\rangle\geq0,\ \forall x\in\mathcal{C}\}, 
\end{align*}
where the inner product means the linear functional $x^*$ acting on $x$. From Gl\"{u}ck and Weber's characterizations of strictly positive elements in \cite[Proposition 2.11 and Corollary 2.8]{GW20}, where they called interior points, we know that $x\in\mathcal{C}$ is strictly positive if and only if
$\langle x,x^*\rangle>0$ for all $x^*\in\mathcal{C}^*$, $x^*\neq 0$. It follows that the the definition of ergodicity property is equivalent to Sawashima's definition of 
irreducibility: For every couple $x\in \mathcal{C},\ x\neq 0$, $x^*\in\mathcal{C},\ x^*\neq 0$, there exists a positive integer $n=n(x,x^*)$ such that $\langle T^nx,x^*\rangle>0$.
\end{remark}

\begin{theorem}[\bf Uniqueness]\label{thm:PF unique ordered Banach space}
Let $(X,\mathcal{C})$ be an ordered Banach space and $T$ is a positive linear map on $X$. If $T$ is ergodic on $X$, then there exists at most one $0\neq x\in\mathcal{C}$ such that $Tx=r(T)x$.
\end{theorem}
\begin{proof}
Suppose there exists $0\neq x\in\mathcal{C}$ such that $Tx=r(T)x$.
Since $T$ is ergodic,  $\displaystyle \sum_{j=1}^N T(x)$ is strictly positive, so $x>0$. 
Suppose there is $y\in\mathcal{C}$ such that $Ty=r(T)y$. 
We consider the following set
\begin{align*}
S:=\{t>0:\ x-ty>0\}.
\end{align*}
Since  $x\leq e_\mathcal{X}\|x\|I$ and  $y\leq e_\mathcal{X}\|y\|I$, 
then $S$ is a non-empty bounded open set. 
Take $t_0=\sup_{t\in S}t$, 
then $x-t_0 y$ is positive but is not strictly positive. 
Note that $T(x-t_0 y)=r(T)(x-t_0 y)$,
 if $x-t_0 y$ is non-zero then it is strictly positive, which is a contradiction. 
Therefore, $x-t_0 y=0$.
\end{proof}
\begin{remark}
Theorem \ref{thm:PF unique ordered Banach space} is another way to formulate the uniqueness result in Theorem 11.1(d) on page 111 in \cite{KLS89}.   
\end{remark}

\begin{remark}
      Krein and Rutman proved the PF theorem for positive compact linear maps on ordered Banach spaces (See e.g. Theorem 1.23 in \cite{BR04}). 
        We prove the PF theorem for positive (not-necessarily compact) linear maps on ordered Banach space with  two extra conditions for positive cones. They proved the uniqueness of PF eigenvector if  a compact linear map $T$ is strongly positive, which means that $T(\mathcal{C}\setminus\{0\})\subseteq \text{Int}(\mathcal{C})$. We prove the uniqueness if a 
        positive linear map $T$ is ergodic. The ergodicity of $T$ means that there exists $N>0$ such that $\sum_{j=1}^N T^j$ is strong positive. So the ergodic condition
        is weaker than the strong positivity.
\end{remark}

\section{Structures of the PF Eigenspace}\label{sec:structure}

In \S \ref{sec:PF theorem} we found that 
the support of the PF eigenvector of an irreducible $\mathfrak{F}$-positive element is identity. 
In this section, we study the eigenspace of PF eigenvectors of (non-irreducible) $\mathfrak{F}$-positive elements.

We show that an $\mathfrak{F}$-positive  element yields a power series that defines  a maximally-supported PF eigenvector $\zeta$. We obtain from this PF eigenvector the PF eigenspace $\mathcal{E}$  which has a $C^*$-algebraic structure.  We introduce exchange algebras $\mathcal{A}$ and $\mathcal{B}=\mathcal{A}^*$ and $\mathcal{C}=\mathcal{A}\cap \mathcal{B}$ for an $\mathfrak{F}$-positive element $x$. We prove that $\zeta$ commutes with $\mathcal{C}$ and is independent of $\mathcal{C}$ and $\zeta\mathcal{C}\subseteq\mathcal{E}$.  Let $\mathcal{C}_\zeta$ be the $C^*$-algebra $\mathcal{C}$ with respect to the $\mathfrak{F}$-positive element $(\zeta^{1/2}\otimes\overline{\zeta^{1/2}})x(\zeta^{-1/2}\otimes\overline{\zeta^{-1/2}})$. We prove that the space of positive PF eigenvectors $\mathcal{E}^+=\zeta^{1/2}\mathcal{C}_\zeta^+\zeta^{1/2}$ under a certain condition.
Let $\mathcal{A}_p$ and $\mathcal{B}_p$ be the exchange algebras $\mathcal{A}$ and $\mathcal{B}$
for the $\mathfrak{F}$-positive element $(p\otimes\overline{p})x(p\otimes\overline{p})$. We prove that $\mathcal{E}^+=(\mathcal{A}_p\zeta)^+=(\zeta\mathcal{B}_p)^+$ under the same condition.
These characterizations of PF eigenspaces are true for completely positive trace-preserving maps.
We elaborate the connection between this multiplicative structure and the multiplication of the $C^*$ algebra of density matrices. 
Based on it, we explore a PF theoretical approach to quantum error correction in the next section.

\subsection{The Maximal-Support PF Eigenvector}\label{subsec:maximal support}
Suppose $\mathscr{P}$ is a $C^*$-planar algebra and $x\in\mathscr{P}_{2,\pm}$ is $\mathfrak{F}$-positive with the spectral radius $r$. We recall that
 the eigenvector space of $x$ corresponding to  $r$ (the PF eigenspace) is denoted by
\begin{align*}
\mathcal{E}=\{y\in \mathscr{P}_{1,\pm}: y*\widehat{x}=ry\} \subseteq \mathscr{P}_{1,\pm}.
\end{align*}
Let $k$ be the order of the resolvent $(z-x)^{-1}$, $z\in\mathbb{C}$, at the pole $r$.
From  Equation \eqref{eq:perron operator} in the proof of Proposition \ref{prop:PF for planar algebra}, the following limit exists:
\begin{align}\label{eq:perron operator sec}
y:=\lim_{\mathbb{R}\ni z\to r^+}(z-r)^k\sum_{j=0}^\infty \dfrac{x^j }{z^{j+1}}.
\end{align}
Moreover, $y\in\mathscr{P}_{2,\pm}$ is a $\mathfrak{F}$-positive element such that  $xy=yx=ry$.
 Recall that
\begin{align*}
\mathscr{Q}=\{\mathcal{R}(y):\ y\in \mathcal{E}^+ \}
\end{align*}
is the set of all range projections of positive eigenvectors and $p=\bigvee_{q\in\mathscr{Q}} q$ is the maximal projection in $\mathscr{Q}$. (Here we use $p$ instead of $p_{\max}$ in Notation \ref{notation:eigenspace} for simplicity.) We define
\begin{align}\label{eq:cutting down positive definite}
\raisebox{-1.6cm}{
\begin{tikzpicture}[yscale=1.5, xscale=2.2]
\draw (0.35,- 0.9)--(0.35, 1.4);
\draw (0.15, -0.9)--(0.15, 1.4) ;
\draw [blue, fill=white] (0,0) rectangle (0.5, 0.5);
\node at (0.25, 0.25) {$x_p$};
\end{tikzpicture}}:=
\raisebox{-1.6cm}{
\begin{tikzpicture}[yscale=1.5, xscale=2.2]
\draw [blue, fill=white] (0,0) rectangle (0.5, 0.5);
\node at (0.25, 0.25) {$x$};
\draw (0.35, 0.5)--(0.5, 0.75)   (0.35, 0)--(0.5, -0.25);
\draw (0.15, 0.5)--(0, 0.75)  (0.15, 0)--(0, -0.25);
\begin{scope}[shift={(-0.2,0.75)}]
\draw [blue, fill=white] (0,0) rectangle (0.4, 0.4);
\node at (0.2, 0.2) {$p$};
\draw (0.2, 0.4)--(0.2, 0.65);
\end{scope}
\begin{scope}[shift={(0.3,0.75)}]
\draw [blue, fill=white] (0,0) rectangle (0.4, 0.4);
\node at (0.2, 0.2) {$\overline{p}$};
\draw (0.2, 0.4)--(0.2, 0.65);
\end{scope}
\begin{scope}[shift={(-0.2,-0.65)}]
\draw [blue, fill=white] (0,0) rectangle (0.4, 0.4);
\node at (0.2, 0.2) {$p$};
\draw (0.2,0)--(0.2,- 0.25);
\end{scope}
\begin{scope}[shift={(0.3,-0.65)}]
\draw [blue, fill=white] (0,0) rectangle (0.4, 0.4);
\node at (0.2, 0.2) {$\overline{p}$};
\draw (0.2,0)--(0.2,- 0.25);
\end{scope}
\end{tikzpicture}}.
\end{align}
We next prove that the spectral radius of $x_p$ is $r$ and
the order $k_p$  of the resolvent $(z-x_p)^{-1}$, $z\in\mathbb{C}$, at the pole $r$, is $1$.
\begin{lemma}
The following equation holds:
\begin{align}\label{eq:cut down2}
\raisebox{-1.6cm}{
\begin{tikzpicture}[yscale=1.5, xscale=2.2]
\draw [blue, fill=white] (0,0) rectangle (0.5, 0.5);
\node at (0.25, 0.25) {$\widehat{x}$};
\draw (0.35, 0.5)--(0.5, 0.75)   (0.35, 0)--(0.5, -0.25);
\draw (0.15, 0.5)--(0, 0.75)  (0.15, 0)--(0, -0.25);
\begin{scope}[shift={(-0.2,0.75)}]
\draw [blue, fill=white] (0,0) rectangle (0.4, 0.4);
\node at (0.2, 0.2) {$\overline{p}$};
\draw (0.2, 0.4)--(0.2, 0.65);
\end{scope}
\begin{scope}[shift={(0.3,0.75)}]
\draw [blue, fill=white] (0,0) rectangle (0.4, 0.4);
\node at (0.2, 0.2) {$p$};
\draw (0.2, 0.4)--(0.2, 0.65);
\end{scope}
\begin{scope}[shift={(-0.2,-0.65)}]
\draw [blue, fill=white] (0,0) rectangle (0.4, 0.4);
\node at (0.2, 0.2) {$\overline{p}$};
\draw (0.2,0)--(0.2,- 0.25);
\end{scope}
\begin{scope}[shift={(0.3,-0.65)}]
\draw [blue, fill=white] (0,0) rectangle (0.4, 0.4);
\node at (0.2, 0.2) {$p$};
\draw (0.2,0)--(0.2,- 0.25);
\end{scope}
\end{tikzpicture}}=
\raisebox{-1.6cm}{
\begin{tikzpicture}[yscale=1.5, xscale=2.2]
\draw [blue, fill=white] (0,0) rectangle (0.5, 0.5);
\node at (0.25, 0.25) {$\widehat{x}$};
\draw (0.35, 0.5)--(0.5, 0.75)   (0.35, 0)--(0.5, -0.25);
\draw (0.15, 0.5)--(0, 0.75)  (0.15, 0)--(0, -0.25);
\begin{scope}[shift={(-0.2,0.75)}]
\draw [blue, fill=white] (0,0) rectangle (0.4, 0.4);
\node at (0.2, 0.2) {$\overline{p}$};
\draw (0.2, 0.4)--(0.2, 0.65);
\end{scope}
\begin{scope}[shift={(0.3,0.75)}]
\draw (0.2, 0)--(0.2, 0.65);
\end{scope}
\begin{scope}[shift={(-0.2,-0.65)}]
\draw [blue, fill=white] (0,0) rectangle (0.4, 0.4);
\node at (0.2, 0.2) {$\overline{p}$};
\draw (0.2,0)--(0.2,- 0.25);
\end{scope}
\begin{scope}[shift={(0.3,-0.65)}]
\draw (0.2,0.4)--(0.2,- 0.25);
\end{scope}
\end{tikzpicture}}\;.
\end{align}    
\end{lemma}
\begin{proof}
Let $y\in\mathscr{P}_{1,\pm}$ the PF eigenvector of $x$ with $\mathcal{R}(y)=p$. So $(I-p)(y\ast \widehat{x})(I-p)=0$, which is equivalent to $(I-p)(p\ast \widehat{x})(I-p)=0$.
This implies
   \begin{align*}
\raisebox{-1.6cm}{
\begin{tikzpicture}[yscale=1.5, xscale=2.2]
\draw [blue, fill=white] (0,0) rectangle (0.5, 0.5);
\node at (0.25, 0.25) {$\widehat{x}$};
\draw (0.35, 0.5)--(0.5, 0.75)   (0.35, 0)--(0.5, -0.25);
\draw (0.15, 0.5)--(0, 0.75)  (0.15, 0)--(0, -0.25);
\begin{scope}[shift={(-0.2,0.75)}]
\draw [blue, fill=white] (0,0) rectangle (0.4, 0.4);
\node at (0.2, 0.2) {$\overline{p}$};
\draw (0.2, 0.4)--(0.2, 0.65);
\end{scope}
\begin{scope}[shift={(0.3,0.75)}]
\draw [blue, fill=white] (0,0) rectangle (0.45, 0.4);
\node at (0.2, 0.2) {$I-p$};
\draw (0.2, 0.4)--(0.2, 0.65);
\end{scope}
\begin{scope}[shift={(-0.2,-0.65)}]
\draw [blue, fill=white] (0,0) rectangle (0.4, 0.4);
\node at (0.2, 0.2) {$\overline{p}$};
\draw (0.2,0)--(0.2,- 0.25);
\end{scope}
\begin{scope}[shift={(0.3,-0.65)}]
\draw [blue, fill=white] (0,0) rectangle (0.45, 0.4);
\node at (0.2, 0.2) {$I-p$};
\draw (0.2,0)--(0.2,- 0.25);
\end{scope}
\end{tikzpicture}}=0\;.
\end{align*} 
By cutting the above picture  horizontally, we have
$
    \raisebox{-1.2cm}{
\begin{tikzpicture}[yscale=1.5, xscale=2.2]
\draw [blue, fill=white] (0,0) rectangle (0.5, 0.5);
\node at (0.25, 0.25) {$\widehat{x}^{1/2}$};
\draw (0.35, 0.5)--(0.5, 0.75)   (0.35, 0)--(0.35, -0.35);
\draw (0.15, 0.5)--(0, 0.75)  (0.15, 0)--(0.15, -0.35);
\begin{scope}[shift={(-0.2,0.75)}]
\draw [blue, fill=white] (0,0) rectangle (0.4, 0.4);
\node at (0.2, 0.2) {$\overline{p}$};
\draw (0.2, 0.4)--(0.2, 0.65);
\end{scope}
\begin{scope}[shift={(0.3,0.75)}]
\draw [blue, fill=white] (0,0) rectangle (0.45, 0.4);
\node at (0.2, 0.2) {$I-p$};
\draw (0.2, 0.4)--(0.2, 0.65);
\end{scope}
\end{tikzpicture}}=0\;.
$
Thus
$
    \raisebox{-1.2cm}{
\begin{tikzpicture}[yscale=1.5, xscale=2.2]
\draw [blue, fill=white] (0,0) rectangle (0.5, 0.5);
\node at (0.25, 0.25) {$\widehat{x}^{1/2}$};
\draw (0.35, 0.5)--(0.5, 0.75)   (0.35, 0)--(0.35, -0.35);
\draw (0.15, 0.5)--(0, 0.75)  (0.15, 0)--(0.15, -0.35);
\begin{scope}[shift={(-0.2,0.75)}]
\draw [blue, fill=white] (0,0) rectangle (0.4, 0.4);
\node at (0.2, 0.2) {$\overline{p}$};
\draw (0.2, 0.4)--(0.2, 0.65);
\end{scope}
\begin{scope}[shift={(0.3,0.75)}]
\draw [blue, fill=white] (0,0) rectangle (0.4, 0.4);
\node at (0.2, 0.2) {$p$};
\draw (0.2, 0.4)--(0.2, 0.65);
\end{scope}
\end{tikzpicture}}=\raisebox{-1.2cm}{
\begin{tikzpicture}[yscale=1.5, xscale=2.2]
\draw [blue, fill=white] (0,0) rectangle (0.5, 0.5);
\node at (0.25, 0.25) {$\widehat{x}^{1/2}$};
\draw (0.35, 0.5)--(0.5, 0.75)   (0.35, 0)--(0.35, -0.35);
\draw (0.15, 0.5)--(0, 0.75)  (0.15, 0)--(0.15, -0.35);
\begin{scope}[shift={(-0.2,0.75)}]
\draw [blue, fill=white] (0,0) rectangle (0.4, 0.4);
\node at (0.2, 0.2) {$\overline{p}$};
\draw (0.2, 0.4)--(0.2, 0.65);
\end{scope}
\begin{scope}[shift={(0.3,0.75)}]
\draw (0.2, 0)--(0.2, 0.65);
\end{scope}
\end{tikzpicture}}\;.
$
This implies that Equation \eqref{eq:cut down2} holds.
\end{proof}
\begin{proposition}
    Suppose $\mathscr{P}$ is a $C^*$-planar algebra and $x\in\mathscr{P}_{2,\pm}$ is $\mathfrak{F}$-positive and $x_p$ is defined in Equation\eqref{eq:cutting down positive definite}. Then we have the following statements:
    \begin{enumerate}
        \item The spectral radius of $x_p$ is $r$ and $x_p$ is $\mathfrak{F}$-positive;
        \item Define
\begin{align*}
     \mathcal{E}_p:=\{y\in \mathscr{P}_{1,\pm}: y*\widehat{x_p}=ry\},
\end{align*}       
then $\mathcal{E}_p^+=\mathcal{E}^+$;
\item The following limit exists:
\begin{align}\label{eq:cut down perron operator sec}
y_p:=\lim_{\mathbb{R}\ni z\to r^+}(z-r)\sum_{j=0}^\infty \dfrac{x_p^j }{z^{j+1}}.
\end{align}
Moreover, $y_p$ is a nonzero $\mathfrak{F}$-positive element in $\mathscr{P}_{2,\pm}$ such that $x_py_p=y_px_p=ry_p$.
    \end{enumerate}
\end{proposition}
\begin{proof}
    (1) It is clear that $\widehat{x_p}$ is positive. 
    Let $r(x_p)$ be the spectral radius of $x_p$ and $z_p\in \mathscr{P}_{1,\pm}$ be the PF eigenvector. 
    Then we have $r(x_p)z_p=\widehat{z_p}\ast\widehat{x_p}=\widehat{z_p}\ast\widehat{x}$. 
    The second equality is due to Equation \eqref{eq:cut down2}. So $r(x_p)\leq r$. 
    Similarly, let $z\in \mathscr{P}_{1,\pm}$ be the PF eigenvector of $x$. 
    Then $\mathcal{R}(z)\leq p$ and $rz=z\ast\widehat{x}=z\ast\widehat{x_p}$.
    We obtain $r\leq r(x_p)$.
    Thus $r(x_p)=r$.

    (2) It is easy to check.

    (3) Let $k_p$ be the order of the resolvent $(z-x_p)^{-1}$, $z\in\mathbb{C}$, at the pole $r$. We define
    \begin{align}
y_p:=\lim_{\mathbb{R}\ni z\to r^+}(z-r)^{k_p}\sum_{j=0}^\infty \dfrac{x_p^j }{z^{j+1}}.
\end{align}
From  Equation \eqref{eq:perron operator} in the proof of Proposition \ref{prop:PF for planar algebra}, we see  that $y_p\in\mathscr{P}_{2,\pm}$ is $\mathfrak{F}$-positive such that $x_py_p=y_px_p=ry_p$. We have 
\begin{align*}
    tr_{1,\pm}(p\ast \widehat{y_p})= 
\raisebox{-0.9cm}{
\begin{tikzpicture}[scale=1.5]
\begin{scope}[shift={(-0.6, 0)}]
\draw [blue] (0,0) rectangle (0.35, 0.5);
\node at (0.175, 0.25) {$p$};
\end{scope}
\draw [blue] (0,0) rectangle (0.7, 0.5);
\node at (0.35, 0.25) {$\widehat{y_p}$};
\draw (0.2, 0.5) .. controls +(0, 0.4) and +(0, 0.4) .. (-0.425, 0.5);
\draw (0.2, 0) .. controls +(0, -0.4) and +(0, -0.4) .. (-0.425, 0);
\draw (0.5, 0.5) .. controls +(0, 0.4) and +(0, 0.4) .. (1.125, 0.5);
\draw (0.5, 0) .. controls +(0, -0.4) and +(0, -0.4) .. (1.125, 0);
\draw (1.125, 0.5)--(1.125, 0);
\end{tikzpicture}}
=\raisebox{-0.9cm}{
\begin{tikzpicture}[scale=1.5]
\draw [blue] (0,0) rectangle (0.7, 0.5);
\node at (0.35, 0.25) {$\widehat{y_p}$};
\draw (0.2, 0.5) .. controls +(0, 0.4) and +(0, 0.4) .. (-0.425, 0.5);
\draw (0.2, 0) .. controls +(0, -0.4) and +(0, -0.4) .. (-0.425, 0);
\draw (0.5, 0.5) .. controls +(0, 0.4) and +(0, 0.4) .. (1.125, 0.5);
\draw (0.5, 0) .. controls +(0, -0.4) and +(0, -0.4) .. (1.125, 0);
\draw (1.125, 0.5)--(1.125, 0)(-0.425, 0.5)--(-0.425, 0);
\end{tikzpicture}}=tr_{2,\pm}(\widehat{y_p})>0,
\end{align*}
which implies $p\ast \widehat{y_p}\neq0$. Note that $p\ast \widehat{x_p}^{\ast j}\leq \lambda z_p\ast \widehat{x_p}^{\ast j}=\lambda r^j z_p$ for some $\lambda>0$, where $z_p$ is the PF eigenvector of $x_p$.
Now if $k_p\geq 2$, then
\begin{align*}
0\leq    p\ast  \widehat{y_p}\leq \lim_{\mathbb{R}\ni z\to r^+}(z-r)^{k_p}\sum_{j=0}^\infty \dfrac{\lambda r^j z_p}{z^{j+1}}=\lambda z_p  \lim_{\mathbb{R}\ni z\to r^+}(z-r)^{k_p-1}=0,
\end{align*}
which leads to a contradiction. So $k_p=1$ and Equation \eqref{eq:cut down perron operator sec} is well-defined.
\end{proof}
We define
\begin{align}\label{eq:maximal support eigenvector}
    \zeta:=p\ast \widehat{y_p}=\lim_{\mathbb{R}\ni z\to r^+}(z-r)\sum_{j=0}^\infty \dfrac{p\ast \widehat{x_p}^{\ast j} }{z^{j+1}}=\lim_{\mathbb{R}\ni z\to r^+}(z-r)\sum_{j=0}^\infty \dfrac{p\ast \widehat{x}^{\ast j} }{z^{j+1}}.
\end{align}
The last equality follows from Equation \eqref{eq:cut down2}.
Then 
\begin{align*}
  \zeta\ast \widehat{x} =\zeta\ast \widehat{x_p}=p\ast \widehat{y_p}\ast\widehat{x_p}=rp\ast \widehat{y_p}=r\zeta.
\end{align*}
This implies $\zeta\in \mathcal{E}^+$ and $\mathcal{R}(\zeta)=p$. So $\zeta$ is the PF eigenvector of $x$ with maximal support.

\subsection{C$^*$-Algebraic Structures of Eigenspaces}\label{subsec:structure}

Let $x\in\mathscr{P}_{2,\pm}$ be a $\mathfrak{F}$-positive element. We introduce three subalgebras  of $\mathscr{P}_{1,\pm}$ for $x$, which are given by the following graphically representations:
\begin{align}\label{eq:algebra A}
    \mathcal{A}:=\left\{y\in\mathscr{P}_{1,\pm}:\ 
    \raisebox{-1.1cm}{
\begin{tikzpicture}[scale=1.4]
\draw (0.2, -0.3)--(0.2, 1.25)   (0.5, -0.3)--(0.5, 1.25) ;
\draw [blue, fill=white] (0,0) rectangle (0.7, 0.5);
\node at (0.35, 0.25) {$x$};
\begin{scope}[shift={(0.35, 0.7)}]
\draw [blue, fill=white] (0,0) rectangle (0.35, 0.35);
\node at (0.175, 0.175) {$\overline{y}$};
\end{scope}
\end{tikzpicture}}
=
\raisebox{-0.9cm}{
\begin{tikzpicture}[scale=1.4]
\draw (0.2, -0.75)--(0.2, .8)   (0.5, -0.75)--(0.5, 0.8) ;
\draw [blue, fill=white] (0,0) rectangle (0.7, 0.5);
\node at (0.35, 0.25) {$x$};
\begin{scope}[shift={(0.35, -0.55)}]
\draw [blue, fill=white] (0,0) rectangle (0.35, 0.35);
\node at (0.175, 0.175) {$\overline{y}$};
\end{scope}
\end{tikzpicture}}
    \right\}=\left\{y\in\mathscr{P}_{1,\pm}:\ \raisebox{-1.1cm}{
\begin{tikzpicture}[scale=1.4]
\draw (0.2, -0.75)--(0.2, .8)   (0.5, -0.75)--(0.5, 0.8) ;
\draw [blue, fill=white] (0,0) rectangle (0.7, 0.5);
\node at (0.35, 0.25) { $\widehat{x}$};
\begin{scope}[shift={(0, -0.55)}]
\draw [blue, fill=white] (0,0) rectangle (0.35, 0.35);
\node at (0.175, 0.175) {$\overline{y}$};
\end{scope}
\end{tikzpicture}}
=
\raisebox{-1.1cm}{
\begin{tikzpicture}[scale=1.4]
\draw (0.2, -0.75)--(0.2, .8)    (0.5, -0.75)--(0.5, 0.8);
\draw [blue, fill=white] (0,0) rectangle (0.7, 0.5);
\node at (0.35, 0.25) { $\widehat{x}$};
\begin{scope}[shift={(0.35, -0.55)}]
\draw [blue, fill=white] (0,0) rectangle (0.35, 0.35);
\node at (0.175, 0.175) {$y$};
\end{scope}
\end{tikzpicture}}
\right\}\subseteq\mathscr{P}_{1,\pm}.
\end{align}
  \begin{align}\label{eq:algebra B}
    \mathcal{B}:=\left\{y\in\mathscr{P}_{1,\pm}:\ 
    \raisebox{-1.1cm}{
\begin{tikzpicture}[scale=1.4]
\draw (0.2, -0.3)--(0.2, 1.25)   (0.5, -0.3)--(0.5, 1.25) ;
\draw [blue, fill=white] (0,0) rectangle (0.7, 0.5);
\node at (0.35, 0.25) {$x$};
\begin{scope}[shift={(0, 0.7)}]
\draw [blue, fill=white] (0,0) rectangle (0.35, 0.35);
\node at (0.175, 0.175) {$y$};
\end{scope}
\end{tikzpicture}}
=
\raisebox{-0.9cm}{
\begin{tikzpicture}[scale=1.4]
\draw (0.2, -0.75)--(0.2, .8)   (0.5, -0.75)--(0.5, 0.8) ;
\draw [blue, fill=white] (0,0) rectangle (0.7, 0.5);
\node at (0.35, 0.25) {$x$};
\begin{scope}[shift={(0, -0.55)}]
\draw [blue, fill=white] (0,0) rectangle (0.35, 0.35);
\node at (0.175, 0.175) {$y$};
\end{scope}
\end{tikzpicture}}
    \right\}=\left\{y\in\mathscr{P}_{1,\pm}:\  \raisebox{-1.1cm}{
\begin{tikzpicture}[scale=1.4]
\draw (0.2, -0.3)--(0.2, 1.25)   (0.5, -0.3)--(0.5, 1.25) ;
\draw [blue, fill=white] (0,0) rectangle (0.7, 0.5);
\node at (0.35, 0.25) {$\widehat{x}$};
\begin{scope}[shift={(0.35, 0.7)}]
\draw [blue, fill=white] (0,0) rectangle (0.35, 0.35);
\node at (0.175, 0.175) {$y$};
\end{scope}
\end{tikzpicture}}
=
 \raisebox{-1.1cm}{
\begin{tikzpicture}[scale=1.4]
\draw (0.2, -0.3)--(0.2, 1.25)   (0.5, -0.3)--(0.5, 1.25) ;
\draw [blue, fill=white] (0,0) rectangle (0.7, 0.5);
\node at (0.35, 0.25) {$\widehat{x}$};
\begin{scope}[shift={(0, 0.7)}]
\draw [blue, fill=white] (0,0) rectangle (0.35, 0.35);
\node at (0.175, 0.175) {$\overline{y}$};
\end{scope}
\end{tikzpicture}}
\right\}\subseteq\mathscr{P}_{1,\pm}.
\end{align}
We have that $\mathcal{A}=\mathcal{B}^*$ since $x$ is $\mathfrak{F}$-positive.
The third algebra $\mathcal{C}$ is the intersection of $\mathcal{A}$ and $\mathcal{B}$, which is a $C^*$-algebra.
   \begin{align}\label{eq:algebra C}
    \mathcal{C}:=\left\{y\in\mathscr{P}_{1,\pm}:\ 
    \raisebox{-1.1cm}{
\begin{tikzpicture}[scale=1.4]
\draw (0.2, -0.3)--(0.2, 1.25)   (0.5, -0.3)--(0.5, 1.25) ;
\draw [blue, fill=white] (0,0) rectangle (0.7, 0.5);
\node at (0.35, 0.25) {$x$};
\begin{scope}[shift={(0.35, 0.7)}]
\draw [blue, fill=white] (0,0) rectangle (0.35, 0.35);
\node at (0.175, 0.175) {$\overline{y}$};
\end{scope}
\end{tikzpicture}}
=
\raisebox{-0.9cm}{
\begin{tikzpicture}[scale=1.4]
\draw (0.2, -0.75)--(0.2, .8)   (0.5, -0.75)--(0.5, 0.8) ;
\draw [blue, fill=white] (0,0) rectangle (0.7, 0.5);
\node at (0.35, 0.25) {$x$};
\begin{scope}[shift={(0.35, -0.55)}]
\draw [blue, fill=white] (0,0) rectangle (0.35, 0.35);
\node at (0.175, 0.175) {$\overline{y}$};
\end{scope}
\end{tikzpicture}},\  \raisebox{-1.1cm}{
\begin{tikzpicture}[scale=1.4]
\draw (0.2, -0.3)--(0.2, 1.25)   (0.5, -0.3)--(0.5, 1.25) ;
\draw [blue, fill=white] (0,0) rectangle (0.7, 0.5);
\node at (0.35, 0.25) {$x$};
\begin{scope}[shift={(0, 0.7)}]
\draw [blue, fill=white] (0,0) rectangle (0.35, 0.35);
\node at (0.175, 0.175) {$y$};
\end{scope}
\end{tikzpicture}}
=
\raisebox{-0.9cm}{
\begin{tikzpicture}[scale=1.4]
\draw (0.2, -0.75)--(0.2, .8)   (0.5, -0.75)--(0.5, 0.8) ;
\draw [blue, fill=white] (0,0) rectangle (0.7, 0.5);
\node at (0.35, 0.25) {$x$};
\begin{scope}[shift={(0, -0.55)}]
\draw [blue, fill=white] (0,0) rectangle (0.35, 0.35);
\node at (0.175, 0.175) {$y$};
\end{scope}
\end{tikzpicture}}
    \right\}\subseteq\mathscr{P}_{1,\pm}.
\end{align}
Equivalently, we have
\begin{align}
 \mathcal{C}:=\left\{y\in\mathscr{P}_{1,\pm}:\ \raisebox{-1.1cm}{
\begin{tikzpicture}[scale=1.4]
\draw (0.2, -0.75)--(0.2, .8)   (0.5, -0.75)--(0.5, 0.8) ;
\draw [blue, fill=white] (0,0) rectangle (0.7, 0.5);
\node at (0.35, 0.25) { $\widehat{x}$};
\begin{scope}[shift={(0, -0.55)}]
\draw [blue, fill=white] (0,0) rectangle (0.35, 0.35);
\node at (0.175, 0.175) {$\overline{y}$};
\end{scope}
\end{tikzpicture}}
=
\raisebox{-1.1cm}{
\begin{tikzpicture}[scale=1.4]
\draw (0.2, -0.75)--(0.2, .8)    (0.5, -0.75)--(0.5, 0.8);
\draw [blue, fill=white] (0,0) rectangle (0.7, 0.5);
\node at (0.35, 0.25) { $\widehat{x}$};
\begin{scope}[shift={(0.35, -0.55)}]
\draw [blue, fill=white] (0,0) rectangle (0.35, 0.35);
\node at (0.175, 0.175) {$y$};
\end{scope}
\end{tikzpicture}},\  \raisebox{-1.1cm}{
\begin{tikzpicture}[scale=1.4]
\draw (0.2, -0.3)--(0.2, 1.25)   (0.5, -0.3)--(0.5, 1.25) ;
\draw [blue, fill=white] (0,0) rectangle (0.7, 0.5);
\node at (0.35, 0.25) {$\widehat{x}$};
\begin{scope}[shift={(0.35, 0.7)}]
\draw [blue, fill=white] (0,0) rectangle (0.35, 0.35);
\node at (0.175, 0.175) {$y$};
\end{scope}
\end{tikzpicture}}
=
 \raisebox{-1.1cm}{
\begin{tikzpicture}[scale=1.4]
\draw (0.2, -0.3)--(0.2, 1.25)   (0.5, -0.3)--(0.5, 1.25) ;
\draw [blue, fill=white] (0,0) rectangle (0.7, 0.5);
\node at (0.35, 0.25) {$\widehat{x}$};
\begin{scope}[shift={(0, 0.7)}]
\draw [blue, fill=white] (0,0) rectangle (0.35, 0.35);
\node at (0.175, 0.175) {$\overline{y}$};
\end{scope}
\end{tikzpicture}}
\right\}\subseteq\mathscr{P}_{1,\pm}.
\end{align}

\begin{remark}\label{rem:three commutant}
    Let $\Phi$: $M_n(\mathbb{C})\to M_n(\mathbb{C})$ be a completely positive map. There are  Kraus operators $F_j\in \mathcal{M}$ such that
\[
\Phi(D)=\sum_{j=1}^m F_j D F_j^*, ~\quad\forall D \in M_n(\mathbb{C})\;.
\]
Then we have
\begin{align*}
    \mathcal{A}&=\{F_j^*:\ j=1,\ldots,m\}'\cap M_n(\mathbb{C}),\\
     \mathcal{B}&=\{F_j:\ j=1,\ldots,m\}'\cap M_n(\mathbb{C}),\\
      \mathcal{C}&=\{F_j,F_j^*:\ j=1,\ldots,m\}'\cap M_n(\mathbb{C}),
\end{align*}
where $\mathcal{S}'\cap M_n(\mathbb{C})=\{A\in M_n(\mathbb{C}):\ AB=BA,\ B\in \mathcal{S}  \}$ is the commutant of $\mathcal{S}$, $\mathcal{S}$ is a subset of $M_n(\mathbb{C})$. We could obtain the above three equations from Equation \eqref{eq:cptp Fourier} by choosing  $\{F_j\}_{j=1}^m$  as ${\rm Tr}(F_i^*F_j)=0$ if $i\neq j$.
\cite[Theorem 8.2]{NC10} assures that different choices of Kraus operators of $\Phi$ remain the commutants invariant.
\end{remark}

The $C^*$-algebra $\mathcal{C}$ is useful to characterize the eigenspace $\mathcal{E}$. 
Let $\zeta\in\mathcal{E}^+$ be the maximal-support eigenvector of the $\mathfrak{F}$-positive element $x$ defined in Equation \eqref{eq:maximal support eigenvector}. 
We will show that $\zeta\mathcal{C}\subseteq\mathcal{E}$ and $\zeta$ commutes with $\mathcal{C}$. 
Moreover, we will give a sufficient condition for $\zeta\mathcal{C}^+=\mathcal{E}^+$.

\begin{theorem}\label{thm:center}
Suppose $\mathscr{P}$ is a C$^*$-planar algebra and $x\in\mathscr{P}_{2,\pm}$ is $\mathfrak{F}$-positive. Then the following statements hold:
\begin{enumerate}
    \item $[p_{\max},y]=0$ for any $y\in \mathcal{C}$. Furthermore, $[\zeta,y]=0$ for any $y\in \mathcal{C}$.
    \item $\zeta \mathcal{C}\subseteq\mathcal{E}$.
    \item 
    For any direct summand $\mathcal{D}$ of $\mathcal{C}$ such that $\mathcal{D}\cong M_n(\mathbb{C})$, we have 
    $$\{\zeta\}''\mathcal{D}=\{p_\mathcal{D}\zeta\}''\mathcal{D}\cong \{p_\mathcal{D}\zeta\}''\otimes\mathcal{D},$$
    where $p_\mathcal{D}$ is the union of projections in $\mathcal{D}$. 
\end{enumerate}
\end{theorem}

\begin{proof}
(1)
For any $y\in\mathcal{C}$, we have
\begin{align*}
(y^*\zeta y)\ast \widehat{x}&=
y^* ( \zeta \ast \widehat{x} ) y=ry^*\zeta y,
\end{align*}
pictorially,
\begin{align*}
\raisebox{-1.4cm}{
\begin{tikzpicture}[scale=1.4]
\draw (0.2, -0.5)--(0.2, 1).. controls +(0, 0.3) and +(0, 0.3) .. (-0.4, 1);
\draw (0.2, -0.5).. controls +(0, -0.3) and +(0, -0.3) .. (-0.4, -0.5)--(-0.4, 1);
\draw  (0.5, -0.75)--(0.5, 1.25)  (0.5, -0.75)--(0.5, 1.25);
\draw [blue, fill=white] (0,0) rectangle (0.7, 0.5);
\node at (0.35, 0.25) { $\widehat{x}$};
\begin{scope}[shift={(-0.6, -0.5)}]
\draw [blue, fill=white] (0,0) rectangle (0.35, 0.35);
\node at (0.175, 0.175) {$y^*$};
\end{scope}
\begin{scope}[shift={(-0.6, 0.65)}]
\draw [blue, fill=white] (0,0) rectangle (0.35, 0.35);
\node at (0.175, 0.175) {$y$};
\end{scope}
\begin{scope}[shift={(-0.7, 0)}]
\draw [blue, fill=white] (0,0) rectangle (0.5, 0.5);
\node at (0.25, 0.25) {$\zeta$};
\end{scope}
\end{tikzpicture}}
=
\raisebox{-1.3cm}{
\begin{tikzpicture}[scale=1.4]
\draw  (0.5, -0.75)--(0.5, 1.25);
\draw [blue, fill=white] (0,0) rectangle (0.7, 0.5);
\node at (0.35, 0.25) { $\widehat{x}$};
\begin{scope}[shift={(0.35, -0.55)}]
\draw [blue, fill=white] (0,0) rectangle (0.35, 0.35);
\node at (0.175, 0.175) {$y^*$};
\end{scope}
\begin{scope}[shift={(0.35, 0.7)}]
\draw [blue, fill=white] (0,0) rectangle (0.35, 0.35);
\node at (0.175, 0.175) {$y$};
\end{scope}
\begin{scope}[shift={(-0.7, 0)}]
\draw [blue, fill=white] (0,0) rectangle (0.5, 0.5);
\node at (0.25, 0.25) {$\zeta$};
\end{scope}
\draw (0.2, 0.5).. controls +(0, 0.3) and +(0, 0.3) .. (-0.4, 0.5);
\draw (0.2, 0).. controls +(0, -0.3) and +(0, -0.3) .. (-0.4, 0);
\end{tikzpicture}}=r
\raisebox{-1.3cm}{
\begin{tikzpicture}[scale=1.4]
\draw  (0.5, -0.75)--(0.5, 1.25);
\begin{scope}[shift={(0.35, -0.55)}]
\draw [blue, fill=white] (0,0) rectangle (0.35, 0.35);
\node at (0.175, 0.175) {$y^*$};
\end{scope}
\begin{scope}[shift={(0.35, 0.7)}]
\draw [blue, fill=white] (0,0) rectangle (0.35, 0.35);
\node at (0.175, 0.175) {$y$};
\end{scope}
\begin{scope}[shift={(0.35, 0.075)}]
\draw [blue, fill=white] (0,0) rectangle (0.35, 0.35);
\node at (0.175, 0.175) {$\zeta$};
\end{scope}
\end{tikzpicture}}
\;.
\end{align*}
We see that $y^*\zeta y \in\mathcal{E}^+$.
Hence $\mathcal{R}(y^* \zeta)=\mathcal{R}( y^*\zeta y)\leq p_{\max}$
and then $y^* \zeta=p_{\max}y^* \zeta$.
Now we obtain that
\begin{align*}
y^* p_{\max}=p_{\max}y^* p_{\max},\ \forall y\in\mathcal{C}.
\end{align*}
This implies that $p_{\max}y=yp_{\max}$ for any $y\in\mathcal{C}$.
Note that $y \left( p_{\max}\ast \widehat{x}^{(\ast j)} \right)= \left( p_{\max}\ast \widehat{x}^{(\ast j)} \right)y$ for any $j \geq0$.
The following is the graph calculus for $j=1$.
\begin{align*}
 \raisebox{-1.1cm}{
\begin{tikzpicture}[scale=1.4]
\draw  (0.5, -0.3)--(0.5, 1.25)  ;
\draw [blue, fill=white] (0,0) rectangle (0.7, 0.5);
\node at (0.35, 0.25) { $\widehat{x}$};
\begin{scope}[shift={(0.35, 0.7)}]
\draw [blue, fill=white] (0,0) rectangle (0.35, 0.35);
\node at (0.175, 0.175) {$y$};
\end{scope}
\begin{scope}[shift={(-0.7, 0)}]
\draw [blue, fill=white] (0,0) rectangle (0.6, 0.5);
\node at (0.3, 0.25) {$p_{\max}$};
\end{scope}
\draw (0.2, 0.5).. controls +(0, 0.3) and +(0, 0.3) .. (-0.4, 0.5);
\draw (0.2, 0).. controls +(0, -0.3) and +(0, -0.3) .. (-0.4, 0);
\end{tikzpicture}}
=
 \raisebox{-1.1cm}{
\begin{tikzpicture}[scale=1.4]
\draw  (0.5, -0.3)--(0.5, 1.25) ;
\draw (-0.4, 0.95)--(-0.4, 0.5);
\draw (0.2, 0.95)--(0.2, 0.5);
\draw [blue, fill=white] (0,0) rectangle (0.7, 0.5);
\node at (0.35, 0.25) { $\widehat{x}$};
\begin{scope}[shift={(0, 0.6)}]
\draw [blue, fill=white] (0,0) rectangle (0.35, 0.35);
\node at (0.175, 0.175) {$\overline{y}$};
\end{scope}
\begin{scope}[shift={(-0.7, 0)}]
\draw [blue, fill=white] (0,0) rectangle (0.6, 0.5);
\node at (0.3, 0.25) {$p_{\max}$};
\end{scope}
\draw (0.2, 0.95).. controls +(0, 0.3) and +(0, 0.3) .. (-0.4, 0.95);
\draw (0.2, 0).. controls +(0, -0.3) and +(0, -0.3) .. (-0.4, 0);
\end{tikzpicture}}
=
 \raisebox{-1.1cm}{
\begin{tikzpicture}[scale=1.4]
\draw  (0.5, -0.3)--(0.5, 1.25)  ;
\draw  (-0.4, 0.95)--(-0.4, 0.5);
\draw (0.2, 0.95)--(0.2, 0.5) ;
\draw [blue, fill=white] (0,0) rectangle (0.7, 0.5);
\node at (0.35, 0.25) { $\widehat{x}$};
\begin{scope}[shift={(-0.6, 0.6)}]
\draw [blue, fill=white] (0,0) rectangle (0.35, 0.35);
\node at (0.175, 0.175) {$y$};
\end{scope}
\begin{scope}[shift={(-0.7, 0)}]
\draw [blue, fill=white] (0,0) rectangle (0.6, 0.5);
\node at (0.3, 0.25) {$p_{\max}$};
\end{scope}
\draw (0.2, 0.95).. controls +(0, 0.3) and +(0, 0.3) .. (-0.4, 0.95);
\draw (0.2, 0).. controls +(0, -0.3) and +(0, -0.3) .. (-0.4, 0);
\end{tikzpicture}}
=
 \raisebox{-1.1cm}{
\begin{tikzpicture}[scale=1.4]
\draw  (0.5, -0.75)--(0.5, 0.8);
\draw (-0.4, -0.45)--(-0.4, 0) ;
\draw (0.2, -0.45)--(0.2, 0) ;
\draw [blue, fill=white] (0,0) rectangle (0.7, 0.5);
\node at (0.35, 0.25) { $\widehat{x}$};
\begin{scope}[shift={(0, -0.45)}]
\draw [blue, fill=white] (0,0) rectangle (0.35, 0.35);
\node at (0.175, 0.175) {$\overline{y}$};
\end{scope}
\begin{scope}[shift={(-0.7, 0)}]
\draw [blue, fill=white] (0,0) rectangle (0.6, 0.5);
\node at (0.3, 0.25) {$p_{\max}$};
\end{scope}
\draw (0.2, 0.5).. controls +(0, 0.3) and +(0, 0.3) .. (-0.4, 0.5);
\draw (0.2, -0.45).. controls +(0, -0.3) and +(0, -0.3) .. (-0.4, -0.45);
\end{tikzpicture}}
=
 \raisebox{-1.1cm}{
\begin{tikzpicture}[scale=1.4]
\draw  (0.5, -0.75)--(0.5, 0.8);
\draw [blue, fill=white] (0,0) rectangle (0.7, 0.5);
\node at (0.35, 0.25) { $\widehat{x}$};
\begin{scope}[shift={(0.35, -0.55)}]
\draw [blue, fill=white] (0,0) rectangle (0.35, 0.35);
\node at (0.175, 0.175) {$y$};
\end{scope}
\begin{scope}[shift={(-0.7, 0)}]
\draw [blue, fill=white] (0,0) rectangle (0.6, 0.5);
\node at (0.3, 0.25) { $p_{\max}$};
\end{scope}
\draw (0.2, 0.5).. controls +(0, 0.3) and +(0, 0.3) .. (-0.4, 0.5);
\draw (0.2, 0).. controls +(0, -0.3) and +(0, -0.3) .. (-0.4, 0);
\end{tikzpicture}}.
\end{align*}
We have $\zeta y=y \zeta$.

(2) For any $y\in\mathcal{C}$, 
\begin{align*}
    (\zeta  y)\ast \widehat{x}=(\zeta\ast \widehat{x})y=r\zeta y,
\end{align*}
pictorially,
\begin{align*}
\raisebox{-1.4cm}{
\begin{tikzpicture}[scale=1.4]
\draw (0.2, -0.5)--(0.2, 1).. controls +(0, 0.3) and +(0, 0.3) .. (-0.4, 1);
\draw (0.2, -0.5).. controls +(0, -0.3) and +(0, -0.3) .. (-0.4, -0.5)--(-0.4, 1);
\draw  (0.5, -0.75)--(0.5, 1.25)  (0.5, -0.75)--(0.5, 1.25);
\draw [blue, fill=white] (0,0) rectangle (0.7, 0.5);
\node at (0.35, 0.25) { $\widehat{x}$};
\begin{scope}[shift={(-0.6, 0.65)}]
\draw [blue, fill=white] (0,0) rectangle (0.35, 0.35);
\node at (0.175, 0.175) {$y$};
\end{scope}
\begin{scope}[shift={(-0.7, 0)}]
\draw [blue, fill=white] (0,0) rectangle (0.5, 0.5);
\node at (0.25, 0.25) {$\zeta$};
\end{scope}
\end{tikzpicture}}
=
\raisebox{-1.3cm}{
\begin{tikzpicture}[scale=1.4]
\draw  (0.5, -0.75)--(0.5, 1.25);
\draw [blue, fill=white] (0,0) rectangle (0.7, 0.5);
\node at (0.35, 0.25) { $\widehat{x}$};
\begin{scope}[shift={(0.35, 0.7)}]
\draw [blue, fill=white] (0,0) rectangle (0.35, 0.35);
\node at (0.175, 0.175) {$y$};
\end{scope}
\begin{scope}[shift={(-0.7, 0)}]
\draw [blue, fill=white] (0,0) rectangle (0.5, 0.5);
\node at (0.25, 0.25) {$\zeta$};
\end{scope}
\draw (0.2, 0.5).. controls +(0, 0.3) and +(0, 0.3) .. (-0.4, 0.5);
\draw (0.2, 0).. controls +(0, -0.3) and +(0, -0.3) .. (-0.4, 0);
\end{tikzpicture}}=r
\raisebox{-1.4cm}{
\begin{tikzpicture}
\draw (0,-2)--(0,1);
\draw [blue, fill=white] (-0.3,-0.3) rectangle (0.3,0.3);
\node at (0,0) {$y$};
\begin{scope}[shift={(0, -1)}]
\draw [blue, fill=white] (-0.3,-0.3) rectangle (0.3,0.3);
\node at (0,0) {$\zeta$};
\end{scope}
\end{tikzpicture}}
\;.
\end{align*}

(3) Note that $\mathscr{P}_{1,\pm}\cong \oplus_{k=1}^m M_{n_k}(\mathbb{C})$.
Suppose $\mathcal{D}\subseteq M_{n_{k_0}}(\mathbb{C})$. 
Then $\mathcal{D}'\cap M_{n_{k_0}}(\mathbb{C})\cong  I_n\otimes M_{t}(\mathbb{C}) \oplus M_{n_{k_0}-nt}(\mathbb{C})$ for some $t\in \mathbb{N}$ via unitary equivalence $\alpha$.
Since $\zeta\in\mathcal{D}'\cap\mathscr{P}_{1,\pm}$, we have $p_\mathcal{D}\zeta$ is in 
$\alpha^{-1}(I_n\otimes M_{t}(\mathbb{C}) \oplus0)$.
Thus $\{p_\mathcal{D}\zeta\}''\mathcal{D}\cong \{p_\mathcal{D}\zeta\}''\otimes\mathcal{D}$. Since
$\zeta p_\mathcal{D}=p_\mathcal{D}\zeta$, we have $\{\zeta\}''\mathcal{D}=\{p_\mathcal{D}\zeta\}''\mathcal{D}$.
\end{proof}
We next prove $\mathcal{\zeta}\mathcal{C}^+=\mathcal{E}^+$ under some conditions and begin with the following two lemmas.
\begin{lemma}\label{lem:W equivalent discribe}
Let $q\in\mathscr{P}_{1,\pm}$ be a projection such that $q\leq p_{\max}$. Suppose $p_{\max}\in \mathcal{C}$.
Then 
$q\in \mathcal{C}$ if and only if $\mathcal{R}(q\ast \widehat{x})\leq q$ and  $\mathcal{R}((p_{\max}-q)\ast  \widehat{x}) \leq p_{\max}-q$.
\end{lemma}

\begin{proof}
Note that a projection $q\in \mathcal{C}$ if and only if
 \begin{align*}
 \raisebox{-1cm}{
\begin{tikzpicture}[xscale=4, yscale=1.2]
\draw (0.175, 1.35)--(0.175, -0.3) (0.525, 1.35)--(0.525, -0.3);
\draw [blue, fill=white] (0,0) rectangle (0.7, 0.5);
\node at (0.35, 0.25) {$\widehat{x}$};
\begin{scope}[shift={(0.35, 0.7)}]
\draw [blue, fill=white] (0.01,0) rectangle (0.34, 0.45);
\node at (0.175, 0.225) {$I-q$};
\end{scope}
\begin{scope}[shift={(0, 0.7)}]
\draw [blue, fill=white] (0.01,0) rectangle (0.34, 0.45);
\node at (0.175, 0.225) {$\overline{q}$};
\end{scope}
\end{tikzpicture}}=0
\text{ and }
 \raisebox{-1cm}{
\begin{tikzpicture}[xscale=4, yscale=1.2]
\draw (0.175, 1.35)--(0.175, -0.3) (0.525, 1.35)--(0.525, -0.3);
\draw [blue, fill=white] (0,0) rectangle (0.7, 0.5);
\node at (0.35, 0.25) {$\widehat{x}$};
\begin{scope}[shift={(0.35, 0.7)}]
\draw [blue, fill=white] (0.01,0) rectangle (0.34, 0.45);
\node at (0.175, 0.225) {$q$};
\end{scope}
\begin{scope}[shift={(0, 0.7)}]
\draw [blue, fill=white] (0.01,0) rectangle (0.34, 0.45);
\node at (0.175, 0.225) {$\overline{I-q}$};
\end{scope}
\end{tikzpicture}}=0
 \end{align*}
They are equivalent to
\begin{align*}
 \raisebox{-1.3cm}{
\begin{tikzpicture}[scale=1.4]
\draw  (0.5, -0.75)--(0.5, 1.25);
\draw [blue, fill=white] (0,0) rectangle (0.7, 0.5);
\node at (0.35, 0.25) { $\widehat{x}$};
\begin{scope}[shift={(0.35, -0.6)}]
\draw [blue, fill=white] (-0.15,0) rectangle (0.65, 0.35);
\node at (0.25, 0.175) {\small $I-q$};
\end{scope}
\begin{scope}[shift={(0.35, 0.75)}]
\draw [blue, fill=white] (-0.15,0) rectangle (0.65, 0.35);
\node at (0.25, 0.175) {\small $I-q$};
\end{scope}
\begin{scope}[shift={(-0.7, 0)}]
\draw [blue, fill=white] (0,0) rectangle (0.5, 0.5);
\node at (0.25, 0.25) {$q$};
\end{scope}
\draw (0.25, 0.5).. controls +(0, 0.425) and +(0, 0.425) .. (-0.45, 0.5);
\draw (0.25, 0).. controls +(0, -0.425) and +(0, -0.425) .. (-0.45, 0);
\end{tikzpicture}}=0
\text{ and }
 \raisebox{-1.3cm}{
\begin{tikzpicture}[scale=1.4]
\draw  (0.5, -0.75)--(0.5, 1.25) ;
\draw [blue, fill=white] (0,0) rectangle (0.7, 0.5);
\node at (0.35, 0.25) { $\widehat{x}$};
\begin{scope}[shift={(0.35, -0.6)}]
\draw [blue, fill=white] (-0.15,0) rectangle (0.5, 0.35);
\node at (0.175, 0.175) {$q$};
\end{scope}
\begin{scope}[shift={(0.35, 0.75)}]
\draw [blue, fill=white] (-0.15,0) rectangle (0.5, 0.35);
\node at (0.175, 0.175) {$q$};
\end{scope}
\begin{scope}[shift={(-0.7, 0)}]
\draw [blue, fill=white] (-0.2,0) rectangle (0.6, 0.5);
\node at (0.2, 0.25) {\small $I-q$};
\end{scope}
\draw (0.25, 0.5).. controls +(0, 0.425) and +(0, 0.425) .. (-0.45, 0.5);
\draw (0.25, 0).. controls +(0, -0.425) and +(0, -0.425) .. (-0.45, 0);
\end{tikzpicture}}=0
\end{align*}
Note that
\begin{align*}
 \raisebox{-1.3cm}{
\begin{tikzpicture}[scale=1.4]
\draw  (0.5, -0.75)--(0.5, 1.25)  ;
\draw [blue, fill=white] (0,0) rectangle (0.7, 0.5);
\node at (0.35, 0.25) { $\widehat{x}$};
\begin{scope}[shift={(0.35, -0.6)}]
\draw [blue, fill=white] (-0.15,0) rectangle (0.5, 0.35);
\node at (0.175, 0.175) {$q$};
\end{scope}
\begin{scope}[shift={(0.35, 0.75)}]
\draw [blue, fill=white] (-0.15,0) rectangle (0.5, 0.35);
\node at (0.175, 0.175) {$q$};
\end{scope}
\begin{scope}[shift={(-0.7, 0)}]
\draw [blue, fill=white] (-0.2,0) rectangle (0.6, 0.5);
\node at (0.2, 0.25) {\small $I-q$};
\end{scope}
\draw (0.25, 0.5).. controls +(0, 0.425) and +(0, 0.425) .. (-0.45, 0.5);
\draw (0.25, 0).. controls +(0, -0.425) and +(0, -0.425) .. (-0.45, 0);
\end{tikzpicture}}=
 \raisebox{-1.3cm}{
\begin{tikzpicture}[scale=1.4]
\draw  (0.5, -0.75)--(0.5, 1.25) ;
\draw [blue, fill=white] (0,0) rectangle (0.7, 0.5);
\node at (0.35, 0.25) { $\widehat{x}$};
\begin{scope}[shift={(0.35, -0.6)}]
\draw [blue, fill=white] (-0.15,0) rectangle (0.5, 0.35);
\node at (0.175, 0.175) {$q$};
\end{scope}
\begin{scope}[shift={(0.35, 0.75)}]
\draw [blue, fill=white] (-0.15,0) rectangle (0.5, 0.35);
\node at (0.175, 0.175) {$q$};
\end{scope}
\begin{scope}[shift={(-0.7, 0)}]
\draw [blue, fill=white] (-0.5,0) rectangle (0.6, 0.5);
\node at (0.08, 0.25) {\small$I-p_{\max}$};
\end{scope}
\draw (0.25, 0.5).. controls +(0, 0.425) and +(0, 0.425) .. (-0.45, 0.5);
\draw (0.25, 0).. controls +(0, -0.425) and +(0, -0.425) .. (-0.45, 0);
\end{tikzpicture}}+
 \raisebox{-1.3cm}{
\begin{tikzpicture}[scale=1.4]
\draw  (0.5, -0.75)--(0.5, 1.25) ;
\draw [blue, fill=white] (0,0) rectangle (0.7, 0.5);
\node at (0.35, 0.25) { $\widehat{x}$};
\begin{scope}[shift={(0.35, -0.6)}]
\draw [blue, fill=white] (-0.15,0) rectangle (0.5, 0.35);
\node at (0.175, 0.175) {$q$};
\end{scope}
\begin{scope}[shift={(0.35, 0.75)}]
\draw [blue, fill=white] (-0.15,0) rectangle (0.5, 0.35);
\node at (0.175, 0.175) {$q$};
\end{scope}
\begin{scope}[shift={(-0.8, 0)}]
\draw [blue, fill=white] (-0.4,0) rectangle (0.7, 0.5);
\node at (0.17, 0.25) {\small$p_{\max}-q$};
\end{scope}
\draw (0.25, 0.5).. controls +(0, 0.425) and +(0, 0.425) .. (-0.45, 0.5);
\draw (0.25, 0).. controls +(0, -0.425) and +(0, -0.425) .. (-0.45, 0);
\end{tikzpicture}}.
\end{align*}
By the assumption, we see the lemma is true.
\end{proof}
We say a nonzero positive element $y\in\mathscr{P}_{1,\mp}$ is a right PF eigenvector of an $\mathfrak{F}$-positive element $x$ if $\widehat{x}\ast y=ry$. 

\begin{lemma}\label{lem:center and range equivalent}
Suppose $\mathscr{P}$ is a  $C^*$-planar algebra and $x\in\mathscr{P}_{2,\pm}$ is $\mathfrak{F}$-positive.
Suppose that there exists a   right PF eigenvector $y_1\in\mathscr{P}_{1,\mp}$ such that  $p_{\max} \leq \overline{\mathcal{R}(y_1)}$.
Then for any $y \in\mathcal{E}^+$ satisfying $\mathcal{R}(y)\zeta=\zeta \mathcal{R}(y)$, we have
$\mathcal{R}((p_{\max}-\mathcal{R}(y))\ast \widehat{x}) = p_{\max}-\mathcal{R}(y)$.
\begin{proof}
Note that
\begin{align*}
r \zeta=\zeta\ast \widehat{ x}
=(\mathcal{R}(y)\zeta) * \widehat{x}+ ((p_{\max}-\mathcal{R}(y))\zeta)\ast \widehat{x}.
\end{align*}
Recall that $\mathcal{R}((\mathcal{R}(y)\zeta) \ast \widehat{x})= \mathcal{R}(y)$.
We obtain
\begin{align*}
r\mathcal{R}(y) \zeta
=& r \mathcal{R}(y) \zeta \mathcal{R}(y) \\
=& (\mathcal{R}(y)\zeta) * \widehat{x }+ \mathcal{R}(y)(((p_{\max}-\mathcal{R}(y))\zeta)\ast \widehat{x})\mathcal{R}(y).
\end{align*}
Note that
\begin{align*}
(\mathcal{R}(y)\zeta) * \widehat{x } *y_1
= r (\mathcal{R}(y)\zeta)  *y_1.
\end{align*}
By taking the convolution of $y_1$, we see that
\begin{align*}
\left( \mathcal{R}(y)(((p_{\max}-\mathcal{R}(y))\zeta)\ast \widehat{x})\mathcal{R}(y) \right)
\ast y_1=0.
\end{align*}
By the assumption that $p_{\max}\leq \overline{\mathcal{R}(y_1)}$ again, we see that $\mathcal{R}(y)(((p_{\max}-\mathcal{R}(y))\zeta)\ast \widehat{x})\mathcal{R}(y) =0$.
Thus
\begin{align*}
\mathcal{R}((p_{\max}-\mathcal{R}(y))\ast \widehat{x})
=\mathcal{R}(((p_{\max}-\mathcal{R}(y))\zeta)\ast \widehat{x})
\leq p_{\max}-\mathcal{R}(y).
\end{align*}
Note that
\begin{align*}
p_{\max}&=\mathcal{R}( p_{\max} \ast \widehat{x})\\
&=\mathcal{R}(\mathcal{R}(y)\ast \widehat{\Phi})\bigvee \mathcal{R}((( p_{\max}  -\mathcal{R}(y))\ast \widehat{x})\\
&=\mathcal{R}(y)\bigvee \mathcal{R}(((p_{\max}-\mathcal{R}(y))\ast  \widehat{x}).
\end{align*}
We have
\begin{align*}
\mathcal{R}((p_{\max}  -\mathcal{R}(y))\ast \widehat{x})=p_{\max} -\mathcal{R}(y).
\end{align*}
This completes the proof of the proposition.
\end{proof}
\end{lemma}

\begin{corollary}\label{cor:center and range equivalent}Suppose $\mathscr{P}$ is a $C^*$-planar algebra and $x\in\mathscr{P}_{2,\pm}$ is $\mathfrak{F}$-positive. 
Suppose that there exists a right PF eigenvector 
$y_1\in\mathscr{P}_{1,\mp}$ such that $p_{\max}=\overline{\mathcal{R}(y_1)}$. 
Then for any  $y\in\mathcal{E}^+$, $\zeta\mathcal{R}(y)=\mathcal{R}(y)\zeta$ if and only if $\mathcal{R}(y)\in\mathcal{C}^+$.
\begin{proof}
Note that $p_{\max}=\overline{\mathcal{R}(y_1)}$ implies  $p_{\max}\in\mathcal{C}^+$.
From Lemmas \ref{lem:W equivalent discribe} and \ref{lem:center and range equivalent}, we could prove the corollary.
\end{proof}
\end{corollary}

\begin{proposition}\label{prop:multiplicative structure commutative}
Suppose $\mathscr{P}$ is a C$^*$-planar algebra and $x\in\mathscr{P}_{2,\pm}$ is $\mathfrak{F}$-positive.
Suppose that there exists a right PF eigenvector 
$y_1\in\mathscr{P}_{1,\mp}$ such that $p_{\max}=\overline{\mathcal{R}(y_1)}$, and $\zeta\mathcal{R}(y)=\mathcal{R}(y)\zeta$ for any $y \in \mathcal{E}^+$. 
Then $$\mathcal{E}^+=\zeta \mathcal{C}^+.$$
 \end{proposition}
 
\begin{proof}
By Theorem \ref{thm:center}, we have  $\zeta\mathcal{C}^+\subseteq \mathcal{E}^+$.
For any extreme point $y\in (\mathcal{E}^+)_1$, we have $\mathcal{R}(y)\in \mathcal{C}^+$ from Corollary \ref{cor:center and range equivalent}.
Note that $\dfrac{\zeta \mathcal{R}(y)}{tr_{1,\pm}(\zeta \mathcal{R}(y)) }\in (\mathcal{E}^+)_1$ and $\mathcal{R}(\zeta \mathcal{R}(y))=\mathcal{R}(y)$.
By Proposition \ref{prop: minimal projection and extremal point 1}, we have
\begin{align*}
\dfrac{\zeta\mathcal{R}(y)}{tr_{1,\pm}(\zeta \mathcal{R}(y))}=y.
\end{align*}
Thus $\mathcal{E}^+\subseteq  \zeta\mathcal{C}^+ $.
In a summary, we have that $\mathcal{E}^+= \zeta \mathcal{C}^+$.
\end{proof}
We need a litter more effort to release the conditions in Proposition \ref{prop:multiplicative structure commutative}.
Define
\begin{align}\label{eq:cutting down positive definite2}
x_{\zeta}:=
\raisebox{-1.6cm}{
\begin{tikzpicture}[yscale=1.5, xscale=2.2]
\draw [blue, fill=white] (0,0) rectangle (0.5, 0.5);
\node at (0.25, 0.25) {$x$};
\draw (0.35, 0.5)--(0.5, 0.75)   (0.35, 0)--(0.5, -0.25);
\draw (0.15, 0.5)--(0, 0.75)  (0.15, 0)--(0, -0.25);
\begin{scope}[shift={(-0.2,0.75)}]
\draw [blue, fill=white] (0,0) rectangle (0.4, 0.4);
\node at (0.2, 0.2) {$\zeta^{-1/2}$};
\draw (0.2, 0.4)--(0.2, 0.65);
\end{scope}
\begin{scope}[shift={(0.3,0.75)}]
\draw [blue, fill=white] (0,0) rectangle (0.4, 0.4);
\node at (0.2, 0.2) {$\overline{\zeta^{-1/2}}$};
\draw (0.2, 0.4)--(0.2, 0.65);
\end{scope}
\begin{scope}[shift={(-0.2,-0.65)}]
\draw [blue, fill=white] (0,0) rectangle (0.4, 0.4);
\node at (0.2, 0.2) {$\zeta^{1/2}$};
\draw (0.2,0)--(0.2,- 0.25);
\end{scope}
\begin{scope}[shift={(0.3,-0.65)}]
\draw [blue, fill=white] (0,0) rectangle (0.4, 0.4);
\node at (0.2, 0.2) {$\overline{\zeta^{1/2}}$};
\draw (0.2,0)--(0.2,- 0.25);
\end{scope}
\end{tikzpicture}}\;.
\end{align}

\begin{remark}
It is clear that $x_\zeta$ is $\mathfrak{F}$-positive. Moreover, the spectral radius of $x_\zeta$ is the same as $x$, that is $r$, and $p_{\max}\ast \widehat{x_\zeta}=rp_{\max}$.
\end{remark}
\begin{notation}\label{notation:ABC algebra}
We denote the $C^*$-algebra $\mathcal{C}$  for $x_\zeta$ by $\mathcal{C}_\zeta$, where $\mathcal{C}$
is defined in Equation \eqref{eq:algebra C}. For the clarity, we also write the graphically representation for $\mathcal{C}_\zeta$.
\begin{align}
    \mathcal{C}_\zeta:=\left\{y\in\mathscr{P}_{1,\pm}:\ 
    \raisebox{-1.1cm}{
\begin{tikzpicture}[scale=1.4]
\draw (0.2, -0.3)--(0.2, 1.25)   (0.5, -0.3)--(0.5, 1.25) ;
\draw [blue, fill=white] (0,0) rectangle (0.7, 0.5);
\node at (0.35, 0.25) {$x_\zeta$};
\begin{scope}[shift={(0.35, 0.7)}]
\draw [blue, fill=white] (0,0) rectangle (0.35, 0.35);
\node at (0.175, 0.175) {$\overline{y}$};
\end{scope}
\end{tikzpicture}}
=
\raisebox{-0.9cm}{
\begin{tikzpicture}[scale=1.4]
\draw (0.2, -0.75)--(0.2, .8)   (0.5, -0.75)--(0.5, 0.8) ;
\draw [blue, fill=white] (0,0) rectangle (0.7, 0.5);
\node at (0.35, 0.25) {$x_\zeta$};
\begin{scope}[shift={(0.35, -0.55)}]
\draw [blue, fill=white] (0,0) rectangle (0.35, 0.35);
\node at (0.175, 0.175) {$\overline{y}$};
\end{scope}
\end{tikzpicture}},\  \raisebox{-1.1cm}{
\begin{tikzpicture}[scale=1.4]
\draw (0.2, -0.3)--(0.2, 1.25)   (0.5, -0.3)--(0.5, 1.25) ;
\draw [blue, fill=white] (0,0) rectangle (0.7, 0.5);
\node at (0.35, 0.25) {$x_\zeta$};
\begin{scope}[shift={(0, 0.7)}]
\draw [blue, fill=white] (0,0) rectangle (0.35, 0.35);
\node at (0.175, 0.175) {$y$};
\end{scope}
\end{tikzpicture}}
=
\raisebox{-0.9cm}{
\begin{tikzpicture}[scale=1.4]
\draw (0.2, -0.75)--(0.2, .8)   (0.5, -0.75)--(0.5, 0.8) ;
\draw [blue, fill=white] (0,0) rectangle (0.7, 0.5);
\node at (0.35, 0.25) {$x_\zeta$};
\begin{scope}[shift={(0, -0.55)}]
\draw [blue, fill=white] (0,0) rectangle (0.35, 0.35);
\node at (0.175, 0.175) {$y$};
\end{scope}
\end{tikzpicture}}
    \right\}\subseteq\mathscr{P}_{1,\pm}.
\end{align}
We denote the PF eigenspace $\mathcal{E}$ for $x_\zeta$
by $\mathcal{E}_\zeta$, where $\mathcal{E}$ is defined in Notation \ref{notation:eigenspace}. Specifically,
\begin{align*}
\mathcal{E}_\zeta=\{y\in \mathscr{P}_{1,\pm}: y*\widehat{x_\zeta}=ry\} \subseteq \mathscr{P}_{1,\pm}.
\end{align*}
 It is clear that  $\mathcal{E}_\zeta^+=\zeta^{-1/2}\mathcal{E}^+\zeta^{-1/2}$. 
\end{notation}

\begin{theorem}[\bf $C^*$-Algebraic Structure]\label{thm:C star structure}
Suppose $\mathscr{P}$ is a  C$^*$-planar algebra and $x\in\mathscr{P}_{2,\pm}$ is $\mathfrak{F}$-positive.
Suppose that there exists a right PF eigenvector $y_1\in\mathscr{P}_{1,\mp}$ of $x$  such that $p_{\max}\bigwedge\overline{\mathcal{R}(y_1)}^\perp=0$. Then 
\begin{align*}
   \mathcal{E}^+=\zeta^{1/2}\mathcal{C}_\zeta^+\zeta^{1/2}.
\end{align*}
\end{theorem}
\begin{proof}
 We have 
 \begin{align*}
     \widehat{x_\zeta}\ast (\overline{\zeta^{1/2}} y_1\overline{\zeta^{1/2}})=
     &
   \overline{\zeta^{1/2}}(  \widehat{x_p}\ast y_1)\overline{\zeta^{1/2}}\\
   =&\overline{\zeta^{1/2}} (\widehat{x}\ast y_1)\overline{\zeta^{1/2}}\\
   =&  r(\overline{\zeta^{1/2}} y_1\overline{\zeta^{1/2}}).
 \end{align*}
Thus $\overline{\zeta^{1/2}} y_1\overline{\zeta^{1/2}}$ is a right PF eigenvector of $x_\zeta$
such that  $p_{\max}=\overline{\mathcal{R}(\overline{\zeta^{1/2}} y_1\overline{\zeta^{1/2}})}$. 
Take the maximal-support eigenvector of $x_\zeta$ as $p_{\max}$, we have $p_{\max}\mathcal{R}(y)=\mathcal{R}(y)=\mathcal{R}(y)p_{\max}$ for any $y \in \mathcal{E}_\zeta^+$. So $x_\zeta$ satisfies the conditions of Proposition \ref{prop:multiplicative structure commutative} and   
 we have $\mathcal{E}_\zeta^+=p_{\max}\mathcal{C}_\zeta^+$. From $\mathcal{E}_\zeta^+=\zeta^{-1/2}\mathcal{E}^+\zeta^{-1/2}$, we could obtain the conclusion.
 \end{proof}
\begin{remark}
    Let $x\in\mathscr{P}_{2,\pm}$ be a $\mathfrak{F}$-positive element such that
      \begin{align}
      \raisebox{-0.45cm}{
\begin{tikzpicture}[scale=1.2]
\path  (0.55, 0.5) .. controls +(0, 0.3) and +(0, 0.3) .. (0.85, 0.5) -- (0.85, 0) .. controls +(0, -0.3) and +(0, -0.3) ..  (0.55, 0) ;
\path  (-0.15, -0.3) rectangle (0.15, 0.8);
\draw (0.15, 0.8)--(0.15, -0.3);
\draw [blue, fill=white] (0,0) rectangle (0.7, 0.5);
\node at (0.35, 0.25) {$\widehat{x}$};
\draw (0.55, 0.5) .. controls +(0, 0.3) and +(0, 0.3) .. (0.85, 0.5) -- (0.85, 0) .. controls +(0, -0.3) and +(0, -0.3) ..  (0.55, 0) ;
\end{tikzpicture}}=I.
  \end{align}
  Then $\widehat{x}:$ $\mathscr{P}_{1,\pm}\to\mathscr{P}_{1,\pm}$, $y\mapsto y\ast \widehat{x}$, 
  defines a trace-preserving, completely positive map. Theorem \ref{thm:C star structure} gives the $C^*$-algebraic structure of the PF eigenspace of $x$. 
\end{remark}

\subsection{Algebraic Structures of Eigenspaces}
In this section, we let $p=p_{\max}$ for simplicity, denote the algebras $\mathcal{A}$ and $\mathcal{B}$ for $x_p$ by  $\mathcal{A}_p$ and $\mathcal{B}_p$, where $\mathcal{A}$ and $\mathcal{B}$ are defined in Equations \eqref{eq:algebra A} and \eqref{eq:algebra B}, and $x_p$ is defined in Equation \eqref{eq:cutting down positive definite}.
We denote the PF eigenspace $\mathcal{E}$ for $x_p$ by $\mathcal{E}_p$, where $\mathcal{E}$ is defined in Notation \ref{notation:eigenspace}. It is clear that $\mathcal{E}_p^+=\mathcal{E}^+$. 

 \begin{lemma}\label{lem:algebras after cut down}
Suppose $\mathscr{P}$ is a  C$^*$-planar algebra and $x\in\mathscr{P}_{2,\pm}$ is $\mathfrak{F}$-positive. Then $$p\mathcal{C}_\zeta^+\subseteq \zeta^{-1/2}\mathcal{A}_p\zeta^{1/2}\cap\zeta^{1/2}\mathcal{B}_{p}\zeta^{-1/2}.$$
 \end{lemma}
 \begin{proof}
 For any $y\in p\mathcal{C}_\zeta^+$, 
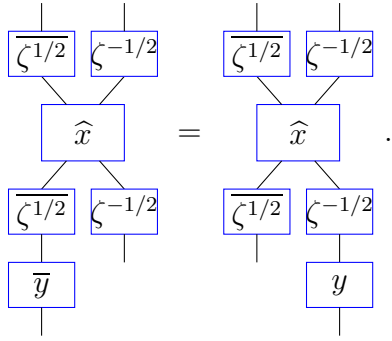
 \begin{align*}
     \raisebox{-2.6cm}{
\begin{tikzpicture}[yscale=1.5, xscale=2.2]
\draw [blue, fill=white] (0,0) rectangle (0.5, 0.5);
\node at (0.25, 0.25) {$\widehat{x}$};
\draw (0.35, 0.5)--(0.5, 0.75)   (0.35, 0)--(0.5, -0.25);
\draw (0.15, 0.5)--(0, 0.75)  (0.15, 0)--(0, -0.25);
\begin{scope}[shift={(-0.2,0.75)}]
\draw [blue, fill=white] (0,0) rectangle (0.4, 0.4);
\node at (0.2, 0.2) {$\overline{\zeta^{1/2}}$};
\draw (0.2, 0.4)--(0.2, 0.65);
\end{scope}
\begin{scope}[shift={(0.3,0.75)}]
\draw [blue, fill=white] (0,0) rectangle (0.4, 0.4);
\node at (0.2, 0.2) {$\zeta^{-1/2}$};
\draw (0.2, 0.4)--(0.2, 0.65);
\end{scope}
\begin{scope}[shift={(-0.2,-0.65)}]
\draw [blue, fill=white] (0,0) rectangle (0.4, 0.4);
\node at (0.2, 0.2) {$\overline{\zeta^{1/2}}$};
\draw (0.2,0)--(0.2,- 0.25);
\end{scope}
\begin{scope}[shift={(0.3,-0.65)}]
\draw [blue, fill=white] (0,0) rectangle (0.4, 0.4);
\node at (0.2, 0.2) {$\zeta^{-1/2}$};
\draw (0.2,0)--(0.2,- 0.25);
\end{scope}
\begin{scope}[shift={(-0.2,-1.3)}]
\draw [blue, fill=white] (0,0) rectangle (0.4, 0.4);
\node at (0.2, 0.2) {$\overline{y}$};
\draw (0.2,0)--(0.2,- 0.25);
\end{scope}
\end{tikzpicture}}=\raisebox{-2.6cm}{
\begin{tikzpicture}[yscale=1.5, xscale=2.2]
\draw [blue, fill=white] (0,0) rectangle (0.5, 0.5);
\node at (0.25, 0.25) {$\widehat{x}$};
\draw (0.35, 0.5)--(0.5, 0.75)   (0.35, 0)--(0.5, -0.25);
\draw (0.15, 0.5)--(0, 0.75)  (0.15, 0)--(0, -0.25);
\begin{scope}[shift={(-0.2,0.75)}]
\draw [blue, fill=white] (0,0) rectangle (0.4, 0.4);
\node at (0.2, 0.2) {$\overline{\zeta^{1/2}}$};
\draw (0.2, 0.4)--(0.2, 0.65);
\end{scope}
\begin{scope}[shift={(0.3,0.75)}]
\draw [blue, fill=white] (0,0) rectangle (0.4, 0.4);
\node at (0.2, 0.2) {$\zeta^{-1/2}$};
\draw (0.2, 0.4)--(0.2, 0.65);
\end{scope}
\begin{scope}[shift={(-0.2,-0.65)}]
\draw [blue, fill=white] (0,0) rectangle (0.4, 0.4);
\node at (0.2, 0.2) {$\overline{\zeta^{1/2}}$};
\draw (0.2,0)--(0.2,- 0.25);
\end{scope}
\begin{scope}[shift={(0.3,-0.65)}]
\draw [blue, fill=white] (0,0) rectangle (0.4, 0.4);
\node at (0.2, 0.2) {$\zeta^{-1/2}$};
\draw (0.2,0)--(0.2,- 0.25);
\end{scope}
\begin{scope}[shift={(0.3,-1.3)}]
\draw [blue, fill=white] (0,0) rectangle (0.4, 0.4);
\node at (0.2, 0.2) {$y$};
\draw (0.2,0)--(0.2,- 0.25);
\end{scope}
\end{tikzpicture}}.
 \end{align*}
 This is equivalent to 
 \begin{align*}
\raisebox{-1.6cm}
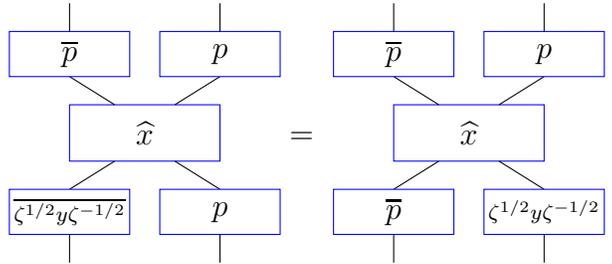
{
\begin{tikzpicture}[yscale=1.5, xscale=4]
\draw [blue, fill=white] (0,0) rectangle (0.5, 0.5);
\node at (0.25, 0.25) {$\widehat{x}$};
\draw (0.35, 0.5)--(0.5, 0.75)   (0.35, 0)--(0.5, -0.25);
\draw (0.15, 0.5)--(0, 0.75)  (0.15, 0)--(0, -0.25);
\begin{scope}[shift={(-0.2,0.75)}]
\draw [blue, fill=white] (0,0) rectangle (0.4, 0.4);
\node at (0.2, 0.2) {$\overline{p}$};
\draw (0.2, 0.4)--(0.2, 0.65);
\end{scope}
\begin{scope}[shift={(0.3,0.75)}]
\draw [blue, fill=white] (0,0) rectangle (0.4, 0.4);
\node at (0.2, 0.2) {$p$};
\draw (0.2, 0.4)--(0.2, 0.65);
\end{scope}
\begin{scope}[shift={(-0.2,-0.65)}]
\draw [blue, fill=white] (0,0) rectangle (0.4, 0.4);
\node at (0.2, 0.2) {\tiny$\overline{\zeta^{1/2}y\zeta^{-1/2}}$};
\draw (0.2, 0)--(0.2, -0.25);
\end{scope}
\begin{scope}[shift={(0.3,-0.65)}]
\draw [blue, fill=white] (0,0) rectangle (0.4, 0.4);
\node at (0.2, 0.2) {$p$};
\draw (0.2, 0)--(0.2, -0.25) ;
\end{scope}
\end{tikzpicture}}
=\raisebox{-1.6cm}{
\begin{tikzpicture}[yscale=1.5, xscale=4]
\draw [blue, fill=white] (0,0) rectangle (0.5, 0.5);
\node at (0.25, 0.25) {$\widehat{x}$};
\draw (0.35, 0.5)--(0.5, 0.75)   (0.35, 0)--(0.5, -0.25);
\draw (0.15, 0.5)--(0, 0.75)  (0.15, 0)--(0, -0.25);
\begin{scope}[shift={(-0.2,0.75)}]
\draw [blue, fill=white] (0,0) rectangle (0.4, 0.4);
\node at (0.2, 0.2) { $\overline{p}$};
\draw (0.2, 0.4)--(0.2, 0.65) ;
\end{scope}
\begin{scope}[shift={(0.3,0.75)}]
\draw [blue, fill=white] (0,0) rectangle (0.4, 0.4);
\node at (0.2, 0.2) {$p$};
\draw (0.2, 0.4)--(0.2, 0.65) ;
\end{scope}
\begin{scope}[shift={(-0.2,-0.65)}]
\draw [blue, fill=white] (0,0) rectangle (0.4, 0.4);
\node at (0.2, 0.2) {$\overline{p}$};
\draw (0.2, 0)--(0.2, -0.25) ;
\end{scope}
\begin{scope}[shift={(0.3,-0.65)}]
\draw [blue, fill=white] (0,0) rectangle (0.4, 0.4);
\node at (0.2, 0.2) {\tiny$\zeta^{1/2}y\zeta^{-1/2}$};
\draw (0.2, 0)--(0.2, -0.25) ;
\end{scope}
\end{tikzpicture}}.
\end{align*}
 So $\zeta^{1/2}y\zeta^{-1/2}\in \mathcal{A}_{p}$.  Analogously, we have $\zeta^{-1/2}y\zeta^{1/2}\in \mathcal{B}_{p}$.
 \end{proof}
 
 Now we obtain the structure of the eigenvector space with respect to $\mathcal{A}_p$ and $\mathcal{B}_p$.

 \begin{theorem}[\bf Algebraic Structure I]\label{thm:multiplicative structure 2}
Suppose $\mathscr{P}$ is a  C$^*$-planar algebra and $x\in\mathscr{P}_{2,\pm}$ is $\mathfrak{F}$-positive.
Suppose that there exists a right PF eigenvector $y_1\in\mathscr{P}_{1,\mp}$ of $x$  such that $p_{\max}\bigwedge\overline{\mathcal{R}(y_1)}^\perp=0$.
 Then
 \begin{align*}
 \mathcal{E}^+=\left\{ v \zeta v^*: v\in \mathcal{A}_{p}\right\}=\left\{ v^* \zeta v: v\in \mathcal{B}_{p}\right\}.
 \end{align*}
 \end{theorem}
\begin{proof}
Suppose that $v\in \mathcal{A}_{p}$.
Then
\begin{align*}
\raisebox{-1.3cm}{
\begin{tikzpicture}[scale=1.4]
\draw (0.2, -0.5)--(0.2, 1).. controls +(0, 0.3) and +(0, 0.3) .. (-0.4, 1);
\draw (0.2, -0.5).. controls +(0, -0.3) and +(0, -0.3) .. (-0.4, -0.5)--(-0.4, 1);
\draw  (0.5, -0.75)--(0.5, 1.25)  (0.5, -0.75)--(0.5, 1.25);
\draw [blue, fill=white] (0,0) rectangle (0.7, 0.5);
\node at (0.35, 0.25) { $\widehat{x_p}$};
\begin{scope}[shift={(-0.6, -0.5)}]
\draw [blue, fill=white] (0,0) rectangle (0.35, 0.35);
\node at (0.175, 0.175) {$v$};
\end{scope}
\begin{scope}[shift={(-0.6, 0.65)}]
\draw [blue, fill=white] (0,0) rectangle (0.35, 0.35);
\node at (0.175, 0.175) {$v^*$};
\end{scope}
\begin{scope}[shift={(-0.7, 0)}]
\draw [blue, fill=white] (0,0) rectangle (0.5, 0.5);
\node at (0.25, 0.25) {$\zeta$};
\end{scope}
\end{tikzpicture}}
=
\raisebox{-1.3cm}{
\begin{tikzpicture}[scale=1.4]
\draw  (0.5, -0.75)--(0.5, 1.25);
\draw [blue, fill=white] (0,0) rectangle (0.7, 0.5);
\node at (0.35, 0.25) { $\widehat{x_p}$};
\begin{scope}[shift={(0.35, -0.55)}]
\draw [blue, fill=white] (0,0) rectangle (0.35, 0.35);
\node at (0.175, 0.175) {$v$};
\end{scope}
\begin{scope}[shift={(0.35, 0.7)}]
\draw [blue, fill=white] (0,0) rectangle (0.35, 0.35);
\node at (0.175, 0.175) {$v^*$};
\end{scope}
\begin{scope}[shift={(-0.7, 0)}]
\draw [blue, fill=white] (0,0) rectangle (0.5, 0.5);
\node at (0.25, 0.25) {$\zeta$};
\end{scope}
\draw (0.25, 0.5).. controls +(0, 0.425) and +(0, 0.425) .. (-0.45, 0.5);
\draw (0.25, 0).. controls +(0, -0.425) and +(0, -0.425) .. (-0.45, 0);
\end{tikzpicture}}
=r\ \raisebox{-1.3cm}{\begin{tikzpicture}[scale=1.4]
\draw (-0.4, -0.75)--(-0.4, 1.25);
\begin{scope}[shift={(-0.6, -0.5)}]
\draw [blue, fill=white] (0,0) rectangle (0.35, 0.35);
\node at (0.175, 0.175) {$v$};
\end{scope}
\begin{scope}[shift={(-0.6, 0.65)}]
\draw [blue, fill=white] (0,0) rectangle (0.35, 0.35);
\node at (0.175, 0.175) {$v^*$};
\end{scope}
\begin{scope}[shift={(-0.7, 0)}]
\draw [blue, fill=white] (0,0) rectangle (0.5, 0.5);
\node at (0.25, 0.25) {$\zeta$};
\end{scope}
\end{tikzpicture}}\;.
\end{align*}
So $v\zeta v^*\in \mathcal{E}_p^+=\mathcal{E}^+$.

By Theorem \ref{thm:C star structure}, for any $y\in \mathcal{E}^+$, there exists $w\in \mathcal{C}_\zeta^+$ such that $ y=\zeta^{1/2} w \zeta^{1/2}$.
By Lemma \ref{lem:algebras after cut down}, we see that there exists $v\in\mathcal{A}_{p}$ such that
\begin{align*}
p w^{1/2} = \zeta^{-1/2} v \zeta^{1/2}.
\end{align*}
Hence
\begin{align*}
y = \zeta^{1/2} (\zeta^{-1/2} v \zeta^{1/2})(\zeta^{-1/2} v \zeta^{1/2})^* \zeta^{1/2}=v \zeta v^*.
\end{align*}
This completes the proof of the theorem.
\end{proof}
We obtain another representation of $\mathcal{E}^+$ with a little effort.
 \begin{theorem}[\bf Algebraic Structure II]\label{thm:multiplicative structure 3}
Suppose $\mathscr{P}$ is a  C$^*$-planar algebra and $x\in\mathscr{P}_{2,\pm}$ is $\mathfrak{F}$-positive.
Suppose that there exists a right PF eigenvector $y_1\in\mathscr{P}_{1,\mp}$ of $x$  such that $p_{\max}\bigwedge\overline{\mathcal{R}(y_1)}^\perp=0$.
 Then
 \begin{align*}
 \mathcal{E}^+=(\mathcal{A}_p\zeta)^+=(\zeta\mathcal{B}_p)^+.
 \end{align*}
 \end{theorem}
\begin{proof}
For any $v\in\mathcal{A}_p$, we have $\mathcal{R}(v\zeta)=\mathcal{R}(v\zeta^{1/2})=\mathcal{R}(v\zeta v^*)\leq p$. The last inequality is due to $v\zeta v^*\in\mathcal{E}^+$ in  Theorem \ref{thm:multiplicative structure 2}. 
So $p\mathcal{A}_p\zeta=\mathcal{A}_p\zeta$. 
By Theorem \ref{thm:C star structure} and Lemma \ref{lem:algebras after cut down}, we know that
$\mathcal{E}^+=\zeta^{1/2}\mathcal{C}_\zeta^+\zeta^{1/2}
\subseteq p\mathcal{A}_{p}\zeta=\mathcal{A}_{p}\zeta$. On the other hand $\mathcal{A}_{p}\zeta\subseteq \mathcal{E}$, so $(\mathcal{A}_{p}\zeta)^+\subseteq \mathcal{E}^+$.
Thus $ \mathcal{E}^+= (\mathcal{A}_{p}\zeta)^+$. 
Similarly, $\mathcal{E}^+=(\zeta\mathcal{B}_p)^+$. 
\end{proof}
\begin{remark}
    Let $x\in\mathscr{P}_{2,\pm}$ be a $\mathfrak{F}$-positive element such that
      $\widehat{x}:$ $\mathscr{P}_{1,\pm}\to\mathscr{P}_{1,\pm}$, $y\mapsto y\ast \widehat{x}$, 
  defines a trace-preserving, completely positive map. Theorems \ref{thm:multiplicative structure 2} and \ref{thm:multiplicative structure 3}   give the algebraic structures of the PF eigenspace of $x$. 
\end{remark}

\begin{remark}
    Let $x\in\mathscr{P}_{2,\pm}$ be a $\mathfrak{F}$-positive element such that
      \begin{align}
    \raisebox{-0.55cm}{
\begin{tikzpicture}[scale=1.2]
\draw [blue] (0,0) rectangle (0.7, 0.5);
\node at (0.35, 0.25) {$\widehat{x}$};
\draw (-0.15, 0)--(-0.15, 0.5) (0.55, 0)--(0.55, -0.2) (0.55, 0.5)--(0.55, 0.7);
\draw (0.15, 0.5) .. controls +(0, 0.3) and +(0, 0.3) .. (-0.15, 0.5);
\draw (0.15, 0) .. controls +(0, -0.3) and +(0, -0.3) .. (-0.15, 0);
\end{tikzpicture}}=I.
\end{align}
  Then $\widehat{x}:$ $\mathscr{P}_{1,\pm}\to\mathscr{P}_{1,\pm}$, $y\mapsto y\ast \widehat{x}$, 
  defines a unital completely positive map. In \cite{ChoEff77}, Choi and Effros proved that the PF eigenspace $\mathcal{E}$ is a $C^*$-algebra with respect to the multiplication 
  $\circ$ defined as follows: 
  \begin{align*}
     y_1\circ y_2= \lim_{k\to\infty}\frac{1}{k+1}\sum_{j=0}^k\frac{(y_1y_2)\ast\widehat{x}^{\ast j}}{r^{j+1}}
  \end{align*}
for any $y_1,y_2\in\mathcal{E}$. 
Note that a trace-preserving, completely positive map is the adjoint of a unital completely positive map.
Here we are able to consider  completely positive maps besides trace-preserving case and describe the C$^*$-algebraic structures and algebraic structures of their PF eigenspaces.
\end{remark}

\section{Quantum Error Correction}\label{sec:QEC}
Quantum error correction provides a method to protect quantum information. Developing robust quantum error correction is necessary for the realization of scalable quantum computing.  The first quantum error correction scheme for a single qubit state was proposed by Shor~\cite{Shor}. 
Inspired by the PF Theorem, we encode logical qubits by mixed states and we prove Theorem \ref{thm:PF to QEC} for quantum error corrections of mixed states. 
Using the multiplicative structure of the PF eigenspace, these mixed states in the PF eigenspace could be regarded as the tensor product of pure states and a fixed ancillar mixed state. 
In \S\ref{Sec:ContrastKL} we contrast this approach  with the Knill-Laflamme Theorem \ref{thm:Knill-Laflamme}, that characterizes quantum error corrections of pure states. 
 
\subsection{Quantum Error Correction} \label{subsec:QEC}
A quantum channel $\Phi$ on $M_n(\mathbb{C})$ is a completely positive, trace preserving (CPTP) linear map. (Sometimes one allows trace non-increasing maps.)

 We introduce some basic notions of error correction.
Let $\Phi_\mathcal{E}$ be a  quantum channel with Kraus operators $\{E_i\}$ representing errors.  Let $D\in M_n(\mathbb{C})$ be a density operator. Error correction means there exists a   completely-positive, trace-preserving map $\Phi_\mathcal{R}$ such that $\Phi_\mathcal{R}\Phi_\mathcal{E}(D)=D$. In many cases, the error $\Phi_\mathcal{E}$ is allowed to be trace-non-increasing, which is called a quantum operation (See Section 8.2 in \cite{NC10}). Error correction with probability means there exists a trace-non-increasing,  completely-positive map $\Phi_\mathcal{R}$ such that $\Phi_\mathcal{R}\Phi_\mathcal{E}(D)=rD$, where $r\leq1$ is the spectral radius of the composition $\Phi_\mathcal{R}\Phi_\mathcal{E}$.  If  there exists a density $D\in M_n(\mathbb{C})$ such that
    ${\rm Tr}(\Phi_\mathcal{R}\Phi_\mathcal{E}(D))<1$, then the quantum operation $\Phi_\mathcal{R}\Phi_\mathcal{E}$
    is non-trace-preserving. In this case other measurement outcomes may occur  and we get the correct output with probability $r$.  We refer the readers to e.g. Section 10 in \cite{NC10} for more details.

In this section, we will apply the structure of PF eigenspace to study quantum error correction. Precisely, we will prove Theorem~\ref{thm:PF to QEC}.
\begin{lemma}\label{lem:two commutant same}
     Suppose $\Phi_\mathcal{E}$: $M_n(\mathbb{C})\to M_n(\mathbb{C})$ is a completely positive map with Kraus operators $\{E_i\}_{i\in\mathcal{I}}$. Let $p\in M_n(\mathbb{C})$ be a projection.
    Let $E_jp=u_j|E_jp|$ be the polar decomposition. Then
     \begin{align*}
         \{u_j^* E_ip,pE_i^*u_j:\ i,j\in \mathcal{I}\}'\cap pM_n(\mathbb{C})p= \{pE_i^*E_jp:\ i,j\in\mathcal{I}\}'\cap pM_n(\mathbb{C})p.
     \end{align*}
     \begin{proof}
For any $y\in \{u_j^* E_ip,pE_i^*u_j:\ i,j\in \mathcal{I}\}'\cap pM_n(\mathbb{C})p$,
\begin{align}
    [y,pE_i^*u_ju_j^* E_jp]=0,\quad \forall i,j\in \mathcal{I}.
    \end{align}
    Since $u_ju_j^*=\mathcal{R}(E_jp)$, $[y,pE_i^* E_jp]=0$. On the other hand, for any $y\in \{pE_i^*E_jp:\ i,j\in\mathcal{I}\}'\cap pM_n(\mathbb{C})p$, since $[y,|E_jp|^{2}]=0$
    $[y,|E_jp|^{-1}]=0$. Then
    \begin{align*}
        [y,pE_i^*u_j]=[y,pE_i^*E_jp|E_jp|^{-1}]=0.
    \end{align*}
    This completes the proof.
     \end{proof}
\end{lemma}

\begin{theorem}[\bf PF to QEC]\label{thm:PF to QEC}
Suppose $\Phi_\mathcal{E}$: $M_n(\mathbb{C})\to M_n(\mathbb{C})$ is a  completely positive trace-preserving map with Kraus operators $\{E_i\}_{i\in\mathcal{I}}$ representing errors. Suppose $p$ is a  projection  $M_n(\mathbb{C})$.
Then there exists  a  completely positive trace-non-increasing map $\Phi_\mathcal{R}$ 
and a matrix subalgebra $\mathcal{D}$ in 
$$\{ pE_i^*E_jp :i,j\in \mathcal{I}\}' \cap pM_n(\mathbb{C}) p,$$
such that  
$\zeta_p \mathcal{D}\neq 0$ and for any $ D\in \mathcal{D}$, we have
$$\Phi_\mathcal{R}\Phi_\mathcal{E}(\zeta_p D)= r \zeta_p D,\quad r\geq  \left\|\sum_{i\in \mathcal{I}} \mathcal{R}(E_ip)\right\|^{-1},$$
where $r$ is the PF eigenvalue of $\Phi_\mathcal{R}\Phi_\mathcal{E}(p\cdot p)$ and $\zeta_p$ is the maximal-support PF eigenvector of $\Phi_\mathcal{R}\Phi_\mathcal{E}(p\cdot p)$ defined in Equation \eqref{eq:maximal support eigenvector}. Moreover, $\zeta_p D=D\zeta_p$  for any $D\in\mathcal{D}$ and $\zeta_p D=0$ implies $D=0$.
\end{theorem}
\begin{proof}
Let  $\displaystyle c=\left\|\sum_{i\in \mathcal{I}} \mathcal{R}(E_ip)\right\|\leq |\mathcal{I}|$. 
We could define the completely positive map $\Phi_\mathcal{R}$ as follows:
\begin{align*}
\Phi_\mathcal{R}(D)=\sum_{j\in\mathcal{I}} R_j DR_j^*,\quad  D\in M_n(\mathbb{C}),
\end{align*}
where $R_j=c^{-1/2}u_j^*$ and $u_j$ is defined in Lemma \ref{lem:two commutant same}. Then
\begin{align*}
\sum_{j\in \mathcal{I}} R_j^* R_j=c^{-1}\sum_{j\in\mathcal{I}} u_j u_j^*=c^{-1}\sum_{j\in \mathcal{I}} \mathcal{R}(E_jp)\leq I,
\end{align*}
which implies that $\Phi_\mathcal{R}$ is trace-non-increasing.

We define $\Phi=\Phi_\mathcal{R}\Phi_\mathcal{E}$ and $\Phi_p=\Phi(p\cdot p)=p\Phi(p\cdot p)p$. Pictorially,
\begin{align*}
    \Phi_{p}:=
\raisebox{-1.6cm}{
\begin{tikzpicture}[yscale=1.5, xscale=2.2]
\draw [blue, fill=white] (0,0) rectangle (0.5, 0.5);
\node at (0.25, 0.25) {$\Phi$};
\draw (0.35, 0.5)--(0.5, 0.75)   (0.35, 0)--(0.5, -0.25);
\draw (0.15, 0.5)--(0, 0.75)  (0.15, 0)--(0, -0.25);
\begin{scope}[shift={(-0.2,0.75)}]
\draw [blue, fill=white] (0,0) rectangle (0.4, 0.4);
\node at (0.2, 0.2) {$p$};
\draw (0.2, 0.4)--(0.2, 0.65);
\end{scope}
\begin{scope}[shift={(0.3,0.75)}]
\draw [blue, fill=white] (0,0) rectangle (0.4, 0.4);
\node at (0.2, 0.2) {$\overline{p}$};
\draw (0.2, 0.4)--(0.2, 0.65);
\end{scope}
\begin{scope}[shift={(-0.2,-0.65)}]
\draw [blue, fill=white] (0,0) rectangle (0.4, 0.4);
\node at (0.2, 0.2) {$p$};
\draw (0.2,0)--(0.2,- 0.25);
\end{scope}
\begin{scope}[shift={(0.3,-0.65)}]
\draw [blue, fill=white] (0,0) rectangle (0.4, 0.4);
\node at (0.2, 0.2) {$\overline{p}$};
\draw (0.2,0)--(0.2,- 0.25);
\end{scope}
\end{tikzpicture}}.
\end{align*}
Recall that the C$^*$-algebra $\mathcal{C}_p$ is defined as follows:
 \begin{align*}
    \mathcal{C}_p:=\left\{D\in M_n(\mathbb{C}):\ 
    \raisebox{-1.1cm}{
\begin{tikzpicture}[scale=1.4]
\draw (0.2, -0.3)--(0.2, 1.25)   (0.5, -0.3)--(0.5, 1.25) ;
\draw [blue, fill=white] (0,0) rectangle (0.7, 0.5);
\node at (0.35, 0.25) {$\Phi_p$};
\begin{scope}[shift={(0.35, 0.7)}]
\draw [blue, fill=white] (0,0) rectangle (0.35, 0.35);
\node at (0.175, 0.175) {$\overline{D}$};
\end{scope}
\end{tikzpicture}}
=
\raisebox{-0.9cm}{
\begin{tikzpicture}[scale=1.4]
\draw (0.2, -0.75)--(0.2, .8)   (0.5, -0.75)--(0.5, 0.8) ;
\draw [blue, fill=white] (0,0) rectangle (0.7, 0.5);
\node at (0.35, 0.25) {$\Phi_p$};
\begin{scope}[shift={(0.35, -0.55)}]
\draw [blue, fill=white] (0,0) rectangle (0.35, 0.35);
\node at (0.175, 0.175) {$\overline{D}$};
\end{scope}
\end{tikzpicture}},\  \raisebox{-1.1cm}{
\begin{tikzpicture}[scale=1.4]
\draw (0.2, -0.3)--(0.2, 1.25)   (0.5, -0.3)--(0.5, 1.25) ;
\draw [blue, fill=white] (0,0) rectangle (0.7, 0.5);
\node at (0.35, 0.25) {$\Phi_p$};
\begin{scope}[shift={(0, 0.7)}]
\draw [blue, fill=white] (0,0) rectangle (0.35, 0.35);
\node at (0.175, 0.175) {$D$};
\end{scope}
\end{tikzpicture}}
=
\raisebox{-0.9cm}{
\begin{tikzpicture}[scale=1.4]
\draw (0.2, -0.75)--(0.2, .8)   (0.5, -0.75)--(0.5, 0.8) ;
\draw [blue, fill=white] (0,0) rectangle (0.7, 0.5);
\node at (0.35, 0.25) {$\Phi_p$};
\begin{scope}[shift={(0, -0.55)}]
\draw [blue, fill=white] (0,0) rectangle (0.35, 0.35);
\node at (0.175, 0.175) {$D$};
\end{scope}
\end{tikzpicture}}
    \right\}.
\end{align*}
By Theorem \ref{thm:center}, we have that $\zeta_p$ commutes with the algebra $\mathcal{C}_p$ and 
\begin{align*}
   \Phi_\mathcal{R}\Phi_{\mathcal{E}}(\zeta_p D)=\Phi_p(\zeta_p D)= r \zeta_p D
\end{align*}
for any $D\in\mathcal{C}_p$. By Remark \ref{rem:three commutant},
\begin{align*}
     \{u_j^* E_ip,pE_i^*u_j:\ i,j\in \mathcal{I}\}'\cap M_n(\mathbb{C})=\mathcal{C}_p.
\end{align*}
By Lemma \ref{lem:two commutant same}, 
\begin{align*}
    \{ pE_i^*E_jp :i,j\in \mathcal{I}\}' \cap pM_n(\mathbb{C}) p\subseteq\mathcal{C}_p.
\end{align*}
Thus, for any $D\in\mathcal{D}$, $\zeta_p D=D\zeta_p$. Since $\mathcal{D}$ is a matrix algebra and 
$\zeta_p \mathcal{D}\neq0$, $\zeta_p D=0$ implies $D=0$.

We next show that $r\geq c^{-1}$. Indeed,
\begin{align*}
\Phi_p(p)&=\sum_{i,j\in\mathcal{I}}R_jE_ipE_i^*R_j^*\\
&=c^{-1}\sum_{i,j\in\mathcal{I}}u_j^*E_ipE_i^*u_j\\
&=c^{-1}\sum_{j\in\mathcal{I}}u_j^*\sum_{i\in\mathcal{I}}(E_ipE_i^*)u_j\\
&\geq c^{-1}\sum_{j\in\mathcal{I}}u_j^*E_jpE_j^*u_j\\
&=c^{-1}\sum_{j\in\mathcal{I}}pE_j^*E_jp \\
&=c^{-1}p.
\end{align*}
By Remark \ref{rem:one box PF}, there exists a nonzero positive-definite matrix $D_0\in M_n(\mathbb{C})$ such that
\begin{align*}
    \raisebox{-0.9cm}{
\begin{tikzpicture}[scale=1.5]
\begin{scope}[shift={(0.95, 0)}]
\draw [blue] (0,0) rectangle (0.35, 0.5);
\node at (0.175, 0.25) {$D_0$};
\end{scope}
\draw [blue] (0,0) rectangle (0.7, 0.5);
\node at (0.35, 0.25) {$\widehat{\Phi_p}$};
\draw  (0.2, 0)--(0.2, -0.4) (0.2, 0.5)--(0.2, 0.9);
\draw (0.5, 0.5) .. controls +(0, 0.4) and +(0, 0.4) .. (1.125, 0.5);
\draw (0.5, 0) .. controls +(0, -0.4) and +(0, -0.4) .. (1.125, 0);
\end{tikzpicture}}
=r
\raisebox{-0.9cm}{
\begin{tikzpicture}[scale=1.5]
\begin{scope}[shift={(1.5, 0)}]
\draw [blue] (0,0) rectangle (0.35, 0.5);
\node at (0.175, 0.25) {$D_0$};
\draw  (0.175, 0)--(0.175, -0.4) (0.175, 0.5)--(0.175, 0.9);
\end{scope}
\end{tikzpicture}}.
\end{align*}
 Then
 \begin{align*}
r\raisebox{-0.9cm}{
\begin{tikzpicture}[scale=1.5]
\begin{scope}[shift={(0.025, 0)}]
\draw [blue] (0,0) rectangle (0.35, 0.5);
\node at (0.175, 0.25) {$D_0$};
\end{scope}
\draw (0.2, 0.5) .. controls +(0, 0.4) and +(0, 0.4) .. (-0.425, 0.5)--(-0.425, 0);
\draw (0.2, 0) .. controls +(0, -0.4) and +(0, -0.4) .. (-0.425, 0);
\end{tikzpicture}}
=r\raisebox{-0.9cm}{
\begin{tikzpicture}[scale=1.5]
\begin{scope}[shift={(-0.6, 0)}]
\draw [blue] (0,0) rectangle (0.35, 0.5);
\node at (0.175, 0.25) {$p$};
\end{scope}
\begin{scope}[shift={(0.025, 0)}]
\draw [blue] (0,0) rectangle (0.35, 0.5);
\node at (0.175, 0.25) {$D_0$};
\end{scope}
\draw (0.2, 0.5) .. controls +(0, 0.4) and +(0, 0.4) .. (-0.425, 0.5);
\draw (0.2, 0) .. controls +(0, -0.4) and +(0, -0.4) .. (-0.425, 0);
\end{tikzpicture}}=
\raisebox{-0.9cm}{
\begin{tikzpicture}[scale=1.5]
\begin{scope}[shift={(-0.6, 0)}]
\draw [blue] (0,0) rectangle (0.35, 0.5);
\node at (0.175, 0.25) {$p$};
\end{scope}
\begin{scope}[shift={(0.95, 0)}]
\draw [blue] (0,0) rectangle (0.35, 0.5);
\node at (0.175, 0.25) {$D_0$};
\end{scope}
\draw [blue] (0,0) rectangle (0.7, 0.5);
\node at (0.35, 0.25) {$\widehat{\Phi_p}$};
\draw (0.2, 0.5) .. controls +(0, 0.4) and +(0, 0.4) .. (-0.425, 0.5);
\draw (0.2, 0) .. controls +(0, -0.4) and +(0, -0.4) .. (-0.425, 0);
\draw (0.5, 0.5) .. controls +(0, 0.4) and +(0, 0.4) .. (1.125, 0.5);
\draw (0.5, 0) .. controls +(0, -0.4) and +(0, -0.4) .. (1.125, 0);
\end{tikzpicture}}\geq c^{-1}\raisebox{-0.9cm}{
\begin{tikzpicture}[scale=1.5]
\begin{scope}[shift={(-0.6, 0)}]
\draw [blue] (0,0) rectangle (0.35, 0.5);
\node at (0.175, 0.25) {$p$};
\end{scope}
\begin{scope}[shift={(0.025, 0)}]
\draw [blue] (0,0) rectangle (0.35, 0.5);
\node at (0.175, 0.25) {$D_0$};
\end{scope}
\draw (0.2, 0.5) .. controls +(0, 0.4) and +(0, 0.4) .. (-0.425, 0.5);
\draw (0.2, 0) .. controls +(0, -0.4) and +(0, -0.4) .. (-0.425, 0);
\end{tikzpicture}}
= c^{-1}\raisebox{-0.9cm}{
\begin{tikzpicture}[scale=1.5]
\begin{scope}[shift={(0.025, 0)}]
\draw [blue] (0,0) rectangle (0.35, 0.5);
\node at (0.175, 0.25) {$D_0$};
\end{scope}
\draw (0.2, 0.5) .. controls +(0, 0.4) and +(0, 0.4) .. (-0.425, 0.5)--(-0.425, 0);
\draw (0.2, 0) .. controls +(0, -0.4) and +(0, -0.4) .. (-0.425, 0);
\end{tikzpicture}}.
\end{align*}
Thus $r\geq c^{-1}$. This completes the proof. 
\end{proof}
\begin{remark}
  Theorem \ref{thm:PF to QEC} is also applicable to completely positive  trace-non-increasing map $\Phi_\mathcal{E}$.
  If ${\rm Tr}\circ \Phi_\mathcal{E} \geq s {\rm Tr} $, for some $s>0$, then 
  $$r\geq  s \left\|\sum_{i\in \mathcal{I}} \mathcal{R}(E_ip)\right\|^{-1}. $$
\end{remark}

\subsection{Comparison with Knill-Laflamme Theory} \label{Sec:ContrastKL} 

Here we relate our Theorem \ref{thm:PF to QEC} to the well-known result of Knill and Laflamme (KL).
\begin{theorem}[\bf Knill-Laflamme \cite{KL97}]\label{thm:Knill-Laflamme}
Suppose a noisy environment is a CPTP map $\Phi_\mathcal{E}$ with Kraus operators $\{E_i\}$. For a projection $p \in M_{2^n}(\mathbb{C})$,
there is a CPTP recovery map $\Phi_\mathcal{R}$ such that
$$\Phi_\mathcal{R} \Phi_\mathcal{E} (D)=D, ~\forall D\in pM_{2^n}(\mathbb{C}) p,$$
if and only if
$$pE_i^*E_jp=\lambda_{i.j}p, ~\lambda_{i.j}\in \bC~, \forall i,j.$$
\end{theorem}

In our framework, the KL condition is equivalent to
$$\mathcal{D}=\{pE_i^*E_jp\}' \cap p M_{2^n}(\mathbb{C}) p=pM_{2^n}(\mathbb{C}) p.$$
This is a strong constraint to construct $p$. In this case, the correctable information is $pM_{2^n}(\mathbb{C}) p$.
\smallskip
In Theorem \ref{thm:PF to QEC}, there is no constraint on $p$ and the recovery map always exists.
The correctable information is 
$$\mathcal{D}\subseteq \{pE_i^*E_jp\}' \cap pM_{2^n}(\mathbb{C}) p.$$
We get the correct output in probability $r$. Fortunately, $r^{-1}$ is bounded by the number of Kraus operators of $\Phi_{\mathcal{E}}$, which has at most polynomial growth $O(n^d)$ for an $[[n,k,d]]$ code. 
The probability could be compensated from polynomial samples.
In principle, $p$ induces a good quantum error correction code if $\mathcal{D}$ is large.

The PF theorem suggested us to study quantum error correcting codes (QECC) whose logical qubits are mixed states.
The recovery map in Theorem \ref{thm:PF to QEC} is similar to the KL recovery map in the KL theory. This special choice of recovery map results in a relation between our QECC of mixed states and the KL QECC of pure states as follows. For any minimal projection $q$ in $\mathcal{D}'\cap pM_{2^n}(\mathbb{C})p$, the projection $q$ satisfies the KL condition. The construction of a projection $q$ verifying the KL condition becomes easier through the construction of $p$ in Theorem \ref{thm:PF to QEC}. 
Encoding logical qubits $\mathcal{D}$ by mixed states from the choice of $p$ always led to a KL type encoding of $\mathcal{D}$ by pure states from $q$. 
If we consider a general recovery map, not necessarily of KL type, then it is not clear whether a QECC of mixed states always leads to a QECC of pure states. 

\section*{Acknowlednements}
We thank Jochen Gl\"{u}ck for extensive comments on an earlier version.

 \bibliographystyle{plain}

\end{document}